\documentclass[reqno,11pt]{amsart}
\usepackage{amsmath,amssymb,mathrsfs,amsthm,amsfonts}
\usepackage[inline]{enumitem}
\usepackage[usenames,dvipsnames]{xcolor}
\usepackage{hyperref}
\usepackage[percent]{overpic}
\usepackage{comment}
\usepackage{stmaryrd}
\usepackage{algorithm}
\usepackage{algorithmic,bbm}
\usepackage{pgfplots}
\usepackage{subcaption}
\usepackage{tikz}
\usetikzlibrary{calc}
\usetikzlibrary{arrows.meta,backgrounds}
\usepackage{multirow,array,longtable,booktabs}
\usetikzlibrary{arrows}

\hypersetup{
	colorlinks=true, linkcolor=blue,
	citecolor=ForestGreen
}
\usepackage[paper=letterpaper,margin=1in]{geometry}
\DeclareMathAlphabet{\mathpzc}{OT1}{pzc}{m}{it}

\newtheorem{theorem}{Theorem}[section]
\newtheorem{lemma}[theorem]{Lemma}

\newtheorem{proposition}[theorem]{Proposition}

\theoremstyle{definition}
\newtheorem{definition}[theorem]{Definition}

\newtheorem{remark}[theorem]{Remark}
\numberwithin{equation}{section}
\allowdisplaybreaks
\usepackage{acronym}

\acrodef{KPZ}{Kardar--Parisi--Zhang}
\acrodef{SHE}{Stochastic Heat Equation}
\acrodef{LDP}{Large Deviation Principle}

\renewcommand{\Pr}{\mathbb{P}}
\newcommand{\Ex}{\mathbb{E}}
\newcommand{\E}{\mathbf{E}}	
\renewcommand{\d}{\mathrm{d}}	
\newcommand{\ind}{\mathbf{1}}	

\newcommand{\p}{\partial}
\newcommand{\Pf}{\mathfrak{F}}
\newcommand{\Br}{\mathfrak{B}}
\makeatletter

\def\note#1{\textup{\textsf{\color{blue}(#1)}}}

\newcommand{\norm}[1]{\Vert#1\Vert}

\newcommand{\R}{\mathbb{R}} 
\newcommand{\Z}{\mathbb{Z}} 

\newcommand{\e}{\varepsilon}

\newcommand{\m}{\mathsf}

\renewcommand{\hat}{\widehat}
\newcommand{\til}{\widetilde}

\newcommand{\az}{\textcolor{cyan}}

\usepackage{graphicx}

\title[Fluctuations exponents of the open KPZ equation in the maximal current phase]{Fluctuation exponents of the open KPZ equation in the maximal current phase}

\author[A.\ A.\ C.\ Hip]{Andres A. Contreras Hip}
\address{A.\ A.\ C.\ Hip,
	Department of Mathematics, University of Chicago,
	\newline\hphantom{\quad \ \ S. Das}
	5734 S.~University Avenue, Chicago, Illinois 60637 USA
}
\email{acontreraship@uchicago.edu}

\author[S.\ Das]{Sayan Das}
\address{S.\ Das,
	Department of Mathematics, University of Chicago,
	\newline\hphantom{\quad \ \ S. Das}
	5734 S.~University Avenue, Chicago, Illinois 60637 USA
}
\email{sayan.das@columbia.edu}

\author[A.\ Zitridis]{Antonios Zitridis}
\address{A.\ Zitridis,
	Department of Mathematics, University of Chicago,
	\newline\hphantom{\quad \ \ S. Das}
	5734 S.~University Avenue, Chicago, Illinois 60637 USA
}
\email{zitridisa@uchicago.edu}

\begin{document}
	\begin{abstract}
		We consider the open KPZ equation $\mathcal{H}(x,t)$ on the interval $[0,L]$ with Neumann boundary conditions depending on parameters $u,v\ge 0$ (the so-called maximal current phase). For $L \sim t^{\alpha}$ and stationary initial conditions, we obtain matching upper and lower bounds on the variance of the height function $\mathcal{H}(0,t)$ for $\alpha \in [0,\frac23]$. Our proof combines techniques from \cite{yu2}, which treated the periodic KPZ equation, with Gibbsian line ensemble methods based on the probabilistic structure of the stationary measures developed in \cite{ck,bld22,bkww,zong,zoe}.
	\end{abstract}
	
	
	\maketitle
	{
		\hypersetup{linkcolor=black}
		\setcounter{tocdepth}{1}
		\tableofcontents
	}
	
	\section{Introduction}

   \subsection{The model and main results} The Kardar–Parisi–Zhang (KPZ) equation is a singular stochastic partial differential equation (SPDE) given by
\begin{align}\label{kpz}
\partial_t \mathcal{H} = \tfrac{1}{2} \partial_{xx} \mathcal{H} + \tfrac{1}{2} (\partial_x \mathcal{H})^2 + \xi,
\end{align}
where $\xi$ denotes {1+1 dimensional} space-time white noise. Introduced in \cite{kpz} as a model for stochastic interface growth, the KPZ equation has since become a central object of study in both mathematics and physics due to its deep connections with directed polymers in random environments, last passage percolation, interacting particle systems, and random matrix theory. We refer the reader to the surveys \cite{corwin2012kardar,hht15,qs15,gan21} for comprehensive overviews of the KPZ equation and its universality class.

In this paper, we study the open KPZ equation, where \eqref{kpz} is restricted to a finite interval $x \in [0, L]$ for $L > 0$ and is subject to Neumann boundary conditions
\begin{equation}\label{openkpz}
\partial_x\mathcal{H}(0,t) = u, \quad \partial_x\mathcal{H}(L,t) = -v,
\end{equation}
for parameters $u, v \in \mathbb{R}$. To emphasize the dependence on the domain length, we denote the solution by $\mathcal{H}_L = \mathcal{H}$. The solution admits a natural interpretation as the log-partition function of the continuum directed random polymer (CDRP) in a strip, with attractive or repulsive interactions at the boundaries $x = 0$ and $x = L$.

Due to the nonlinearity and the singular nature of the noise, the equation with these boundary conditions is ill-posed in the classical sense. A standard way to make sense of this setup is via the Cole–Hopf transformation, defining $\mathcal{Z}_L := e^{\mathcal{H}_L}$. Then $\mathcal{Z}_L$ formally satisfies the stochastic heat equation (SHE)
\begin{align}\label{opshe}
\partial_t \mathcal{Z}_L = \tfrac{1}{2} \partial_{xx} \mathcal{Z}_L + \mathcal{Z}_L \xi,
\end{align}
with Robin boundary conditions
\begin{align}\label{robin}
\partial_x \mathcal{Z}_L(0,t) = (u - \tfrac{1}{2}) \mathcal{Z}_L(0,t), \qquad \partial_x \mathcal{Z}_L(L,t) = -(v - \tfrac{1}{2}) \mathcal{Z}_L(L,t).
\end{align}
The inclusion of this 1/2 factor is just a convention that ensures that the point $u=v=0$ is special in terms of the phase diagram for the open KPZ equation (see Figure \ref{fig:enter-labelx}). It is known that this equation admits a mild solution that is unique, almost surely positive (for a certain class of initial data), and adapted to the natural filtration of the noise; see \cite{cs18,parekh2019kpz}. The solution to the open KPZ equation is then defined via $\mathcal{H}_L := \log \mathcal{Z}_L$.

\medskip

A fundamental question in the study of the open KPZ equation is to understand the nature of its stationary state. We say that a probability measure $\Lambda(\cdot)$ on $C[0, L]$ is stationary if the solution $\mathcal{H}_L$ to \eqref{kpz} with boundary conditions \eqref{openkpz} and initial condition $\Lambda$ satisfies the invariance relation
\begin{align*}
 {\mathcal{H}_L(\cdot,t)- \mathcal{H}_L(0,t)}\stackrel{d}{=}\Lambda(\cdot) \mbox{ for all }t\ge 0.
 \end{align*}
Recently, there has been significant progress in characterizing these stationary measures \cite{ck,bld01,bld22,bkww,zong,zoe} (see also the review \cite{cor22}). It is now known that for every $u, v \in \mathbb{R}$, there exists a unique stationary measure $\Lambda = \Lambda_{u,v}$, given in terms of a suitably reweighted Brownian motion (see Section \ref{sec:2.1} for precise description).

\smallskip

In this work, we focus on the fluctuation of the solution itself. Our first result provides an asymptotic formula for the variance.

\begin{theorem}\label{varde} Fix any $t,L>0$. Assume $u,v\ge 0$. Let $\mathcal{H}_L(0,t)$ be the stationary solution of the open KPZ equation \eqref{kpz} on $[0,L]$ with boundary data \eqref{openkpz}. There exists an absolute constant $C>0$ such that
    \begin{align*}
        \left|\sqrt{\operatorname{Var}(\mathcal{H}_L(0,t))}-\sqrt{t}\sigma_L\right| \le C\sqrt{L}.
    \end{align*} 
    Here, $\sigma_L^2$ is the asymptotic variance which has the following functional formula:
    \begin{align}\label{def:sigma}
        \sigma_L^2 = \Ex\left[\frac{\int\limits_0^L e^{\Lambda_1(x)+2\Lambda_2(x)+\Lambda_3(x)}dx}{\int\limits_0^L e^{\Lambda_1(x)+\Lambda_2(x)}dx\int\limits_0^L e^{\Lambda_2(x)+\Lambda_3(x)}dx}\right]
    \end{align}
    where $\Lambda_1,\Lambda_2,\Lambda_3$ are three independent copies of the stationary solution $\Lambda$.
\end{theorem}

We remark that the recent work \cite{barraquand2025large} expresses cumulants of the KPZ equation in
terms of a functional equation involving an integral operator. It would be interesting to match the formulas in \cite{barraquand2025large} with ours described above in \eqref{def:sigma}.

\medskip

As for the exact nature of the fluctuations, when the domain length $L$ is fixed and $t\to \infty$, due to known ergodicity results (see, e.g., \cite{parekh2022ergodicity,km22}), we expect Gaussian fluctuations for the height function $\mathcal{H}_L(0,t)$. A more intriguing regime arises when both $L$ and $t$ go to infinity \textit{jointly}. For concreteness, we take $L=\lambda t^{\alpha}$ for $\lambda,\alpha>0$. Naturally, the fluctuation behavior in this regime depends sensitively on the boundary parameters $u,v$ as they influence the large-scale geometry of the polymer paths and the structure of the stationary measure (Figure \ref{fig:enter-labelx}).

When $u < 0,v > u$ (the \textit{low-density phase}), or when $v < 0,u>v$ (the \textit{high-density phase}), it is predicted that the polymer paths are localized near $x=0$ or $x=L$, respectively. In the symmetric case $u = v < 0$, localization occurs near either boundary with probability $1/2$.  Because of this predicted localization near the boundaries, one expects the height function to exhibit Gaussian fluctuations of order $O(\sqrt{t})$ regardless of the value of $\alpha$ in these phases.

\begin{figure}[h!]
    \centering
    \includegraphics[width=0.5\linewidth]{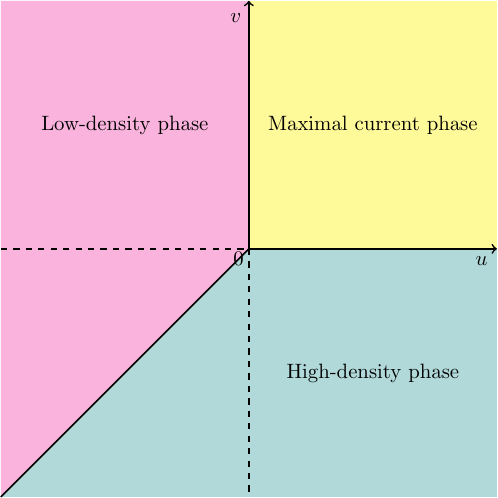}
    \caption{Phase diagram of the open KPZ equation.}
    \label{fig:enter-labelx}
\end{figure}

In contrast, when $u, v \ge 0$ (the so-called \textit{maximal current phase}), the polymer paths are expected to be delocalized in the bulk. In this phase, the fluctuation exponents are predicted to depend nontrivially on the scaling parameter $\alpha$ in this phase. The model is expected to interpolate between the Gaussian class $(\alpha=0)$ and the KPZ class $(\alpha=\infty)$, where one recovers the familiar KPZ scaling exponents: $1/3$ for height fluctuations and $2/3$ for transversal fluctuations -- the exponents observed in full-space and supercritical half-space KPZ settings \cite{wu1,ds25}.

Our second result stated below captures the precise fluctuation exponent for $\mathcal{H}_{\lambda t^{\alpha}}(0,t)$ when $\alpha\in (0,2/3]$ in the maximal current phase. We achieve this by computing exact decay rate of the asymptotic variance $\sigma_L^2$ defined in \eqref{def:sigma}.

	\begin{theorem} \label{t.main1} Fix any $t\ge 1$, $\lambda>0$, and $\alpha\in (0,2/3)$. For each $L\ge1$, let $\mathcal{H}_L$ be the solution to the open KPZ equation \eqref{kpz} on $[0,L]$ subjected to boundary conditions \eqref{openkpz} started from stationary initial data. Assume $u,v\ge 0$. There exists constant $C(u,v)>1$ such that for all $L\ge 1$ $$C^{-1}L^{-1/2} \le \sigma_L^2 \le CL^{-1/2}$$ where $\sigma_L^2$ is the asymptotic variance  defined in \eqref{def:sigma}. Consequently, by Theorem \ref{varde},
    there exist constants $C_1(u,v,\lambda,\alpha)$,  $C_2(u,v,\lambda,\alpha)>1$  such that 
    \begin{align*}
      C_1^{-1} \cdot  t^{1-\alpha/2}  \le & \operatorname{Var}(\mathcal{H}_{\lambda t^{\alpha}}(0,t)) \le C_1\cdot  t^{1-\alpha/2}, \\
      & \operatorname{Var}(\mathcal{H}_{\lambda t^{2/3}}(0,t)) \le C_2\cdot  t^{2/3},
    \end{align*}
    and there exists $\delta>0$ and $C_3(u,v,\delta)>0$ such that for all $\lambda\in (0,\delta)$
\begin{align*}
    C_3^{-1}\cdot  t^{2/3} \le \operatorname{Var}(\mathcal{H}_{\lambda t^{2/3}}(0,t)). \phantom{\le C\cdot  t^{2/3}.}
\end{align*}
	\end{theorem}

As the transversal fluctuations of the KPZ equation in full space are of order $O(t^{2/3})$, the scale $L = O(t^{2/3})$ represents a critical regime in which the height function exhibits nontrivial spatial correlations across the domain. This is often referred to as the relaxation regime. The cases $\alpha < 2/3$ and $\alpha > 2/3$ are known as the super-relaxation and sub-relaxation regimes, respectively. Theorem \ref{t.main1} captures the fluctuation exponent in the super-relaxation regime and part of the relaxation regime. We expect that $\operatorname{Var}(\mathcal{H}_L(0,t))$ should be of order $t^{2/3}$ when $L = \lambda t^{2/3}$ for large $\lambda$, and more generally when $L \gg t^{2/3}$; however, our current methods are not sufficient to establish this rigorously.

\subsection{Context} Our main theorems contribute to a growing line of work aimed at understanding phase transition behavior in KPZ-type models with boundaries or finite-size effects, where the nature of fluctuations changes as the system size or boundary parameters are varied. This has been studied extensively in the periodic and half-space settings, which we now briefly review.

\subsubsection*{Periodic setting}
The long-time behavior of finite-size KPZ systems, where the system size also grows with time, was first systematically studied in the mathematics literature in the context of the totally asymmetric simple exclusion process (TASEP) on the torus. In this setting, \cite{derrida1993exact} was the first to derive that the asymptotic variance decays as $L^{-1/2}$ and explained how this decay relates to the $2/3$ KPZ exponent. In the setting of the periodic KPZ equation, the work of Dunlap, Gu, and Komorowski \cite{yu2} established analogous variance bounds to those in Theorem~\ref{t.main1}, covering both the super-relaxation regime and part of the sub-relaxation regime. Their techniques serve as an important point of comparison for the present work (see Section~\ref{sec:1.3} for details).

Subsequently, by exploiting the integrable structure of the model, a series of works \cite{tasep0,tasep1,tasep2,tasep3,tasep4,tasep5,tasep6} obtained exact formulas, fluctuation exponents, and limiting distributions for TASEP on the torus. In particular, it has been shown that the rescaled height fluctuations in the relaxation regime interpolate between Gaussian and Tracy–Widom statistics. For first-passage percolation and longest increasing subsequences, precise fluctuation exponents and Gaussian fluctuations have been rigorously established in the super-relaxation regime \cite{chatterjee2013central,dey2018longest}.

\subsubsection*{Half-space setting}
In the setting where $L=\infty$ and $v=0$, one obtains the half-space KPZ equation. A phase transition from Gaussian statistics to Tracy–Widom-type statistics is observed in this model by setting $u=ct^{-1/3}$
  and varying the constant $c$. This transition was first identified in a series of works by Baik and Rains \cite{br1,br01,br3} in the context of last passage percolation. Multi-point fluctuations were subsequently analyzed in \cite{sis}, and similar results were later proven for exponential LPP in \cite{bbcs0, bbcs} using the framework of Pfaffian Schur processes. For the half-space KPZ equation and its integrable discretization\textemdash the half-space log-gamma polymer\textemdash analogous phase transitions and fluctuation exponents have been studied in mathematics works \cite{bbc20,bw22,ims22,bcd,dz23,halfairy,vic,ds25} and in physics works 
  \cite{gld,bbc,ito,de,kr,bkld2,bld1,bkld}.  In a recent work \cite{tao}, the authors study the stationary half-space KPZ equation with boundary parameter $u=ct^{-\beta}$ for $\beta \in [0,1/3]$ and $c<0$, and establish variance bounds for the height function in a spirit similar to ours. We note, however, that their proof relies on entirely different techniques. In addition to the above models, the Baik--Rains phase transition has recently been rigorously proven for half-space ASEP and six-vertex models in the works \cite{he1, he2}.

\medskip

Finally, we note that all of the aforementioned results pertain to the spatial dimension $d=1$. For $d\ge 2$, we refer to \cite{CSZ17, DG22, GHL23a, GHL23b, KN24, Tao24, spa, Ch19, ck19, gy23}, where various important properties of the SHE and/or the KPZ equation such as fluctuations, spatial ergodicity, and intermittency have been investigated.


\subsection{Proof ideas} \label{sec:1.3}

In this section, we sketch the key ideas behind the proofs of our main results. As mentioned earlier, our approach is inspired by the proof techniques developed in \cite{yu2}. We begin by outlining the main ideas from their work and then highlight the key similarities and differences in our setting.

\subsubsection{Key ideas in \cite{yu2}} In \cite{yu2}, the authors consider the KPZ equation in the periodic setting and prove analogous variance results. The stationary measure in this case is given by $\mathfrak{B}$, a Brownian bridge on $[0,L]$ from $0$ to $0$. Using the Clark–Ocone formula, they derive bounds on the variance for each fixed $t$ and $L$, and obtain the following expression for the asymptotic variance $\sigma_{L,\mathrm{per}}^2$:
\begin{align*}
    \sigma_{L,\mathrm{per}}^2 = \Ex\left[\frac{\int\limits_0^L e^{B_1(x)+2B_2(x)+B_3(x)}dx}{\int\limits_0^L e^{B_1(x)+B_2(x)}dx\int\limits_0^L e^{B_2(x)+B_3(x)}dx}\right]
\end{align*}
where $B_1, B_2, B_3$ are independent copies of $\mathfrak{B}$. The core idea in their analysis is to show that $\sigma_{L,\mathrm{per}}^2 \sim L^{-1/2}$ by analyzing this expression. This analysis is nontrivial and relies crucially on the structure of the Brownian bridge. Exploiting its properties, they reformulate the variance as
\begin{align}\label{ired}
    \sigma_{L,\mathrm{per}}^2 =L\cdot \Ex\left[\frac{1}{\int\limits_0^L e^{B_1(x)+B_2(x)}dx\int\limits_0^L e^{B_2(x)+B_3(x)}dx}\right]= L\cdot \Ex\left[\prod_{k=1}^2 \left(\int_0^L e^{v_k\cdot V(x)}dx\right)^{-1}\right]
\end{align}
where $V$ is a $2$-dimensional Brownian bridge and $v_k=-2^{-1/2}(\sqrt{3},(-1)^k)$, $k=1,2$. The first reduction above is via circular translation invariance of Brownian bridges and the second reduction is by writing correlated Brownian bridges $(B_1+B_2,B_2+B_3)$ in terms of a $2D$ Brownian bridge from the origin to origin. The key insight here is that the dominant contribution to the above expectation comes from paths for which $$\omega(V):=\sup_{x\in [0,L]} \max_{k=1,2} v_k\cdot V(x)$$ is bounded above. Their proof then proceeds by estimating the probability that $\omega(V)\le a$ which involves understanding Brownian bridges killed upon exiting the wedge $\omega^{-1}((-\infty,a])$. This is made possible by explicit transition density formulas for such killed bridges in terms of Bessel functions, which play a central role in their technical arguments. 

\begin{figure}[h!]
    \centering
   \includegraphics[width=0.5\linewidth]{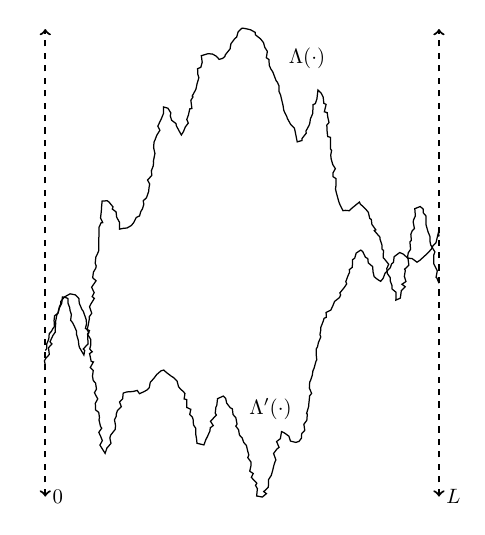}
    \caption{A typical sample from $\Pr_{u,v}^L$ law for $u,v>0$. The term $\int_0^L e^{-(\Lambda(s)-\Lambda'(s))} ds$ in the exponent imposes a strong penalty whenever $\Lambda < \Lambda'$, effectively pushing $\Lambda$ above $\Lambda'$. On the other hand, the boundary terms $-u(\Lambda(0)-\Lambda'(0))$ and $-v(\Lambda(L)-\Lambda'(L))$ encourage the two processes to be close at $x = 0$ and $x = L$. }
    \label{puv}
\end{figure}

\subsubsection{Our proof} 
We now turn to the proof idea of our theorems. The proof of Theorem \ref{varde} is also an application of the Clark–Ocone formula and largely follows the same structure as in \cite{yu2}, with necessary modifications to account for the presence of boundary conditions. We therefore focus on explaining the proof strategy for Theorem \ref{t.main1}. While our goal is likewise to show that $\sigma_L^2$, as defined in \eqref{def:sigma}, exhibits the same $L^{-1/2}$ decay rate, neither a reduction analogous to \eqref{ired} nor explicit formulas -- such as those involving Bessel functions used in \cite{yu2} -- are available in our setting. This is due to the more intricate structure of the stationary measure $\Lambda$, which we now explain.

\subsubsection*{Stationary measure description} 
Let $\Pr_{\mathrm{free}}$ denote the law of a pair $(\Lambda, \Lambda')$ of independent Brownian motions, with $\Lambda$ started from $0$ and $\Lambda'$ started from a point distributed according to Lebesgue measure on $\mathbb{R}$. This defines an infinite measure. We define a new probability measure $\Pr_{u,v}^L$ on the pair $(\Lambda, \Lambda') : [0,L] \to \mathbb{R}$ via the Radon–Nikodym derivative
\begin{align*}
\frac{d\Pr_{u,v}^L}{d\Pr_{\mathrm{free}}}(\Lambda,\Lambda') \propto \exp\left(-u(\Lambda(0)-\Lambda'(0)) - v(\Lambda(L)-\Lambda'(L)) - \int_0^L e^{-(\Lambda(s)-\Lambda'(s))}\,ds\right).
\end{align*}
A recent sequence of works \cite{ck,bld22, bkww, zong, zoe} has shown that the law of $\Lambda(x)$ under this measure yields the stationary measure for the open KPZ equation \eqref{openkpz} on $[0,L]$ when $u + v > 0$. Figure \ref{puv} describes how a sample from the measure $\Pr_{u,v}^L$ typically looks like (for $u,v>0$).

\begin{figure}[h!]
    \centering
     \begin{subfigure}[b]{0.45\textwidth}
			\centering \includegraphics[width=\linewidth]{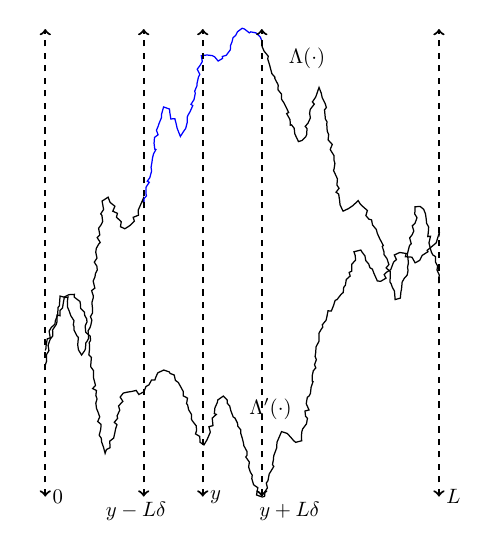}
            \caption{}
            \end{subfigure}
            \begin{subfigure}[b]{0.45\textwidth}
			\centering \includegraphics[width=\linewidth]{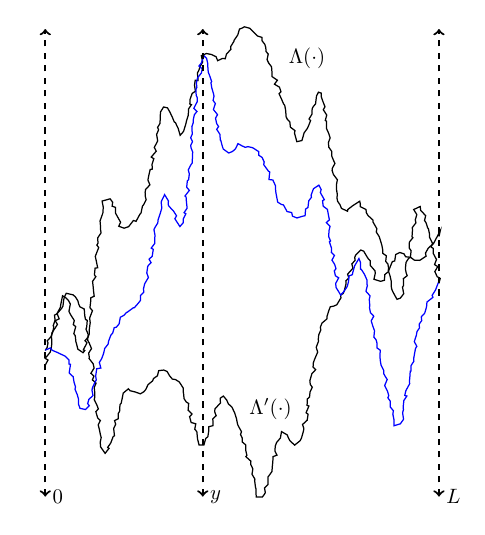}
            \caption{}
            \end{subfigure}
    \caption{(A) $\Lambda$ on the region $[y-L\delta,y+L\delta]$ is so far from $\Lambda'$ (with high probability) that it roughly looks like a Brownian bridge. (B) The blue curve denotes a Brownian bridge joining $(0,\Lambda(0)), (y,\Lambda(y)),$ and $(L,\Lambda(L))$. By stochastic monotonicity, $\Lambda$ is stochastically larger than the blue curve.}
    \label{ide}
\end{figure}

\subsubsection*{Variance decay rate} 
Returning to the expression for $\sigma_L^2$ in \eqref{def:sigma}, we note that unlike in the periodic setting, there is no circular translation invariance for the stationary measure. To proceed, we rewrite:
\begin{align}\label{sigmaLexp}
\sigma_L^2 = \Ex\left[\frac{\int_0^L e^{\Lambda_1(x)+2\Lambda_2(x)+\Lambda_3(x)}\,dx}{\int_0^L e^{\Lambda_1(x)+\Lambda_2(x)}\,dx \cdot \int_0^L e^{\Lambda_2(x)+\Lambda_3(x)}\,dx}\right] = \sum_{y=0}^{L-1} \Pf(y),
\end{align}
where
\begin{align}\label{def:pfy}
\Pf(y) := \Ex\left[\frac{\int_y^{y+1} e^{\Lambda_1(x)+2\Lambda_2(x)+\Lambda_3(x) - \Lambda_1(y) - 2\Lambda_2(y) - \Lambda_3(y)}\,dx}{\int_0^L e^{\Lambda_1(x)+\Lambda_2(x) - \Lambda_1(y) - \Lambda_2(y)}\,dx \cdot \int_0^L e^{\Lambda_2(x)+\Lambda_3(x) - \Lambda_1(y) - \Lambda_2(y)}\,dx}\right].
\end{align}
To get a sense of how $\Pf(y)$ behaves, it is helpful to study the decay of the following simpler quantity:
\begin{align}\label{def:ppfy}
\hat\Pf(y) := \Ex\left[\frac{1}{\int_0^L e^{\Lambda_1(x)+\Lambda_2(x) - \Lambda_1(y) - \Lambda_2(y)}\,dx \cdot \int_0^L e^{\Lambda_2(x)+\Lambda_3(x) - \Lambda_1(y) - \Lambda_2(y)}\,dx}\right].
\end{align}
At an intuitive level, $\Pf(y)$ and $\hat\Pf(y)$ should have the same decay rate. This is because the numerator in \eqref{def:pfy} is an integral over a unit interval and hence is expected to be of constant order $O(1)$. So to understand the behavior of $\sigma_L^2$, it suffices to focus on estimating the typical size of $\hat\Pf(y)$. We do this in two steps.

\medskip

\textbf{Step 1. Reduction to Brownian bridges with random endpoints.} 
To study the decay rate of \eqref{def:ppfy}, we reduce the analysis to computations involving Brownian bridges with random endpoints. The key idea is that when \( y \) lies in the bulk (say, in the interval \( [L/4, 3L/4] \)), there is sufficient separation between the three curves \( \Lambda_i \) and \( \Lambda_i' \) for \( i=1,2,3 \). As a result, the processes \( \Lambda_i \) near \( y \) should behave approximately like Brownian motions.

To make this precise, we show that conditional on the sigma-field \( \sigma(\Lambda_i(y \pm L\delta), \Lambda_i(y) \mid i=1,2,3) \), for sufficiently small \( \delta > 0 \), the Radon–Nikodym derivatives in the interval \( [y - L\delta, y + L\delta] \) are bounded from below. This means that for the purpose of obtaining a lower bound, we may effectively ignore the RN derivatives and instead work with Brownian bridges having random endpoints \( \Lambda_i(y \pm L\delta) \) and \( \Lambda_i(y) \), see Figure \ref{ide}(A).

For the upper bound, we use a stochastic monotonicity argument, which says that the process \( (\Lambda_i(x) - \Lambda_i(y)) \), conditional on \( \sigma(\Lambda_i(0), \Lambda_i(y), \Lambda_i(L)) \), is stochastically larger than a Brownian bridge connecting \( \Lambda_i(0) \), \( \Lambda_i(y) \), and \( \Lambda_i(L) \), see Figure \ref{ide}(B). This provides a comparison framework that allows us to control the fluctuations from above.

\medskip

\textbf{Step 2. Estimating functionals under Brownian bridges with varying endpoints.} 
Once reduced to the Brownian bridge setting, we apply a linear transformation (as in \cite{yu2}) to express the functionals as those of 2D Brownian bridge from the origin to a random endpoint. The key difference in our case is the presence of the random endpoint, which requires extending the estimates in \cite{yu2} to ones that depend on the endpoints. This extension is relatively straightforward for the lower bound, but significantly more challenging for the upper bound. To keep the exposition light, we defer a detailed discussion of the lower bound to Section \ref{sec:5.1}, noting here that the main difficulty arises from the fact that the argument in \cite{yu2} relies on stopping times for 2D Brownian bridges from the origin to the origin, for which precise tail estimates are available via a conformal transformation -- a tool not available in our setting.

\subsubsection*{Other technical aspects} 
We emphasize that our actual proof works with $\Pf(y)$, not $\hat\Pf(y)$. While the idea of estimating $\Pf(y)$ is morally similar to that for $\hat\Pf(y)$, there are several technical differences that must be addressed. We chose to present the simplified version involving $\hat\Pf(y)$ above, as it captures the core intuition while avoiding the more delicate technicalities required for $\Pf(y)$. In particular, to show that the numerator in $\Pf(y)$ is of order $O(1)$, and to control the random endpoints appearing in the steps above, one needs exponential moment estimates for the stationary measures. A significant portion of our work is devoted to understanding the diffusive scaling limits of the two layer Gibbs representation of the stationary measures and establishing various exponential moments for the underlying process. This is done in Section \ref{sec:2} of the paper.

\subsection{Extensions to other models and phases} 

Our proof strategy offers a general framework for establishing variance bounds in open KPZ-type models that admit a similar two-layer representation of the stationary measure -- such as geometric last passage percolation and the log-gamma polymer on a strip \cite{zong}.

In this work, however, we do not address the low and high density phases. In these regimes, we expect that $\sigma_L^2$ does not decay with $L$. Let us briefly outline the main challenges involved in analyzing this setting.

First, while the two-layer description of the stationary measure is conjectured to hold for all $u,v \in \mathbb{R}$ (see Theorem 1.8 in \cite{zong}), a rigorous proof is currently available only when $u,v \in \mathbb{R}$ with $u+v > 0$. Suppose we restrict attention to the low density phase ($u<0$, $v>u$) under the assumption that $u+v>0$ so that the existing theory applies. In this case, it is known that $\Lambda(x)$ converges to $\mathfrak{B}(x) + ux$, where $\mathfrak{B}$ is a Brownian motion \cite{hy}. Heuristically, then, we should have the following weak convergence:
\begin{align*}
\frac{\int\limits_0^L e^{\Lambda_1(x)+2\Lambda_2(x)+\Lambda_3(x)}dx}{\int\limits_0^L e^{\Lambda_1(x)+\Lambda_2(x)}dx\int\limits_0^L e^{\Lambda_2(x)+\Lambda_3(x)}dx} \stackrel{d}{\to} \frac{\int\limits_0^\infty e^{\Br_1(x)+2\Br_2(x)+\Br_3(x)+4ux}\,dx}{ \int\limits_0^\infty e^{\Br_1(x)+\Br_2(x)+2ux}\,dx \int\limits_0^\infty e^{\Br_2(x)+\Br_3(x)+2ux}\,dx },
\end{align*}
where $\Br_1, \Br_2, \Br_3$ are independent standard Brownian motions. Since $\Br_i(x) = O(\sqrt{x})$ and $u<0$, the above improper integrals converge almost surely, and the entire ratio on the right-hand side is almost surely finite. However, it appears nontrivial to prove that this limiting random variable lies in $L^1$, particularly for small values of $|u|$. Extending the above convergence to $L^1$ would require a delicate analysis of the two-layer Gibbs structure in this regime -- one that is qualitatively different from the maximal current phase studied in this paper. A rigorous treatment of these low and high density phases remains an interesting direction for future work.


    \subsection*{Outline} The rest of the paper is organized as follows. In Section \ref{sec:2}, we investigate the diffusive properties of the stationary measure. In Section \ref{sec:3}, we prove our main theorems, assuming upper and lower bounds on the asymptotic variance. The matching upper and lower bounds on the variance are established in Sections \ref{sec:4} and \ref{sec:5}, respectively. Appendix \ref{appC} collects moment bounds for the open SHE.

    \subsection*{Acknowledgements} We thank Guillaume Barraquand, Ivan Corwin, and Yu Gu for helpful discussions, as well as for their encouragement and feedback on an earlier draft of the paper. SD thanks Shalin Parekh and Christian Serio  for useful discussions.

	\section{Diffusive properties of stationary measures} \label{sec:2}

\subsection{Gibbsian line ensembles and stationary measures} \label{sec:2.1}
    In this section, we state the description of the stationary measure and relate it with Gibbs measures. 

\cite{ck} proved the existence of a stationary measure for the open KPZ equation on $[0,1]$ for all $u, v \in \R$ and characterized it via a multi-point Laplace transform under the condition $u + v > 0$. Later, \cite{bld22,bkww} provided a probabilistic description of this measure by inverting the multi-point Laplace transform, valid for $L=1$, $u+v>0$. Most recently, \cite{zoe} extended the validity of this description to all $L>0$ and $u+v>0$. We state this description below.

\medskip
    
    Let $U$ be a Brownian motion started from $0$ with diffusion coefficient $2$. Let $V: [0,L]\to \R$ be independent from $U$. The law of $V$ is
absolutely continuous with respect to the free Brownian (infinite) measure with Lebesgue measure for $V(0)$
and Brownian law from there with diffusion coefficient $2$, with Radon-Nikodym derivative proportional to
\begin{align*}
    \exp\left(-uV(0)-vV(L)-\int_0^L e^{-V(s)}ds\right).
\end{align*}
\begin{theorem}[Theorem 1.4 in \cite{zoe}]\label{thminv} Assume $u+v>0$. With the above definitions,
  $\frac12(U(x)+V(x)-V(0))$ is the stationary measure for the open KPZ equation \eqref{openkpz} on $[0,L]$.
\end{theorem}    

Uniqueness of the above stationary measure follows from the work of \cite{km22,parekh2022ergodicity}.
We now give an  alternative description of the stationary measure.

\begin{definition}[Open-KPZ Gibbs measure] \label{openkpzgibbs} Assume $u+v>0$. Let $\Pr_{\m{free}}$ denotes the law of two independent Brownian motions started from $0$ and a value that is distributed according to Lebesgue measure on $\R$ respectively. This is an infinite measure. Suppose $u,v \ge 0$ and $L\ge 1$. We define the Open-KPZ Gibbs measure to be the law $\Pr_{u,v}^L$ on $\Lambda,\Lambda' : [0,L]\to \R$ given in terms of the following Radon-Nikodym derivative
	\begin{align}\label{Wuv}
    \frac{d\Pr_{u,v}^L}{d\Pr_{\m{free}}}(\Lambda,\Lambda') \propto W_{u,v}:=\exp\left(-u(\Lambda(0)-\Lambda'(0))-v(\Lambda(L)-\Lambda'(L))-\int_0^L e^{-(\Lambda(s)-\Lambda'(s))}ds\right).  
\end{align}
\end{definition}
\vspace{2mm}

Using properties of Brownian motions (see the discussions in Section 1.6.1 of \cite{zong}) and a change of variables, one can check that the law of $\Lambda(\cdot)$ under $\Pr_{u,v}^L$ (defined in Definition \ref{openkpzgibbs}) is the stationary measure for the open KPZ equation \eqref{openkpz} in $[0,L]$. We remark that a slightly different Gibbsian formulation is given in \cite{zong}, which is valid for all $u,v\in \R$, modulo a polymer convergence result (see Theorem 1.8 therein). For the purposes of the present work, however, it is more advantageous to use the formulation in Definition \ref{openkpzgibbs},since it can be directly related to the $H$-Brownian Gibbs measure studied in \cite{kpzle}.

\begin{definition}[$H$-Brownian Gibbs measure] \label{hbrm} Fix $k \ge 1$, $\vec{a},\vec{b}\in \R^k$ and $f,g: [a,b]\to \R\cup \{\pm\infty\}$. Let $H : \R\to [0,\infty)$ be a convex continuous function.
  Let $\Pr_{[t_1,t_2]}^{k;\vec{a};\vec{b}}$ denotes the law of $k$ independent Brownian bridges $(B_i)_{i=1}^k$ on $[t_1,t_2]$ from $\vec{a}$ to $\vec{b}$. Define a measure $\Pr_{W;[t_1,t_2]}^{k;\vec{a};\vec{b};f,g}$ absolutely continuous with respect to $\Pr_{[t_1,t_2]}^{k;\vec{a};\vec{b};f,g}$ with the following Radon-Nikodym derivative:
    \begin{align*}
        \frac{d\Pr_{H;[t_1,t_2]}^{k;\vec{a};\vec{b};f,g}}{d\Pr_{[t_1,t_2]}^{k;\vec{a};\vec{b}}}(B_1,B_2,\ldots, B_k) =W_H:=\prod_{i=0}^k \exp\bigg(-\int_{t_1}^{t_2} H(B_{i+1}(s)-B_{i}(s)) ds\bigg)
    \end{align*}
    where $B_0:=f$ and $B_{k+1}:=g$. We call $\Pr_{H;[t_1,t_2]}^{k;\vec{a};\vec{b};f,g}$ as the $H$-Brownian Gibbs measure.
\end{definition}

    The above measure was first introduced in \cite{kpzle}. They showed that the full-space KPZ equation with narrow wedge initial data can be viewed as the top curve of a line ensemble whose conditional distribution restricted to finite region is a $H$-Brownian Gibbs measure with $H(x)=e^{x}$. Using this line ensemble structure, various properties of the full-space KPZ equation are proved in the literature. We refer to \cite{cgh,dg21,dz22,dz22b,wu1,wu2,ds25} for references.   and it arises naturally in the context of full-space and half-space KPZ line ensembles \cite{kpzle,ds25}.

    In our context, the above measures are related to stationary measures as follows. Suppose $(\Lambda, \Lambda') \sim \Pr_{u,v}^L$ according to Definition \ref{openkpzgibbs}. Take any $0\le t_1<t_2 \le L$. From Gibbsian description, it is clear that conditional distributions can be expressed as $H$-Brownian Gibbs measures defined in Definition \ref{hbrm} with $H(x)=e^x$. We have
    \begin{equation}
        \label{gprop}
        \begin{aligned}
  &  \mathrm{Law}\bigg((\Lambda, \Lambda')|_{[t_1,t_2]} \mbox{ conditioned on } (\Lambda, \Lambda')|_{(t_1,t_2)^c} \bigg) = \Pr_{W; [t_1,t_2]}^{2;(a,a');(b,b'),+\infty,-\infty},
\\ &
    \mathrm{Law}\bigg(\Lambda|_{[t_1,t_2]} \mbox{ conditioned on } \Lambda|_{(t_1,t_2)^c}, \Lambda' \bigg) = \Pr_{W; [t_1,t_2]}^{1;a;b,+\infty,\Lambda'},
\\ &
    \mathrm{Law}\bigg(\Lambda'|_{[t_1,t_2]} \mbox{ conditioned on } \Lambda, \Lambda'|_{(t_1,t_2)^c} \bigg) = \Pr_{W; [t_1,t_2]}^{1;a';b'; \Lambda,-\infty},
\end{aligned}
    \end{equation}
where $a=\Lambda(t_1)$, $a'=\Lambda'(t_1)$, $b=\Lambda(t_2)$, and $b'=\Lambda'(t_2)$. $H$-Brownian Gibbs measure enjoys certain stochastic monotonicity, which we record below.

\begin{lemma}[Lemma 2.6 in \cite{kpzle}] Fix $k \ge 1$ and $t_1<t_2$. Suppose $\vec{a}^{(1)}, \vec{a}^{(2)}, \vec{b}^{(1)}, \vec{b}^{(2)} \in \R^k$ such that $\vec{a}_k^{(1)} \ge \vec{a}_k^{(2)}$ and $\vec{b}_k^{(1)} \ge \vec{b}_{k}^{(2)}$. Suppose $f^{(1)},f^{(2)},g^{(1)},g^{(2)} : [t_1,t_2]\to \R$ such that $f^{(1)} \ge f^{(2)}$ and $g^{(1)} \ge g^{(2)}$. There exists a coupling of $\vec{\mathcal{Q}}^{(1)} \sim \Pr_{W;[t_1,t_2]}^{k;\vec{a}^{(1)};\vec{b}^{(1)};f^{(1)};g^{(1)}}$ and $\vec{\mathcal{Q}}^{(2)} \sim \Pr_{W;[t_1,t_2]}^{k;\vec{a}^{(2)};\vec{b}^{(2)};f^{(2)};g^{(2)}}$ such that  almost surely $\mathcal{Q}_j^{(1)}(s) \ge \mathcal{Q}_j^{(2)}(s)$ for all $j\in \llbracket 1, k \rrbracket$, and $s\in [t_1,t_2]$. \end{lemma}

\subsection{Diffusive limits of the Open-KPZ Gibbs measure}

In this section, we study diffusive limits of  $\Pr_{u,v}^L$ measure defined in Definition \ref{openkpzgibbs}. Throughout this subsection we assume $(\Lambda,\Lambda')$ is distributed as $\Pr_{u,v}^L$ and consider the following processes on $[0,1]$:
\begin{align}
   \nonumber & B_L(x):=\tfrac1{\sqrt{L}} \Lambda(xL), \quad B_L'(x):=\tfrac1{\sqrt{L}} \Lambda'(xL), \\ \label{ulvl}
   & U_L(x):=B_L(x)+B_L'(x), \quad V_L(x):=B_L(x)-B_L'(x).
\end{align}
Due to the diffusive scaling we have
	\begin{align}
		 \Pr_{u,v}^L((B_L, B_L') \in A):=\frac{\Ex_{\m{free}}[W_{u,v}^L\ind_{(B_L,B_L')\in A}]}{\Ex_{\m{free}}[W_{u,v}^L]}. \label{3.2a}
	\end{align}
    where \begin{align*}
		W_{u,v}^L=\exp\left(-u\sqrt{L}V_L(0)-v\sqrt{L}V_L(1)-L\int_0^1 e^{-\sqrt{L}V_L(s)}ds\right)  
	\end{align*}
is same as $W_{u,v}$ defined in \eqref{Wuv}. The main result of this subsection is the following.

\begin{theorem}[Diffusive limits] \label{l.difflimits}  We have the following.
		\begin{enumerate}[label=(\alph*),leftmargin=18pt]
			\item If $u,v>0$, $B_L, B_L'$ converge to non-intersecting Brownian motions that start from zero and end at the same point.
			\item If $u>0, v=0$, $B_L, B_L'$  converges to non-intersecting Brownian motions that start from zero.
			\item If $u=0, v>0$, $B_L, B_L'$  converge to $B,B'$ where $B(x)$ and $B'(x)-B(1)$ are independent standard Brownian motions that are conditioned to be $B(x) \ge B'(x)$ on $[0,1]$.
		\end{enumerate}
        Furthermore, we have pointwise $L^2$ convergence.
	\end{theorem}

    \begin{figure}[h!]
        \centering
        \includegraphics[width=14cm]{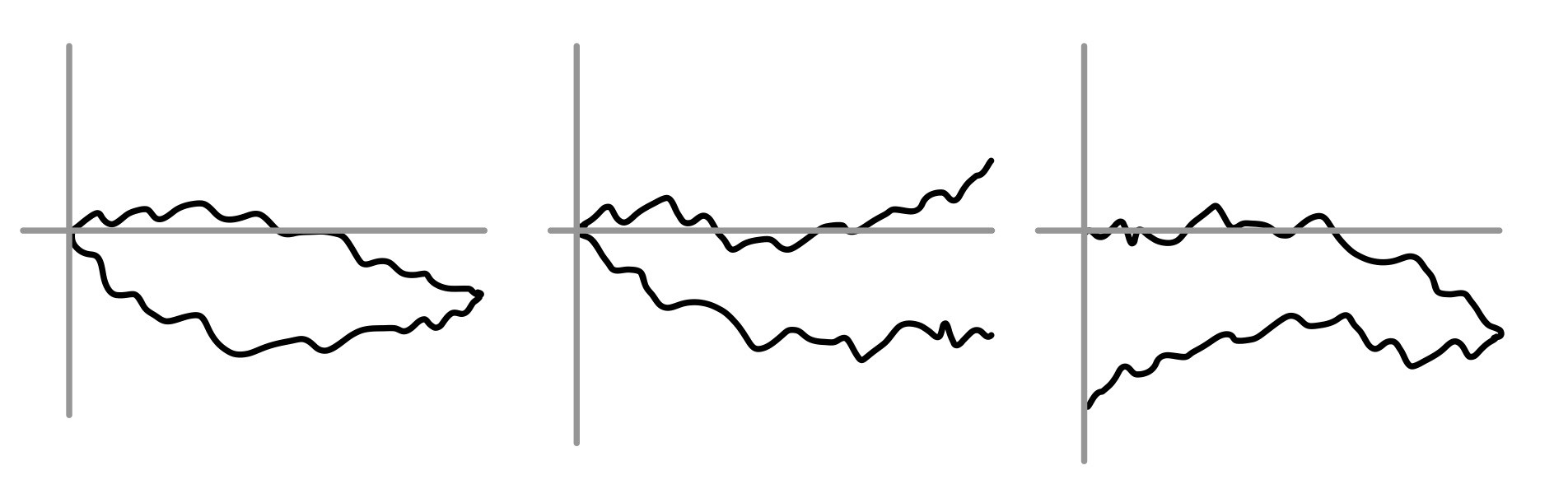} 
        \caption{Limiting measures when $u,v>0$ (Left), $u>0,v=0$ (Middle), and $u=0,v>0$ (Right).}
        \label{lims}
    \end{figure}

We refer to Figure \ref{lims} for depiction of the limiting measures in the three cases. We require a couple of preliminary lemmas before proving Theorem \ref{l.difflimits}. Our first lemma provides a lower bound for the denominator of \eqref{3.2a}.
\begin{lemma} \label{polylbd} Suppose $u+v>0$.
\begin{enumerate}[label=(\alph*),leftmargin=20pt]
    \item \label{part1} Suppose $u,v>0$. There exists  a constant $c>0$ depending on $u,v$ such that $\Ex_{\m{free}}[W_{u,v}^L] \ge cL^{-3/2}$ for all $L\ge 1$.
    \item \label{part2} Suppose $u=0$ or $v=0$. There exists  a constant $\til{c}>0$ depending on $u,v$ such that $\Ex_{\m{free}}[W_{u,v}^L] \ge \til{c} L^{-1/2}$ for all $L\ge 1$.
\end{enumerate}
\end{lemma}
\begin{proof} Clearly it suffices to prove the lemma for all large enough $L$. We shall drop the subscript $\m{free}$ from the notation for simplicity. Suppose $u,v>0$. Consider the event
\begin{align*}
    \m{A}:=\left\{V_L(0), V_L(1) \in [0,1/\sqrt{L}], \ V_L(x) > 0 \mbox{ for all }x\in[0,1]\right\}.
\end{align*}
In words, the $\m{A}$ event requires the boundary points to be within certain $O(1/\sqrt{L})$ window and $B_L$ always stay above $B_L'$.  
Note that
\begin{equation} \label{inder}
    \begin{aligned}
\Ex[W_{u,v}^L] \ge \Pr(\m{A})\cdot \frac{\Ex[W_{u,v}^L\ind_{\m{A}}]}{\Pr(\m{A})} & \ge e^{-|u|-|v|}\cdot \Pr(\m{A}) \cdot  \frac{\Ex\left[\exp\left(-L\int\limits_0^1 e^{-\sqrt{L}V_L(s)}ds\right)\ind_{\m{A}}\right]}{\Pr(\m{A})}.
\end{aligned}
\end{equation}
The ratio on the r.h.s.~of the above equation can be viewed as the conditional expectation of $\exp\left(-L\int\limits_0^1 e^{-\sqrt{L}V_L(s)}ds\right)$ conditioned on the event $\m{A}$. By stochastic monotonicity, this conditional expectation is increasing w.r.t.~ the endpoints of $V_L(\cdot)$. Thus sending $V_L(0),V_L(1) \downarrow 0$, we see that
\begin{align*}
    \mbox{r.h.s.~of \eqref{inder}} \ge e^{-|u|-|v|}\cdot \Pr(\m{A})\cdot \Ex_+\left[\exp\left(-L\int_0^1 e^{-\sqrt{L}V_L(s)}ds\right)\right]
\end{align*}
where $\Pr_+$ is the law of a Brownian excursion. 
We now claim that 
\begin{align}\label{claim01}
    \Pr(\m{A}) \ge c_2 L^{-3/2}, \quad \Ex_+\left[\exp\left(-L\int_0^1 e^{-\sqrt{L}V_L(s)}ds\right)\right] \ge c_1
\end{align}
holds for all large enough $L$. Part~\ref{part1} of the lemma follows by plugging this bound back in \eqref{inder}. We thus focus on proving \eqref{claim01}. Towards this end, note that conditioned on $V_L(0), V_L(1)$, $V_L$ is a Brownian bridge from $V_L(0)$ to $V_L(1)$ with diffusion coefficient $2$. By the reflection principle, the probability of a Brownian bridge from $x>0$ to $y>0$ with diffusion coefficient $2$ stays positive is $(1-e^{-xy})$. 
Thus,
\begin{align*}
\Pr(\m{A})=\sqrt{L}\int_{0}^{1/\sqrt{L}}\int_{0}^{1/\sqrt{L}} \frac{1}{\sqrt{4\pi}}e^{-(x-y)^2/4}(1-e^{-xy})dx dy.
\end{align*}
which is bounded from below by $c_2L^{-3/2}$ for some constant $c_2>0$. This verifies the first part of the claim. 

Using the one point density of Brownian excursion from \cite{durrett1977weak} we have that
\begin{align*}
    \Ex_+\left[L\int_0^1 e^{-\sqrt{L}V_L(s)}ds\right] & = L\int_0^1 \Ex_+[e^{-\sqrt{L}V_L(s)}]ds \\ & = L\int_0^1 \int_0^\infty e^{-\sqrt{L}y}\cdot \frac{2y^2e^{-y^2/2s(1-s)}}{\sqrt{2\pi s^3(1-s)^3}}dyds \\ & = L\int_0^1 \int_0^\infty e^{-z\sqrt{Ls(1-s)}}\cdot \frac{2z^2e^{-u^2/2}}{\sqrt{2\pi}}dz ds \\ & = 2L\int_0^{\infty} \frac{2z^2e^{-z^2/2}}{\sqrt{2\pi}} \int_0^{1/2} e^{-z\sqrt{Ls(1-s)}} ds dz.
\end{align*}
For the inner integral we note that
\begin{align*}
    \int_0^{1/2} e^{-z\sqrt{Ls(1-s)}} ds \le \int_0^{1} e^{-z\sqrt{Ls/2}} ds \le \frac{4}{z^2L}.
\end{align*}
Thus $\Ex_+\left[L\int_0^1 e^{-\sqrt{L}V_L(s)}ds\right] \le c_3$ for some $c_3>0$. This implies 
\begin{align*}
    \Pr_+\bigg(L\int_0^1 e^{-\sqrt{L}V_L(s)}ds \le 2c_3\bigg) \ge 1/2
\end{align*}
which forces
\begin{align*}
    \Ex_+\left[\exp\left(-L\int_0^1 e^{-\sqrt{L}V_L(s)}ds\right)\right ] \ge e^{-2c_3}\cdot \Pr_+\bigg(L\int_0^1 e^{-\sqrt{L}V_L(s)}ds \le 2c_3\bigg) \ge \frac12e^{-2c_3}.
\end{align*}
This verifies the second part of \eqref{claim01} completing the proof of part~\ref{part1}.

\medskip

For part~\ref{part2}, without loss of generality, assume $u>0$ and $v=0$. We consider the following event instead:
\begin{align*}
    \til{\m{A}}:=\left\{V_L(0) \in [0,1/\sqrt{L}], \  V_L(1) \in [0,1], \ V_L(x) > 0 \mbox{ for all }x\in[0,1]\right\}.
\end{align*}
By the same arguments as part~\ref{part1} we have
\begin{align*}
    \Ex[W_{u,v}^L] \ge e^{-|u|}\cdot \Pr(\til{\m{A}})\cdot \Ex_+\left[\exp\left(-L\int_0^1 e^{-\sqrt{L}V_L(s)}ds\right)\right].
\end{align*}
We have already noted that the above expectation is at least $c_2$. By the same argument used to bound $\Pr(\m{A})$, we may similarly bound $\Pr(\til{\m{A}})$ as follows.
\begin{align*}
\Pr(\til{\m{A}})=\sqrt{L}\int_{0}^{1/\sqrt{L}}\int_{0}^{1} \frac{1}{\sqrt{4\pi}}e^{-(x-y)^2/4}(1-e^{-xy})dx dy \ge \til{c}_1L^{-1/2}.
\end{align*}
This verifies part~\ref{part2}.
\end{proof}

Having established a lower bound on the partition function $\Ex[W_{u,v}^L]$, we may now deduce various probability statements under $\Pr_{u,v}^L$ by estimating the numerator. Our first lemma in this direction shows that $u$ and $v$ weight parameters has a pinning effect at starting and ending points.

\begin{lemma}\label{pconv} If $u>0$, then $\Ex_{u,v}^L[V_L(0)^2] \xrightarrow{L\rightarrow +\infty} 0$. If $v>0$, then $\Ex_{u,v}^L[V_L(1)^2] \xrightarrow{L\rightarrow +\infty} 0$.
\end{lemma}

\begin{proof} \textbf{Can't be too low.} Suppose $u>0$. Fix any $\e>0$. We first show that
\begin{align}\label{hi1}
    \Ex_{u,v}^L[V_L(0)^2\ind_{\{V_L(0)\le -\e\}}] \to 0 \mbox{ if } u>0, \quad \Ex_{u,v}^L[V_L(1)^2\ind_{\{V_L(1)\le -\e\}}] \to 0 \mbox{ if } v>0.
\end{align}
Let us rewrite $W_{u,v}^L$ as
\begin{align*}
    W_{u,v}^L=\exp\left(-(u+v)\sqrt{L}V_L(0)-v\sqrt{L}(V_L(1)-V_L(0))-L\int_0^1 e^{-\sqrt{L}(V_L(s))}ds\right).
\end{align*}
Note that under the free law $V_L(0)$ is distributed as Lebesgue measure and $V_L(s)-V_L(0) \stackrel{d}{=} B(s)$ where $B$ is a Brownian motion started from $0$ with diffusion coefficient $2$. Thus by Cauchy-Schwarz inequality:
\begin{align}\label{int1}
    \Ex[W_{u,v}^L V_L(0)^2\ind_{\{V_L(0) \le -\e\}}] = \int_{-\infty}^{-\e} x^2 e^{-(u+v)x\sqrt{L}} \sqrt{\Ex[e^{-2Y_1}]}\sqrt{\Ex[e^{-2Y_2(x)}]} \,dx
\end{align}
where $Y_1:=\exp(v\sqrt{2L}B(0))$ and $Y_2(x):=L\int_0^1 e^{-\sqrt{L}(\sqrt{2} B(s)+x)}ds$.  Let us define the event
\begin{align*}
    \m{A}:= \left\{\sup_{s\in [0,L^{-8}]} B(s) \le \frac{|x|}{2}\right\}.
\end{align*}
We have $\Pr(\m{A}^c) \le \exp(-x^2L^{8}/16)$ and on  $\m{A}$, $$Y_2(x) \ge L\int_0^{L^{-8}} e^{-\sqrt{L}(|x|/2+x)}ds = L^{-7} e^{-\sqrt{L}(|x|/2+x)}.$$
When $x\le 0$, we have $|x|/2+x = -|x|/2$. Thus,
\begin{align*}
    \Ex[e^{-2Y_2(x)}] \le \exp(-x^2L^8/16)+\exp\big(-2L^{-8}e^{\sqrt{L}|x|/2}\big).
\end{align*}
Plugging this estimate back in \eqref{int1} along with the fact that $\Ex[e^{2Y_1}]=e^{v^2L}$, we see that the integral in \eqref{int1} is going to zero exponentially fast as $L\to \infty$ for each fixed $\e>0$.
\begin{align*}
    \Ex[W_{u,v}^L] = \int_{-\infty}^{-\e} e^{-(u+v)x\sqrt{L}}e^{v^2L/2}[e^{-L^{-7}e^{-\sqrt{L}|x|/2}}+e^{-cx^2L^8}] \,dx.
\end{align*}
Thanks to the polynomial lower bound from Lemma \ref{polylbd}, we can conclude  the first part of \eqref{hi1}.
An analogous argument shows that if $u>0$, we also have $\Ex_{u,v}^L\left[V_L(1)^2\ind_{\{V_L(0) \le -\e\}}\right] \to 0.$ and if $v>0$ we have
\begin{align}\label{etog}
    \Ex_{u,v}^L\left[V_L(0)^2\ind_{\{V_L(1) \le -\e\}}\right] \to 0, \quad \Ex_{u,v}^L\left[V_L(1)^2\ind_{\{V_L(1) \le -\e\}}\right] \to 0.
\end{align}
This verifies the second part of \eqref{hi1}.

\medskip
\noindent\textbf{Can't be too high.} We now claim that
\begin{align}\label{hi11}
    \Ex_{u,v}^L[V_L(0)^2\ind_{\{V_L(0)\ge \e\}}] \to 0 \mbox{ if } u>0, \quad \Ex_{u,v}^L[V_L(1)^2\ind_{\{V_L(1)\ge \e\}}] \to 0\mbox{ if } v>0.
\end{align}
Suppose $u,v>0$. Note that
\begin{align*}
    & \Ex\left[W_{u,v}^L \cdot V_L(0)^2\ind_{\{V_L(0)\ge \e, V_L(1) \ge -u\e/2v\}}\right] \\ & \le \Ex\left[e^{-u\sqrt{L}V_L(0)-v\sqrt{L}V_L(1)}\cdot V_L(0)^2\ind_{\{V_L(0)\ge \e, V_L(1) \ge u\e/2v\}}\right] \\ & \le \Ex\left[e^{-u\sqrt{L}V_L(0)+u\sqrt{L}\e/2}\cdot V_L(0)^2\ind_{\{V_L(0)\ge \e\}}\right]   \le e^{u\sqrt{L}\e/2}\int_{\e}^\infty  x^2e^{-ux\sqrt{L}}dx.
\end{align*}
The last integral is of the order $O(e^{-c\e \sqrt{L}})$ which decays faster than polynomial. Hence again thanks to  the polynomial lower bound from Lemma \ref{polylbd}, we can conclude 
\begin{align}\label{hi2}
    \Ex_{u,v}^L\left[V_L(0)^2\ind_{\{V_L(0) \ge \e, V_L(1) \ge -u\e/2v\}}\right] \to 0.
\end{align}
However, by \eqref{etog} we also have 
\begin{align}\label{hi3}
    \Ex_{u,v}^L\left[V_L(0)^2\ind_{\{V_L(1) \le -u\e/2v\}}\right] \to 0.
\end{align}
Thus combining \eqref{hi2} and \eqref{hi3} we arrive at the first part of \eqref{hi11} when both $u,v>0$. If $v=0$, then considering an additional event is not necessary. Using $W_{u,v}^L \le e^{-u\sqrt{L}V_L(0)}$, we can directly verify that $\Ex[W_{u,v}^{L}\ind_{\{V_L(0) \ge \e\}}]=O(e^{-c\e\sqrt{L}})$ which implies the first part of \eqref{hi11} (again via Lemma \ref{polylbd}). An analogous argument leads to the second part of \eqref{hi11} as well. This completes the proof.
\end{proof}

With the control on the end-points, we may now establish Theorem \ref{l.difflimits}.

	\begin{proof}[Proof of Theorem \ref{l.difflimits}] Note that the parameters $u,v$ can be swapped by inverting the time. Indeed, if  $(\Lambda,\Lambda')$ is distributed as $\Pr_{u,v}^L$, then $(\til\Lambda,\til\Lambda')$ is distributed as $\Pr_{v,u}^L$ where $$\til\Lambda(x):=\Lambda(L-x)-\Lambda(L)+\Lambda(0)\mbox{ and }\til\Lambda'(x):=\Lambda'(L-x)-\Lambda'(L)+\Lambda'(0).$$ 
Hence it suffices to consider the case when $u>0$ and $v\ge 0$. Recall the sum and difference processes, $U_L$ and $V_L$, from \eqref{ulvl}. As the Radon-Nikodym derivative depends only on the difference, conditioned on $B_L'(0)$, the following three things hold simultaneously:
\begin{itemize}[leftmargin=20pt]
    \item $U_L$ and $V_L$ are independent,
    \item  $U_L$ is Brownian motion with diffusion coefficient $2$ started from $B_L'(0),$
    \item $V_L$ is absolutely continuous w.r.t.~ a Brownian motion with diffusion coefficient $2$ started from $-B_L'(0)$  with a Radon-Nikodym derivative proportional to:
$$\exp\left(-v\sqrt{L}\cdot V_L(1)-L\int_0^1 e^{-\sqrt{L}V_L(s)}ds\right),$$
\end{itemize}
where in above we crucially used the fact that the sum and difference of two independent Brownian motions are independent. 
We know that $B_L'(0)\to 0$ by Lemma \ref{polylbd}. 
Thus as $L\to \infty$, $U_L$ converges to a Brownian motion $U$ with diffusion coefficient $2$ started from $0$. The law of $V_L$ was studied recently in \cite[Theorem 7.1]{ds25} where it was shown to converge to $V$, where
\begin{itemize}
    \item If $v=0$, $V$ is a Brownian motion with diffusion coefficient $2$ started from $0$ conditioned to stay positive on $(0,1]$.
    \item If $v>0$, $V$ is a Brownian excursion with diffusion coefficient $2$.
\end{itemize}
and $U,V$ are independent. Thus $B_L \to \frac12(U+V)$ and $B_L'\to \frac12(U-V)$. The law of the pair $(\frac12(U+V),\frac12(U-V))$ is precisely the one stated in the theorem. As $\Ex_{u,v}^L[V_L(0)^2]\to 0$ via Lemma \ref{pconv}, $L^2$ convergence for $U_L(x)$ is immediate. The $L^2$ convergence of $V_L$ can be done using stochastic monotonicity. 
	\end{proof}

    \subsection{Tail estimates of the stationary measures}

In this subsection, we report Gaussian tail estimates for the various observables of the stationary measure. 

\begin{proposition}
    \label{tailest} There exists a constant $C>0$ such that for all $L, M\ge 1$ and $y\in [0,L-1]$ we have
    \begin{align}\label{131}
        \Pr_{u,v}^L\bigg(\sup_{x\in [y,y+1]}|\Lambda(x)-\Lambda(y)| \ge M\bigg) \le Ce^{-M^2/C}.
    \end{align}
     \begin{align}\label{1310}
        \Pr_{u,v}^L\bigg(\sup_{x\in [y,y+1]}|\Lambda'(x)-\Lambda'(y)| \ge M\bigg) \le Ce^{-M^2/C}.
    \end{align}
     \begin{align}\label{121}
        \Pr_{u,v}^L\bigg(\sup_{x\in [0,y]}|\Lambda(x)| \ge M\sqrt{y+1}\bigg) \le Ce^{-M^2/C}.
    \end{align}
    \begin{align}\label{111}
        \Pr_{u,v}^L\bigg(\sup_{x\in [y+1,L]}|\Lambda(x)-\Lambda(L)| \ge M\sqrt{L-y}\bigg) \le Ce^{-M^2/C}.
    \end{align}
\end{proposition}

A key input of the proof of the above proposition is the following upper bound for $\Ex[W_{u,v}^L]$ which matches with the lower bound in Lemma \ref{polylbd}.

\begin{lemma}\label{wupbd} Fix $L\ge 1$ and $u,v>0$. Let $y\ge 0$ be such that $[y/L,y/L+1/L]\subset [0,1]$.  
There exists a constant $C > 0$ depending on $u, v$ such that
\begin{align}\label{wupbdeq}
   \Ex\left[\exp\left(-u\sqrt{L}V_L(0)-v\sqrt{L}V_L(1)-L\int_{[0,1]\setminus [y/L,y/L+1/L]} e^{-\sqrt{L}V_L(s)}ds\right)\right] \le CL^{-3/2}.
\end{align}
If $u=0, v>0$ or $u>0, v=0$ we have
\begin{align}\label{wupbdeq2}
   \Ex\left[\exp\left(-u\sqrt{L}V_L(0)-v\sqrt{L}V_L(1)-L\int_{[0,1]\setminus [y/L,y/L+1/L]} e^{-\sqrt{L}V_L(s)}ds\right)\right] \le CL^{-1/2}.
\end{align}
\end{lemma}

The proof of Lemma \ref{wupbd} hinges on the following technical estimates on expectation of certain functional of Brownian bridge.
\begin{lemma}\label{rbound} Fix $L\ge 1$ and $x,y\in \R$. Let $a\ge 0$ be such that $[a/L,a/L+1/L]\subset [0,1]$. Let $B$ be a Brownian bridge on $[0,1]$ from $x/\sqrt{L}$ to $y/\sqrt{L}$ with diffusion coefficient $2$. Define \begin{align*}
    R:=\exp\left(-L\int_{[0,1]\setminus [a/L,a/L+1/L]} e^{-\sqrt{L} B(s)}ds\right).
\end{align*}
We have the following bounds on $\Ex[R]$.
\begin{enumerate}[label=(\alph*),leftmargin=20pt]
    \item \label{parta} If $x<0$ we have
    \begin{align*}
    \Ex[R] \le \exp\left(-e^{|x|/2}\right)+C\exp(-(|x|-4(y-x)_+/L)^2/C).
\end{align*}
\item \label{partb} If $y<0$ we have
\begin{align*}
    \Ex[R] \le \exp\left(-e^{|y|/2}\right)+C\exp(-(|y|-4(x-y)_+/L)^2/C).
\end{align*}
\item \label{partc} If $0\le x,y\le (\log L)^3$ we have
\begin{align*}
    \Ex[R] \le C(\max\{x,y\}+1)/L.
\end{align*}
\item \label{partd} If $\min\{x,y\} \le 0$ and $|x|,|y|\le (\log L)^3$, we have
\begin{align*}
    \Ex[R] \le Ce^{-(-\min\{x,y\})^2/C}(|x-y|+1)/L.
\end{align*}
\end{enumerate}
\end{lemma}
Let us first prove Lemma \ref{wupbd} assuming Lemma \ref{rbound}.
\begin{proof}[Proof of Lemma \ref{wupbd}]  We shall prove \eqref{wupbdeq}. The proof of \eqref{wupbdeq2} is similar.
Note that $V_L(0)$ $V_L(s)-V_L(0) \stackrel{d}{=} B(s)$ where $B$ is a Brownian motion with diffusion coefficient $2$. Conditioning on the endpoints, we get 
\begin{align}
\Ex&\left[\exp\left(-u\sqrt{L}V_L(0)-v\sqrt{L}V_L(1)-L\int_{[0,1]\setminus [a/L,a/L+1/L]} e^{-\sqrt{L}V_L(s)}ds\right)\right]\nonumber\\
&\hspace{1cm}= L^{-1/2}\int_\R \int_\R e^{-ux-vy} \mathbb{E}[R]  \frac{1}{\sqrt{4\pi}}e^{-\frac{(x-y)^2}{4L}} dxdy \label{equation1}
\end{align}
where $R$ is as in Lemma \ref{rbound}. To estimate the double integral in \eqref{equation1}, we split the range of integrals into several regions.

\begin{enumerate}[label=(\alph*), leftmargin=20pt]
    \item Consider the range of integral on the region: $x<-(\log L)^2$ and $x\le y \le x + L|x|/8$. Thanks to the bound in Lemma \ref{rbound}\ref{parta}, we have that $\Ex[R] \le Ce^{-x^2/C}$ on this region. Plugging this bound and the bound $e^{-(x-y)^2/4L}\le 1$, for $L$ large enough we get
    \begin{align*}
        & \int_{-\infty}^{-(\log L)^2}\int_{x}^{x+L|x|/8} e^{-ux-vy} \mathbb{E}[R]  \frac{1}{\sqrt{4\pi}}e^{-\frac{(x-y)^2}{4L}} dydx \\ & \le C\int_{-\infty}^{-(\log L)^2}\int_{0}^{L|x|/8} e^{-(u+v)x-vz-x^2/C}    dz dx \le CL^{-1}.
    \end{align*}
    On the region $x<-(\log L)^2$ and $y \ge x + L|x|/8$, we use the trivial bound $R\le 1$ and use the gaussian factor $e^{-(x-y)^2/4L}$ to get
    \begin{align*}
       & \int_{-\infty}^{-(\log L)^2}\int_{x+L|x|/8}^\infty e^{-ux-vy}\mathbb{E}[R]  \frac{1}{\sqrt{4\pi}}e^{-\frac{(x-y)^2}{4L}} dydx \\ & \le C\int_{-\infty}^{-(\log L)^2}\int_{L|x|/8}^\infty e^{-(u+v)x-vz-z^2/4L}    dz dx \le CL^{-1}.
    \end{align*}
    Combining the above two displays we obtain,
    \begin{align*}
        \int_{-\infty}^{-(\log L)^2}\int_{x}^\infty e^{-ux-vy}\mathbb{E}[R]  \frac{1}{\sqrt{4\pi}}e^{-\frac{(x-y)^2}{4L}} dxdy \le CL^{-1}.
    \end{align*}
    An analogous argument using Lemma \ref{rbound}\ref{partb} instead shows that \begin{align*}
        \int_{-\infty}^{-(\log L)^2}\int_{y}^\infty e^{-ux-vy}\mathbb{E}[R]  \frac{1}{\sqrt{4\pi}}e^{-\frac{(x-y)^2}{4L}} dxdy \le CL^{-1}.
    \end{align*}
    This take cares of the region where $x\le -(\log L)^2$ or $y\le -(\log L)^2$. 
    \item We may now restrict ourselves to the region where $x,y \ge -\log^2 L$. When any one of the variables is larger than $\log^3L$, using the bound $R\le 1$ and $e^{-(x-y)^2/4L}\le 1$ we get
    \begin{align*}
        \int_{-(\log L)^2}^\infty \int_{(\log L)^3}^\infty e^{-ux-vy}\Ex[R]\frac{1}{\sqrt{4\pi}}e^{-\frac{(x-y)^2}{4L}} dydx \le \frac{1}{\sqrt{4\pi}}\int_{-(\log L)^2}^\infty \int_{(\log L)^3}^\infty e^{-ux-vy} dydx \le CL^{-1}, \\
        \int_{-(\log L)^2}^\infty \int_{(\log L)^3}^\infty e^{-ux-vy}\Ex[R]\frac{1}{\sqrt{4\pi}}e^{-\frac{(x-y)^2}{4L}} dxdy \le \frac{1}{\sqrt{4\pi}}\int_{-(\log L)^2}^\infty \int_{(\log L)^3}^\infty e^{-ux-vy} dxdy \le CL^{-1}.
    \end{align*}

    \item We now control the region where $|x|,|y|\le \log^3L$. We divide this range into three parts depending on the sign of $x$ and $y$. First, when $x,y\ge 0$, using Lemma \ref{rbound}\ref{partc} we obtain
\begin{align*}
       & \int_{0}^{\log^3 L}\int_{0}^{\log^3 L} e^{-ux-vy} \mathbb{E}[R]  \frac{1}{\sqrt{4\pi}}e^{-\frac{(x-y)^2}{4L}} dxdy \\ & \le CL^{-1}\int_{0}^{\log^3 L}\int_{0}^{\log^3 L} e^{-ux-vy} (\max\{x,y\}+1)  dxdy \le C'L^{-1}.
    \end{align*}  
On the other hand, if $x<0$ and $x\le y$, using Lemma \ref{rbound}\ref{partd} we obtain
\begin{align*}
   & \int_{-(\log L)^3}^{0 }\int_{x}^{\log^3 L} e^{-ux-vy} \mathbb{E}[R]  \frac{1}{\sqrt{4\pi}}e^{-\frac{(x-y)^2}{4L}} dydx \\ & \le CL^{-1}\int_{-(\log L)^3}^{0 }\int_{x}^{\log^3 L} e^{-ux-vy-|x|^2/C} (|x-y|+1)  dydx  \le C'L^{-1},
\end{align*}
and analogously
\begin{align*}
   & \int_{-(\log L)^3}^{0 }\int_{y}^{\log^3 L} e^{-ux-vy} \mathbb{E}[R]  \frac{1}{\sqrt{4\pi}}e^{-\frac{(x-y)^2}{4L}} dxdy   \le CL^{-1}.
\end{align*}
Combining the three above estimates we derive
\begin{align*}
       & \int_{-\log^3 L}^{\log^3 L}\int_{-\log^3 L}^{\log^3 L} e^{-ux-vy} \mathbb{E}[R]  \frac{1}{\sqrt{4\pi}}e^{-\frac{(x-y)^2}{4L}} dxdy \le CL^{-1}.
    \end{align*} 
\end{enumerate}
Plugging in the bounds obtained in above three cases back in \eqref{equation1} leads to \eqref{wupbdeq}. 
\end{proof}

\begin{proof}[Proof of Lemma \ref{rbound}] Suppose $x<0$. Consider the event:
\begin{align*}
    \m{E}:=\left\{B(s) \le \frac{x}{2\sqrt{L}} \mbox{ for all }s\in [0,2/L]\right\}.
\end{align*}
As $R\le 1$, we have $\Ex[R] \le \Ex[R\cdot \ind_{\m{E}}]+\Pr(\m{E}^c)$. Note that $R\cdot \ind_{\m{E}} \le \exp\left(-e^{|x|/2}\right)$. By Brownian bridge computations we have $\Pr(\m{E}^c)$ is bounded above by $C\exp(-(|x|-4(y-x)_+/L)^2/C)$. We thus arrive at the bound in \ref{parta}. When $y<0$, an analogous argument by controlling the Brownian bridge near $s=1$ leads to the bound in \ref{partb}.

\medskip

We now turn towards the proof of \ref{partc} and \ref{partd}. For each $r\in \Z$, define the events:
\begin{align*}
    \m{A}_r:=\{B(s) \ge -r/\sqrt{L} \mbox{ for all }s\in [0,1]\}, \quad \m{B}_r:=\m{A}_{r}\cap \m{A}_{r-1}^c.
\end{align*}
Let $\tau_{r,1}$ and $\tau_{r,2}$ be the first and last time  the process $B(s)$ hits $r/\sqrt{L}$ respectively. Define
\begin{align*}
    \m{C}_{r,1} := \left\{\sup_{s\in [\tau_{r,1},\tau_{r,2}+2/L]} B(s) \le -\frac{r}{2\sqrt{L}} \right\},  \quad \m{C}_{r,2}:=\left\{\sup_{s\in [\tau_{r,2}-2/L,\tau_{r,2}]} B(s) \le -\frac{r}{2\sqrt{L}} \right\}.
\end{align*}
Note that $\m{A}_r \uparrow \Omega$ as $r\to \infty$ and $\m{A}_r$ is empty  for $r < -\min\{x,y\}$. Thus, 
\begin{align}\label{discr}
    \Ex[R]=\sum_{r=-\min\{x,y\}}^\infty \Ex[R\cdot \ind_{\m{B}_r}].
\end{align}
We focus on bounding $\Ex[R\cdot \ind_{\m{B}_r}]$. We first bound $\Ex\left[R\cdot \ind_{\m{B}_{r+1}\cap \{\tau_{r,1}\le \frac12\}}\right]$. 
Note that $$\ind_{\m{B}_{r+1} \cap \{\tau_{r,1}\le 1/2\}} = \ind_{\m{B}_{r+1} \cap \{\tau_{r,1}\le 1/2\} \cap \m{C}_{r,1}} +\ind_{\m{B}_{r+1} \cap \{\tau_{r,1}\le 1/2\}\cap \m{C}_{r,1}^c}.$$
On $\m{C}_{r,1}\cap \{\tau_{r,1}\le 1/2\}$, we have $R\le e^{-e^{r/2}}$. So, 
\begin{equation}\label{rbre}
    \begin{aligned}
    \Ex[R\cdot \ind_{\m{B}_r\cap \{\tau_{r,1}\le 1/2\} \cap \m{C}_{r,1}}] & \le e^{-e^{r/2}}\cdot \Pr(\m{A}_{r+1}) \\ & = e^{-e^{r/2}}\cdot \big(1-e^{-(x+r+1)(y+r+1)/L}\big) \\ & \le e^{-e^{r/2}}\frac{(x+r+1)(y+r+1)}{L}.
\end{aligned}
\end{equation}

On the other hand, using the trivial bound $R\le 1$ we have
$$\Ex[R\cdot \ind_{\m{B}_r\cap \{\tau_{r,1}\le 1/2\} \cap \m{C}_{r,1}^c}]\le \Pr( \m{B}_{r+1}\cap \{\tau_{r,1}\le 1/2\} \cap \m{C}_{r,1}^c).$$ 
Since $\tau_{r,1}$ is a stopping time, the process $B(s)$ on $[\tau_{r,1},1]$ is a Brownian bridge from $-r/\sqrt{L}$ to $y/\sqrt{L}$. Note that $\m{B}_{r+1} \subset \{ B(s) \ge -(r+1)/\sqrt{L} \mbox{ for all }s\in [\tau_{r,1}+2/L,1]\}.$ Conditioning on the value of $\sqrt{L} B(\tau_{r,1}+2/L)=u$ we have 
\begin{align*}
    \Pr( \m{B}_{r+1}\cap \{\tau_{r,1}\le 1/2\} \cap \m{C}_{r,1}^c) \le \sup_{c\in [0,1/2]}\int_{-r-1}^\infty \min\{1, e^{(r/2)(r/2+u)}\}(1-e^{-(u+r+1)(y+r+1)/L})f_c(u)du
\end{align*}
where $f_c(u)$ is the density of $\sqrt{L} B(\tau_{r,1}+2/L)$ with $\tau_{r,1}=c$ \textemdash a Gaussian random variable with mean $-r/\sqrt{L}+\frac{2/L}{(1-c)}(y+r)/\sqrt{L}$ and variance $2-\frac{4}{L(1-c)}$. It is not hard to check that the above explicit integral is at most $Ce^{-r^2/C}(y+r+1)/L$.  Combining this with the estimate in \eqref{rbre} we get that
$$\Ex\left[R\cdot \ind_{\m{B}_{r+1}\cap \{\tau_{r,1}\le \frac12\}}\right] \le Ce^{-r^2/C}(y+r+1)/L.$$ Using a similar argument (using $\m{C}_{r,2}$ event instead), we see that $$\Ex[R\cdot \ind_{\m{B}_{r+1}\cap \{\tau_{r,2}\ge \frac12\}}] \le Ce^{-r^2/C}(x+r+1)/L. $$ Adding the last two estimates  we get $\Ex[R\cdot \ind_{\m{B}_{r+1}}] \le Ce^{-r^2/C}(\max\{x,y\}+r+1)/L$. Plugging this bound back in \eqref{discr} and summing over $r$ leads to both \ref{partc} and \ref{partd}.
\end{proof}

\begin{proof}[Proof of Proposition \ref{tailest}] For convenience, we split the proof into two steps. 

\medskip

\noindent\textbf{Step 1}. In this and next step we prove \eqref{131} for $y\in [1,L-2]$. The proof of other cases in Lemma \ref{tailest} follows upon certain modifications of this case (explained at the end of the proof). 

Suppose $y\in [1,L-2]$. Fix $M$ large enough. Let $\tau_1$ be the first time in $[y-1,y]$ such that $\Lambda(x)-\Lambda'(x)\ge -\frac{M}{4}$. If no such time exists we set $\tau_1=-\infty$. Let $\tau_2$ be the last time in $[y,y+1]$ such that $\Lambda(x)-\Lambda'(x)\ge -\frac{M}{4}$. If no such time exists we set $\tau_2=-\infty$. Let us define the events
\begin{align*}
    \m{A}_1:=\left\{\sup_{x\in [y-1,y]} \Lambda(x)-\Lambda'(x) \ge -M/4\right\}, \qquad \m{A}_2:=\left\{\sup_{x\in [y+1,y+2]} \Lambda(x)-\Lambda'(x) \ge -M/4\right\}.
\end{align*}
We claim that 
\begin{align} \label{132}
    & \Pr_{u,v}^L(\m{A}_1^c) \le Ce^{-e^{M/4}}, \quad \Pr_{u,v}^L(\m{A}_1^c) \le Ce^{-e^{M/4}}, \\ \label{133}
    & \Pr_{u,v}^L\bigg(\m{A}_1\cap \m{A}_2 \cap \bigg\{\sup_{x\in [y,y+1]}|\Lambda(x)-\Lambda(y)| \ge M\bigg\}\bigg) \le Ce^{-M^2/C},
\end{align}
for some constant $C>0$. Applying an union bound gives us \eqref{131}. Let us thus focus on proving \eqref{132} and \eqref{133}. We prove \eqref{132} in this step. We prove the first inequality of \eqref{132}, the second inequality is analogous. Define
\begin{align*}
   {W}_{u,v}^{-}:=\exp\left(-u(\Lambda(0)-\Lambda'(0))-v(\Lambda(L)-\Lambda'(L))-\int_{[0,L]\setminus [y-1,y]} e^{-(\Lambda(s)-\Lambda'(s))}ds\right).
\end{align*}
Note that on $\m{A}_1^c$, we have  $W_{u,v} \le {W}_{u,v}^{-}\cdot e^{-e^{M/4}}$. Thus,
\begin{align}\label{rex}
\Pr_{u,v}^L(\m{A}_1^c)=\frac{\Ex[W_{u,v}\ind_{\m{A}_1^c}]}{\Ex[W_{u,v}]} \le e^{-e^{M/4}}\frac{\Ex\left[{W}_{u,v}^-\right]}{\Ex[W_{u,v}]} = e^{-e^{M/4}}\frac{\Ex\left[{W}_{u,v}^{L,-}\right]}{\Ex[W_{u,v}^{L}]}. 
\end{align}
where $W_{u,v}^{L,-}$ is the scaled version of $W_{u,v}^{-}$:
\begin{align*}
   {W}_{u,v}^{L,-}:=\exp\left(-u\sqrt{L}V_L(0)-v\sqrt{L}V_L(1)-L\int_{[0,1]\setminus [y/L-1/L,y/L]} e^{-\sqrt{L}V_L(s)}ds\right).
\end{align*}
By Lemmas \ref{polylbd} and \ref{wupbd}, we see that the ratio of expectations on the r.h.s.~of \eqref{rex} is bounded. Thus, $\Pr_{u,v}^L(\m{A}_1^c) \le e^{-e^{M/4}}$. 

\medskip

\noindent\textbf{Step 2.} In this step, we prove \eqref{133}. Note that $[\tau_1,\tau_2]$ is a stopping domain. Consider the $\sigma$-field $\mathcal{F}=\sigma((\Lambda(x),\Lambda'(x)) : x\notin (\tau_1,\tau_2))$. Note that $\m{A}_1\cap\m{A}_2\subset \mathcal{F}$.
\begin{equation}
    \begin{aligned}
        \label{23a}
    & \Pr_{u,v}^L\bigg(\m{A}_1\cap \m{A}_2 \cap \bigg\{\sup_{x\in [y,y+1]}|\Lambda(x)-\Lambda(y)| \ge M\bigg\}\bigg) \\ & = \Ex_{u,v}^L\left[\ind_{\m{A}_1\cap\m{A}_2}\Ex\left[\ind\{\sup_{x\in [y,y+1]}|\Lambda(x)-\Lambda(y)| \ge M\}\mid \mathcal{F}\right]\right].
    \end{aligned}
\end{equation}
By the Gibbs property \eqref{gprop} we have that
\begin{align}
\Ex_{u,v}^L\left[\ind_{\{\sup_{x\in [y,y+1]}|\Lambda(x)-\Lambda(y)| \ge M\}}\mid \mathcal{F}\right]  & = \Pr_{\hat{W}}^{(a,a')\to (b,b')}\left(\sup_{x\in [y,y+1]}|\Lambda(x)-\Lambda(y)| \ge M\right) \nonumber\\ & \le \Pr_{\hat{W}}^{(a,a')\to (b,b')}\left(\sup_{x\in [\tau_1,\tau_2]}|\Lambda(x)-a| \ge M/2\right). \label{23b}
\end{align}
where $a=\Lambda(\tau_1), a'=\Lambda'(\tau_1)$, $b=\Lambda(\tau_2), a'=\Lambda'(\tau_2)$ and $\Pr_{\hat{W}}^{(a,a')\to (b,b')}$ is the absolutely continuous with respect to $\Pr^{(a,a')\to (b,b')}$, which is the law of two independent Brownian bridges on $[\tau_1,\tau_2]$ from $(a,a')$ to $(b, b')$, via the Radon-Nikodym derivative:
\begin{align*}
    \hat{W}:=\exp\left(-\int_{\tau_1}^{\tau_2} e^{\Lambda'(s)-\Lambda(s)}ds\right).
\end{align*}
By stochastic monotonicity:
\begin{align}  \ind_{\m{A}_1\cap\m{A}_2} \cdot &\Pr_{\hat{W}}^{(a,a')\to (b,b')}\left(\sup_{x\in [\tau_1,\tau_2]}\Lambda(x)-a \ge M/2\right) \nonumber\\ 
   &\le \ind_{\m{A}_1\cap\m{A}_2} \cdot\Pr_{\hat{W}}^{(a+M/4,a'+M/4)\to (b,b')}\left(\sup_{x\in [\tau_1,\tau_2]} \Lambda(x)-a \ge M/2\right). \label{234}
\end{align}
On the event, $\m{A}_1\cap\m{A}_2$, $\tau_2-\tau_1\le 3$ and $a+M/4 =a'$ and $b+M/4=b'$ (by continuity of processes equality must occur at stopping times). Under this scenario $\Ex^{(a+M/4,a')\to (b+M/4,b')}[\hat{W}] =\Ex^{(0,0)\to (0,0)}[\hat{W}]=c$, which is free of $M$. As $0<\hat{W}\le 1$ we have
\begin{align}
   \mbox{r.h.s.~of~\eqref{234}}  & \le c^{-1} \cdot \ind_{\m{A}_1\cap\m{A}_2} \cdot \Pr^{(a+M/4,a'+M/4)\to (b,b')}\left(\sup_{x\in [\tau_1,\tau_2]}\Lambda(x)-a \ge M/2\right)\nonumber\\ & \le c^{-1}\cdot Ce^{-M^2/C}, \label{235}
\end{align}
where the last inequality follows by Brownian bridge tail estimates. Similarly by stochastic monotonicity:
\begin{align*}
    \Pr_{\hat{W}}^{(a,a')\to (b,b')}\left(\sup_{x\in [\tau_1,\tau_2]}\Lambda(x)-a \ge M/2\right) \le \Pr_{\hat{W}}^{(a,-\infty)\to (b,-\infty)}\left(\inf_{x\in [\tau_1,\tau_2]} \Lambda(x)-a \le -M/2\right) \le Ce^{-M^2/C}.
\end{align*}
where again we use Brownian bridge tail estimates for the last inequality. Combining the above inequality with \eqref{235} we get that
\begin{align*}
    \ind_{\m{A}_1\cap\m{A}_2} \cdot \Pr_{\hat{W}}^{(a,a')\to (b,b')}\left(\sup_{x\in [\tau_1,\tau_2]}|\Lambda(x)-a| \ge M/2\right) \le Ce^{-M^2/C}.
\end{align*}
Plugging this back in \eqref{23b}, in view of \eqref{23a}, we derive \eqref{133}. 

\medskip

The proof of \eqref{121} and and the proof of \eqref{131} when $y\in [0,1]$ is exactly the same as the above proof upon consider the stopping domain $[0,\tau_2]$ and the $\sigma$-field $\mathcal{F}_+:=\sigma((\Lambda(x),\Lambda'(x)) : x\notin [0,\tau_2))$. The proof of \eqref{111} and the proof of \eqref{131} when  $y\in [L-2,L-1]$ is exactly the same as the above proof upon consider the stopping domain $[\tau_1,L]$ and the $\sigma$-field $\mathcal{F}_-:=\sigma((\Lambda(x),\Lambda'(x)) : x\notin (\tau_1,L])$. Finally, the proof of \eqref{1310} follows from \eqref{131} by noting that if $(\Lambda,\Lambda') \sim \Pr_{u,v}^L$, then $(-\Lambda'+\Lambda'(0),-\Lambda+\Lambda'(0)) \sim \Pr_{u,v}^L$.
\end{proof}

    \section{Proof of Theorems \ref{varde} and  \ref{t.main1}} \label{sec:3}

    In this section, we prove Theorems \ref{varde} and \ref{t.main1} assuming bounds on the asymptotic variance. As explained in the introduction, the proof of Theorem \ref{varde} follows similar ideas to those in Section 2 of \cite{yu2}, where an analogous statement was proved for the periodic KPZ case.

\smallskip

To outline the proof, we first introduce some notation. Let $\mathcal{Z}(x,s;y,t)$ be the time-$t$ solution to the SHE \eqref{opshe} with boundary conditions \eqref{robin} started from Dirac delta initial data $\delta(x)$ at time $s$. These are often called propagators of the SHE from $(x,s)$ to $(y,t)$.  Let $C_+([0,L])$ be the space of all non-negative continuous functions on $[0,L]$ that are not identical to zero.
For $f,g\in C_+([0,L])$, using the propagators, we may define  generalizations of partition functions:
\begin{align*}
   \mathcal{Z}(g,s;y,t) & :=\int_0^L \mathcal{Z}(x,s;y,t)g(x)dx, \quad \mathcal{Z}(x,s;f,t)  :=\int_0^L \mathcal{Z}(x,s;y,t)f(y)dy, \\ \mathcal{Z}(g,s;f,t) & :=\int_0^L\int_0^L \mathcal{Z}(x,s;y,t)g(x)f(y)dxdy.
\end{align*}
Here, $\mathcal{Z}(g,s;x,t)$ is the solution to the SHE \eqref{opshe} with boundary conditions \eqref{robin} started from initial data $g(x)$ at time $s$ and
$\mathcal{Z}(g,s;f,t)$ can be viewed as the partition function of the point-to-point
directed polymer on the strip of length $t-s$, with the two endpoints sampled from $g$ and $f$ respectively. When the $g\equiv 1$ or $f\equiv1$, we simply write $\mathcal{Z}(\mathbf{1},s;y,t)$ or $\mathcal{Z}(y,s;\mathbf{1},t)$ for $\mathcal{Z}(g,s;y,t)$ or $\mathcal{Z}(x,s;f,t)$ respectively. For the log-partition functions, we write
\begin{align*}
    \mathcal{H}(g,s;y,t):=\log \mathcal{Z}(g,s;y,t), \quad \mathcal{H}(x,s;f,t):=\log \mathcal{Z}(x,s;f,t),\quad \mathcal{H}(g,s;f,t):=\log\mathcal{Z}(g,s;f,t).
\end{align*}
and for the normalized partition function we write
\begin{align}\label{npart}
    \rho(g,s;y,t):=\frac{ \mathcal{Z}(g,s;y,t)}{\mathcal{Z}(g,s;\mathbf{1},t)}, \quad \rho(x,s;f,t):=\frac{ \mathcal{Z}(x,s;f,t)}{\mathcal{Z}(\mathbf{1},s;f,t)}.
\end{align}

We will allow $g,f$ to be random. Since allowing random initial and ending data allows several sources of randomness, to avoid confusion,  we use $\E$ to denote the expectation only with respect to the noise $\xi$. The expectations with respect to $f, g$,  will
be denoted by $\Ex_f , \Ex_g$ respectively. When there is no possibility of
confusion, we use $\Ex$ as the expectation with respect to all possible randomnesses. When $g$ is equal in distribution to $e^{\Lambda(x)}$, where $\Lambda(x)$ is the stationary measure, $\mathcal{H}(e^{\Lambda},0;y,t)$ is the height function corresponding to the stationary open KPZ equation. In this case, we drop $(e^{\Lambda},0)$ from the notation and simply write $\mathcal{H}(y,t)$ and $\mathcal{H}(f,t)$. 

\medskip

The key idea behind Theorem \ref{varde} is the following decomposition formula for the centered height function (the analogous decomposition for the periodic KPZ appears in Proposition 3.1 in \cite{yu2}).
\begin{proposition}\label{decom} Fix any $t>0$ and adopt the above notations and conventions. Let $\mathcal{H}(0,t)$ be the stationary solution of the open KPZ equation \eqref{kpz} on $[0,L]$ with boundary data \eqref{openkpz}. We have
\begin{align*}
    \mathcal{H}(0,t)-\Ex[\mathcal{H}(0,t)] = I_L(t)+Y_L(0)-Y_L(t)
\end{align*}
where 
\begin{align}\label{iljl}
    & I_L(t)=\Ex_{\Lambda_2}\left[\mathcal{H}(e^{\Lambda_2},t)\right]-\Ex_{\Lambda_2}\E\left[\mathcal{H}(e^{\Lambda_2},t)\right], \qquad Y_L(s)=\Ex_{\Lambda_2}\left[\mathcal{H}(e^{\Lambda_2},s)\right]-\mathcal{H}(0,s).
\end{align}
where $\Lambda_2$ is distributed according to the stationary measure independent of the noise and the initial data.
\end{proposition}
\noindent  Given the above decomposition, we observe that
\begin{align*}
\left|\sqrt{\operatorname{Var}(\mathcal{H}(0,t))}-\sqrt{\operatorname{Var}(I_L(t))}\right| \le \sqrt{\operatorname{Var}(Y_L(0)-Y_L(t))}.
\end{align*}
Theorem \ref{varde} then follows from the following two propositions.

\begin{proposition}[Leading order] \label{proplead} For each $t>0$ we have $\operatorname{Var}(I_L(t))=t\sigma_L^2$ where $\sigma_L^2$ is defined in \eqref{def:sigma}.
    \end{proposition}

\begin{proposition}[Subleading term] \label{prop C.6}
  For any $t\geq 0$, then we have as $L\to \infty$  
$$\frac{1}{L}\operatorname{Var}[Y_L(t)]\rightarrow \operatorname{Var}\left[ \mathbb{E}_{B_2}\max_{x\in[0,1]}\{B_1(x)+B_2(x)\} \right],$$
where $B_1,B_2$ are independent copies of $B$ coming from Theorem \ref{l.difflimits}.
\end{proposition}

In the following subsection, Section \ref{sec3.1}, we prove the decomposition formula (Proposition \ref{decom}) and the asymptotics for the subleading term (Proposition \ref{prop C.6}). The proof of Proposition \ref{proplead}, on the other hand, requires a few additional tools. We defer its proof to Section \ref{sec:lead}.

\subsection{Proof of Propositions \ref{decom} and \ref{prop C.6}} \label{sec3.1}

In this section, we prove Propositions \ref{decom} and \ref{prop C.6}. We continue with the same notation introduced before Proposition \ref{decom}.

\begin{proof}[Proof of Proposition \ref{decom}] Let us denote $e^{\Lambda_1}$ to denote the initial data for $\mathcal{H}$ so that $\mathcal{H}(e^{\Lambda_2},t)=\mathcal{H}(e^{\Lambda_1},0;e^{\Lambda_2},t)$. Given the definitions of $I_L(t)$ and $Y_L(s)$ from \eqref{iljl}, we have that
\begin{align*}
    I_L(t)+Y_L(0)-Y_L(t)=\mathcal{H}(0,t)+Y_L(0)-\Ex_{\Lambda_2}\E[\mathcal{H}(e^{\Lambda_2},t)]
\end{align*}
Thus it suffices to show
\begin{align}\label{claim3}
    \Ex_{\Lambda_2}\E[\mathcal{H}(e^{\Lambda_2},t)]-\Ex[\mathcal{H}(0,t)]= Y_L(0).
\end{align}
Towards this end, we first claim that the initial and final data for $\mathcal{H}$ can be swapped, i.e., for any given $f,g \in C_+([0,L])$ we have
\begin{equation}
    \label{swap}
    \begin{aligned}
    \mathcal{H}(g,0;f,t)  & =\int_0^L\int_0^L \mathcal{Z}(y,0;x,t)g(y)f(x)  dxdy \\ & \stackrel{d}{=} \int_0^L\int_0^L \mathcal{Z}(y,0;x,t)g(x)f(y)  dxdy=\mathcal{H}(f,0;g,t).
\end{aligned}
\end{equation}
Indeed, \eqref{swap} follows as at the level of propagators we have the following invariance in swapping:
\begin{align}
    \label{swinv}
    \{\mathcal{Z}(y,0;x,t)\}_{x,y\in [0,L]} \stackrel{d}{=} \{\mathcal{Z}(x,0;y,t)\}_{x,y\in [0,L]}.
\end{align}
The above claim is standard in the SHE setting and can be proved by approximation to the smooth noise case, which in turn can be shown using the Feynman–Kac formula \cite[Proposition 2.2]{parekh2022ergodicity}. We refer to Appendix B of \cite{yu2} for details related to periodic KPZ.

\medskip

Thanks to \eqref{swap} we have that
\begin{equation}
    \label{inde}
    \begin{aligned}
    \Ex_{\Lambda_2}\E[\mathcal{H}(e^{\Lambda_2},t)]-\Ex[\mathcal{H}(0,t)] 
 & =\Ex_{\Lambda_2}\E[\mathcal{H}(e^{\Lambda_2},0;e^{\Lambda_1},t)]-\Ex[\mathcal{H}(0,t)] \\ & =\Ex_{\Lambda_2}\E\left[\log \int_0^L\frac{\mathcal{Z}(e^{\Lambda_2},0;x,t)e^{\Lambda_1(x)}dx}{\mathcal{Z}(e^{\Lambda_2},0;0,t)}\right] \\ & =\Ex_{\Lambda_2}\left[\log \int_0^L e^{\Lambda_2(x)}e^{\Lambda_1(x)}dx\right]
\end{aligned}
\end{equation}
where the last equality is due to the stationarity: $\mathcal{Z}(e^{\Lambda_2},0;\cdot,t)/\mathcal{Z}(e^{\Lambda_2},0;0,t)\stackrel{d}{=} e^{\Lambda_2(\cdot)}$.
 On the other hand, as $\mathcal{H}(0,s)=\Ex_{\Lambda_2}[\log \mathcal{Z}(e^{\Lambda_1},0;0,s)]$, for any $s\ge 0$ we have
 \begin{equation}
     \label{eqlaw}
     \begin{aligned}
    Y_L(s) & =\Ex_{\Lambda_2}[\mathcal{H}(e^{\Lambda_2},t)]-\mathcal{H}(0,s) \\ & = \Ex_{\Lambda_2}\left[\log \int_0^L\frac{\mathcal{Z}(e^{\Lambda_1},0;x,s)e^{\Lambda_2(x)}dx}{\mathcal{Z}(e^{\Lambda_1},0;0,s)}\right]=\Ex_{\Lambda_2}\left[\int_0^L e^{\Lambda_1(x)}e^{\Lambda_2(x)}dx\right]
\end{aligned}
 \end{equation}

where the last equality is again via stationarity: $\mathcal{Z}(e^{\Lambda_1},0;\cdot,t)/\mathcal{Z}(e^{\Lambda_1},0;0,t)\stackrel{d}{=} e^{\Lambda_1(\cdot)}$. In view of \eqref{inde}, this verifies \eqref{claim3}, completing the proof.
\end{proof}

    \begin{proof}[Proof of Proposition \ref{prop C.6}]
        Thanks to \eqref{eqlaw}, applying a diffusive scaling yields
        \begin{align*}
            \frac{Y_L(s)-\log L}{\sqrt{L}} = \frac1{\sqrt{L}}\Ex_{B_{L,2}} \left[\log \int_0^1 e^{\sqrt{L}(B_{L,1}(x)+B_{L,2}(x))}dx\right],
        \end{align*}
        where $B_{L,i}(x):=\frac1{\sqrt{L}}\Lambda_i(xL)$. In Theorem \ref{l.difflimits}, we studied diffusive limits of $B_{L,i}$ and showed that $B_{L,i}\to B_i$ (in distribution and in $L^2$) where $B_i$'s are independent copies of $B$ coming from Theorem \ref{l.difflimits}. In particular, by Skorohod theorem and Laplace principle, we may get a common probability space where
        \begin{align*}
            \frac1{\sqrt{L}}\log \int_0^1 e^{\sqrt{L}(B_{L,1}(x)+B_{L,2}(x))}dx \to \max_{x\in [0,1]} (B_1(x)+B_2(x)). 
        \end{align*}
        almost surely as $L\to \infty$. On the other hand,
         \begin{align*}
            \min_{x\in [0,1]} (B_{L,1}(x)+B_{L,2}(x)) \le \frac1{\sqrt{L}}\log \int_0^1 e^{\sqrt{L}(B_{L,1}(x)+B_{L,2}(x))}dx  \le  \max_{x\in [0,1]} (B_{L,1}(x)+B_{L,2}(x)).
        \end{align*}
        By Theorem \ref{l.difflimits}, both upper and lower bounds in above equation are in $L^2$. Thus the desired convergence follows by an application of dominated convergence theorem. 
    \end{proof}

\subsection{Proof of Proposition \ref{proplead}} \label{sec:lead}

We prove Proposition \ref{proplead} in this section. As before, we assume $\Lambda_1,\Lambda_2$ are two independent copies of the stationary measure. By definition \eqref{iljl},
\begin{align}\label{fomr}
I_L(t) \stackrel{d}{=} \Ex_{\Lambda_2}\left[\mathcal{H}(e^{\Lambda_1},0;e^{\Lambda_2},t)-\E\left[\mathcal{H}(e^{\Lambda_1},0;e^{\Lambda_2},t)\right]\right].
\end{align}

We divide the proof into three steps. In \textbf{Step 1}, we derive a suitable formula for the random variable inside the outer expectation, assuming the square integrability of a certain random variable. In \textbf{Step 2}, we use this expression to prove the proposition. In \textbf{Step 3}, we verify the square integrability of the aforementioned random variable.

\medskip

\noindent\textbf{Step 1.} In this step, we derive a formula for $\mathcal{H}(g,0;f,t)-\E[\mathcal{H}(g,0;f,t)]$ for any $f,g\in C_{+}([0,L])$. Let $(\mathcal{F}_s)_{s\ge 0}$ be the filtration generated by $\xi$ and $D$ denote the Malliavin derivative w.r.t.~$\xi$. For any square integrable random variable $F$ measurable w.r.t.~$\mathcal{F}_t$, the Clark-Ocone formula \cite{spa} states that
\begin{align*}
F-\E[F] = \int_0^t \int_0^L \E\left[D_{y,s}F\mid \mathcal{F}_s\right]\xi(y,s)dyds
\end{align*}
provided $D_{y,s}F$ is square integrable. We apply the above formula with $F:=\mathcal{H}(g,0;f,t)$ where $g,f\in C_+([0,L])$ to get
\begin{align*}
\mathcal{H}(g,0;f,t)-\E[\mathcal{H}(g,0;f,t)] = \int_0^t \int_0^L \E\left[D_{s,y}\mathcal{H}(g,0;f,t)\mid \mathcal{F}_s\right]\xi(y,s)dyds.
\end{align*}
We postpone the verification of square integrability of $D_{y,s}\mathcal{H}(g,0;f,t)$ to \textbf{Step 3}. Let us simplify the expression for the Malliavin derivative. 
Observe that
\begin{align}\nonumber
   D_{y,s}\mathcal{H}(g,0;f,t)  = D_{y,s}\log \mathcal{Z}(g,0;f,t)  = \frac{D_{y,s}\mathcal{Z}(g,0;f,t)}{\mathcal{Z}(g,0;f,t)}  & = \frac{\mathcal{Z}(g,0;y,s)\mathcal{Z}(y,s;f,t)}{\int_{0}^L \mathcal{Z}(g,0;y',s)\mathcal{Z}(y',s;f,t)dy'}\\ & =\frac{\rho(g,0;y,s)\mathcal{Z}(y,s;f,t)}{\int_{0}^L \rho(g,0;y',s)\mathcal{Z}(y',s;f,t)dy'} \label{mald}
\end{align}
where the second equality follows by known formula for the Malliavin derivative of the SHE \cite{mal} and the last equality follows from the definition of normalized partition function \eqref{npart}. 
Note that $\rho(g,0;\cdot,s)$ is measurable w.r.t.~$\mathcal{F}_s$ and $\mathcal{Z}(\cdot,s;\cdot,t)$ is independent of $\mathcal{F}_s$. Thus, if we define
\begin{align}
            \label{def:mg}
            \mathscr{G}_{s,t}(h,y;f):=\E\left[\frac{\mathcal{Z}(y,s;f,t)}{\mathcal{Z}(h,s;f,t)}\right]
        \end{align}
we have
\begin{align*}
    \E\left[D_{y,s}\mathcal{H}(g,0;f,t)\mid \mathcal{F}_s\right]= \mathscr{G}_{s,t}(\rho(g;\cdot,s),y;f)\cdot \rho(g,0;y,s).
\end{align*}
This leads us to the following formula:
\begin{align}\label{step1fomr}
            \mathcal{H}(g,0;f,t)-\E[\mathcal{H}(g,0;f,t)]=\int_0^t \int_0^L \mathscr{G}_{s,t}(\rho(g,0;\cdot,s),y;f)\rho(g,0;y,s)\xi(y,s)dyds.
        \end{align}

\medskip

\noindent\textbf{Step 2.} If we take $f=e^{\Lambda_2}$ in \eqref{def:mg} and then take expectation w.r.t.~$\Lambda_2$. we obtain

\begin{align*}
\Ex_{\Lambda_2}\E\left[\mathscr{G}_{s,t}(h,y;f)\right] & = \Ex_{\Lambda_2}\E\left[\frac{\mathcal{Z}(y,s;e^{\Lambda_2},t)/\mathcal{Z}(0,s;e^{\Lambda_2},t)}{\mathcal{Z}(h,s;e^{\Lambda_2},t)/\mathcal{Z}(0,s;e^{\Lambda_2},t)}\right] \\ & = \Ex_{\Lambda_2}\E\left[\frac{\mathcal{Z}(e^{\Lambda_2},s;y,t)/\mathcal{Z}(e^{\Lambda_2},s;0,t)}{\mathcal{Z}(e^{\Lambda_2},s;h,t)/\mathcal{Z}(e^{\Lambda_2},s;0,t)}\right] = \Ex\left[\frac{e^{\Lambda_2(y)}}{\int_0^L e^{\Lambda_2(x)}h(x)dx}\right].
\end{align*}
The first equality above is by definition. The second equality is via the swap invariance stated in \eqref{swinv}. The third equality follows via stationarity: $\mathcal{Z}(e^{\Lambda_2},s;\cdot,t)/\mathcal{Z}(e^{\Lambda_2},s;0,t) \stackrel{d}{=} e^{\Lambda_2(\cdot)}$.
Thus in view of \eqref{fomr} and \eqref{step1fomr} we have
\begin{align*}
   I_L(t)\stackrel{d}{=}\int_0^t \int_0^L \Ex\left[\frac{e^{\Lambda_2(y)}}{\int_0^L e^{\Lambda_2(x)}\rho(e^{\Lambda_1};x,s)dx}\right]\rho(e^{\Lambda_1},0;y,s)\xi(y,s)dyds.
\end{align*}
By Ito's isometry we thus have
\begin{align}\label{varfomr}
    \operatorname{Var}(I_L(t))=\Ex[I_L(t)^2]=\int_0^t\int_0^L \Ex\left[\left(\Ex_{\Lambda_2}\left[\frac{e^{\Lambda_2(y)}}{\int_0^L e^{\Lambda_2(x)}\rho(e^{\Lambda_1};x,s)dx}\right]\right)^2\rho^2(e^{\Lambda_1},0;y,s)\right]dyds.
\end{align}
Assuming $\Lambda_3$ to be another copy of the stationary measure we may write the above integrand as
\begin{align*}
    & \left(\Ex_{\Lambda_2}\left[\frac{e^{\Lambda_2(y)}}{\int_0^L e^{\Lambda_2(x)}\rho(e^{\Lambda_1};x,s)dx}\right]\right)^2\rho(e^{\Lambda_1},0;y,s)\\ & = \Ex_{\Lambda_2,\Lambda_3}\left[\frac{e^{\Lambda_2(y)+\Lambda_3(y)}}{\int_0^L e^{\Lambda_2(x)}\rho(e^{\Lambda_1};x,s)dx\int_0^L e^{\Lambda_3(x)}\rho(e^{\Lambda_1};x,s)dx}\right]\rho(e^{\Lambda_1},0;y,s) \\ & \stackrel{d}{=} \Ex_{\Lambda_2,\Lambda_3}\left[\frac{e^{\Lambda_2(y)+\Lambda_3(y)}}{\int_0^L e^{\Lambda_2(x)+\Lambda_1(x)}dx\int_0^L e^{\Lambda_3(x)+\Lambda_1(x)}dx}\right] \cdot e^{2\Lambda_1(y)}
\end{align*}
where the last equality in distribution follows via stationarity: $\rho(e^{\Lambda_1},0;\cdot,s) \stackrel{d}{=} \frac{e^{\Lambda_1(\cdot)}}{\int_0^L e^{\Lambda_1}(x)dx}$. Taking an expectation w.r.t.~$\Lambda_1$ in the above formula, in view of \eqref{varfomr}, yields Proposition \ref{proplead}.

\medskip

\noindent\textbf{Step 3.} In this step, we verify that $D_{y,s}\mathcal{H}(g,0;f,t)$ is indeed square integrable, i.e.,
\begin{align*}
    \int_0^t \int_0^L \Ex\left|D_{y,s}\mathcal{H}(g,0;f,t)\right|^2 dyds <\infty.
\end{align*}
The proof essentially relies on moment estimates for open SHE which we collect in Appendix \ref{appC}. Without loss of generality assume $\int_0^L f(x)dx=\int_0^L g(x)dx=1$. We shall show 
\begin{align}\label{tosg}
    \int_0^{t/2} \int_0^L \Ex\left|D_{y,s}\mathcal{H}(g,0;f,t)\right|^2 dyds <\infty.
\end{align}
The proof for the boundedness of the integral over $s\in [t/2,t]$ is analogous. 
From \eqref{mald} we have
\begin{align*}
    |D_{y,s}\mathcal{H}(g,0;f,t)|^2 & =\frac{\rho(g,0;y,s)^2\mathcal{Z}(y,s;f,t)^2}{\left(\int_0^L \rho(g,0;y',s)\mathcal{Z}(y',s;f,t)dy'\right)^2} \\ & \le \rho(g,0;y,s)^2\mathcal{Z}(y,s;f,t)^2\int_0^L \rho(g,0;y',s)\mathcal{Z}(y',s;f,t)^{-2}dy'
\end{align*}
where the above inequality follows from Jensen's inequality. As $\mathcal{Z}(\cdot,s;\cdot,t)$ and $\rho(g,0;\cdot,s)$ are independent, we get
\begin{align*}
    \Ex|D_{y,s}\mathcal{H}(g,0;f,t)|^2  \le \Ex\left[\rho(g,0;y,s)^2\int_0^L \rho(g,0;y',s) \cdot \Ex\left[\mathcal{Z}(y,s;f,t)^{2}\mathcal{Z}(y',s;f,t)^{-2}\right]dy'\right]
\end{align*}
Thanks to Proposition \ref{p:negmombd}, $\Ex\left[\mathcal{Z}(y,s;f,t)^{2}\mathcal{Z}(y',s;f,t)^{-2}\right]\le C$. As $\rho(g,0;\cdot,s)$ is a density, we arrive at
\begin{align}\label{dbd}
    \Ex|D_{y,s}\mathcal{H}(g,0;f,t)|^2  \le C\cdot \Ex\left[\rho(g,0;y,s)^2\right]
\end{align}
Recall the expression for $\rho(g,0;y,s)$ from \eqref{npart}. By Cauchy-Schwarz inequality:
\begin{align}\label{tsg2}
    \Ex[\rho(g,0;y,s)^2] =\Ex[\mathcal{Z}(g,0;y,s)^2\mathcal{Z}(g,0;\mathbf{1},s)^{-2}] \le \sqrt{\Ex[\mathcal{Z}(g,0;\mathbf{1},s)^{-4}]\Ex[\mathcal{Z}(g,0;y,s)^{4}]}.
\end{align}
By Proposition \ref{p.unbd}, we have
\begin{align*}
  \Ex[\mathcal{Z}(g,0;\mathbf{1},s)^{-4}] \le \int_0^L \Ex[\mathcal{Z}(x,0;\mathbf{1},s)^{-4}]g(x)dx  \le C,
\end{align*}
whereas Proposition \ref{p:mombd} {and Jensen's} inequality lead to
\begin{align*}
    \sqrt{\Ex[\mathcal{Z}(g,0;y,s)^{4}]} \le \left(\int_0^L \Ex[\mathcal{Z}(x,0;y,s)^4]^{1/4} g(x)dx\right)^2 & \le \int_0^L \sqrt{\Ex[\mathcal{Z}(x,0;y,s)^4]} g(x)dx \\ & \le C\int_0^L  s^{-1}e^{-|x-y|/\sqrt{s}} g(x)dx.
\end{align*}
Plugging the above two bounds back in \eqref{tsg2}, in view of \eqref{dbd}, yields
\begin{align*}
    \int_0^{t/2} \int_0^L \Ex\left|D_{y,s}\mathcal{H}(g,0;f,t)\right|^2 dyds \le C\int_0^{t/2} \int_0^L \int_0^L s^{-1} e^{-|x-y|/\sqrt{s}}g(x)dx dyds 
\end{align*}
which can be seen to be finite by performing the integration over $y$ first and then over $s$ followed by $x$. This verifies \eqref{tosg}.
\subsection{Proof of Theorem \ref{t.main1}} Given Theorem \ref{varde}, the proof of Theorem \ref{t.main1} follows by computing the asymtotics of $\sigma_L^2$ as $L\to \infty$. Recall $\Pf(y)$ from \eqref{def:pfy}. The following two proposition gives uniform upper and lower bounds for each $\Pf(y)$.
	
	\begin{proposition}\label{p.lowbd} There exists $c>0$ such that for all $L\ge 2$,
		$$\Pf(y) \ge cL^{-3/2}\mbox{ for all }y\in [L/4,3L/4].$$ 
	\end{proposition}

\begin{proposition}\label{p.upbd} There exists $C>0$ such that for all $L \ge 2$
    \begin{align}\label{part11}
		\Pf(y) & \le Cy^{-3/4}(L-y)^{-3/4}\mbox{ for all }y\in [L^{1/4},L-L^{1/4}-1], \\ \label{part12}
		\Pf(y) &  \le CL^{-3/4}\mbox{ for all }y\in [0,L^{1/4}]\cup [L-L^{1/4},L-1].
	\end{align}
\end{proposition}

Proposition \ref{p.lowbd} and \ref{p.upbd} are the core technical parts of the paper and are proven in Sections \ref{sec:4} and \ref{sec:5} respectively. Let us prove Theorem \ref{t.main1} assuming them.

\begin{proof}[Proof of Theorem \ref{t.main1}] Thanks to Proposition \ref{p.lowbd} and \eqref{sigmaLexp} we have
$\sigma_L^2 =\sum_{y=0}^{L-1} \Pf(y) \ge \sum_{y=\lceil L/4 \rceil}^{\lfloor 3L/4 \rfloor} \mathcal{F}(y) \ge cL^{-1/2}/4.$
On the other hand, using \eqref{part11} and \eqref{part12} we have
\begin{align*}
    \sigma_L^2=\sum_{y=0}^{L-1}\Pf(y) & =\sum_{y\in [0,L^{1/4}]\cup [L-L^{1/4},L-1]\cap \Z}\Pf(y)+\sum_{y\in [L^{1/4},L-L^{1/4}-1]\cap \Z}\Pf(y) \\ & \le CL^{-1}+C\sum_{y=1}^{L-1} y^{-3/4}(L-y)^{3/4}  \le C'\left(1+\int_0^1x^{-3/4}(1-x)^{-3/4}dx\right)\le C''L^{-1/2}.
\end{align*}
Thus, there exists $C>0$ such that for all large enough $L$ we have $C^{-1}L^{-1/2} \le \sigma_L^2 \le CL^{-1/2}$. This verifies the first claim of the theorem.
When $L=\lambda t^{\alpha}$, using Theorem \ref{varde} we get that
\begin{align*}
    \sqrt{\operatorname{Var}(\mathcal{H}(0,t))} & \le C\sqrt{t}L^{-1/4}+C\sqrt{L}=C\lambda^{1/4}t^{1/2-\alpha/4}+C\sqrt{\lambda}t^{\alpha/2}, \\
    \sqrt{\operatorname{Var}(\mathcal{H}(0,t))} & \ge C^{-1}\sqrt{t}L^{-1/4}-C\sqrt{L} \ge C^{-1}\lambda^{1/4}t^{1/2-\alpha/4}-C\lambda^{1/2}t^{\alpha/2}.
\end{align*}
When $\alpha \in (0, 2/3)$, or when $\lambda \in (0, \delta)$ and $\alpha = 2/3$ for sufficiently small $\delta$, we have $C^{-1}\lambda^{1/4} t^{1/2 - \alpha/4} \ge 1/2C\lambda^{1/2}t^{\alpha/2}$. Theorem \ref{t.main1} then follows from the above matching upper and lower bounds (up to constants).
\end{proof}

	\section{Lower Bound for the asymptotic variance} \label{sec:4}

In this section, we prove Proposition \ref{p.lowbd}. The key idea, as explained in the introduction, is to exploit the stochastic monotonicity inherent in the Gibbsian line ensemble structure of the stationary measures. This allows us to reduce the problem to a computation involving Brownian bridges. In Section \ref{sec:lwbd1}, we detail the monotonicity argument, and in Section \ref{sec:lwbd2}, we carry out the computation in the Brownian bridge setting.
    
	\subsection{Local Brownianity in the bulk} \label{sec:lwbd1} We split the proof of Proposition \ref{p.lowbd} into four steps for clarity. We work under the assumptions  $u, v \ge 0$ with $u + v > 0$ in first three steps. In Step 4, we will explain  how the proof can be modified to handle the simpler case $u = v = 0$.

		\medskip
		
		\noindent\textbf{Step 1: Outer and Inner events.} In this step define a bunch of events that are necessary for our proof. Fix $M,\delta>0$ and define 
		\begin{equation}
			\label{outerevents}
			\begin{aligned}
				\m{Loc}_i(y) & := \big\{ |\Lambda_i(x)-\Lambda_i(y)| \le M \mbox{ for all }x\in [y,y+1] \big\}, \\ 
                \m{Loc}_i'(y) & := \big\{  |\Lambda_i'(x)-\Lambda_i'(y)| \le M \mbox{ for all }x\in [y,y+1] \big\} \\
				\m{Sep}_i(y) & :=\big\{\Lambda_i(y) \in [\sqrt{L},2\sqrt{L}], \ \Lambda_i'(y) \in [-2\sqrt{L},-\sqrt{L}]\big\}, \\
				\m{MLD}_i(y) & :=\Big\{\Lambda_i(y\pm L\delta)-\Lambda_i(y) \in [-2\sqrt{\delta L},-\sqrt{\delta L}]\big\}. \\ \m{MLD}_i'(y) & := \Big\{ \Lambda_i'(y\pm L\delta)-\Lambda_i'(y) \in [-2\sqrt{\delta L},2\sqrt{\delta L}] \Big\}, \\ 
				\m{NMax}_i(y) & :=\big\{\Lambda_i(x)-\Lambda_i(y) \le -\sqrt{\delta L}/2 \mbox{ for all }x\in [0,y-L\delta] \cup [y+L\delta,L]\big\}, \\
				\m{D}_{\m{out}}(y) & := \bigcap_{i=1}^3 \big(\m{Loc}_i(y)\cap \m{Loc}_i'(y) \m{Sep}_i(y) \cap \m{MLD}_i(y) \cap \m{MLD}_i'(y) \cap \m{NMax}_i(y)\big). 
			\end{aligned}
		\end{equation}
		
		Let us briefly describe what do the above events mean. In words, the event $\m{Loc}_i(y), \m{Loc}_i'(y)$ controls the process $\Lambda_i, \Lambda_i'$ on a \textit{local scale}, namely, on an interval of size $1$ around $y$. The event $\m{Sep}_i(y)$ fixes the location of $\Lambda_i(y)$ and $\Lambda_i'(y)$ on certain $O(\sqrt{L})$ intervals. Note that the intervals are so chosen that it ensures an $O(\sqrt{L})$ \textit{separation} between $\Lambda_i(y)$ $\Lambda_i'(y)$. This is crucial and hence derives the name of the event. The $\m{MLD}_i(y), \m{MLD}_i'(y)$ events are \textit{macroscopic local difference} events that controls the difference of $\Lambda_i,\Lambda_i'$ on two points that are separated by a small window of macroscopic length (i.e., of order $O(L)$).  Finally the event $\m{NMax}_i(y)$ (\textit{near-max}) essentially states that $\Lambda_i(y)$ is close to the maximum of $\Lambda_i$ in a certain quantitative sense.
		Let us define the sigma field:
		\begin{align*}
			\mathcal{F}(y):=\sigma\big\{(\Lambda_i(x),\Lambda_i'(x)) \mid x\in [0,y-L\delta]\cup [y,y+1] \cup [y+L\delta,L], i=1,2,3\big\}.
		\end{align*}
		
		\begin{figure}[h!]
			\centering
			\begin{overpic}[width=10cm]{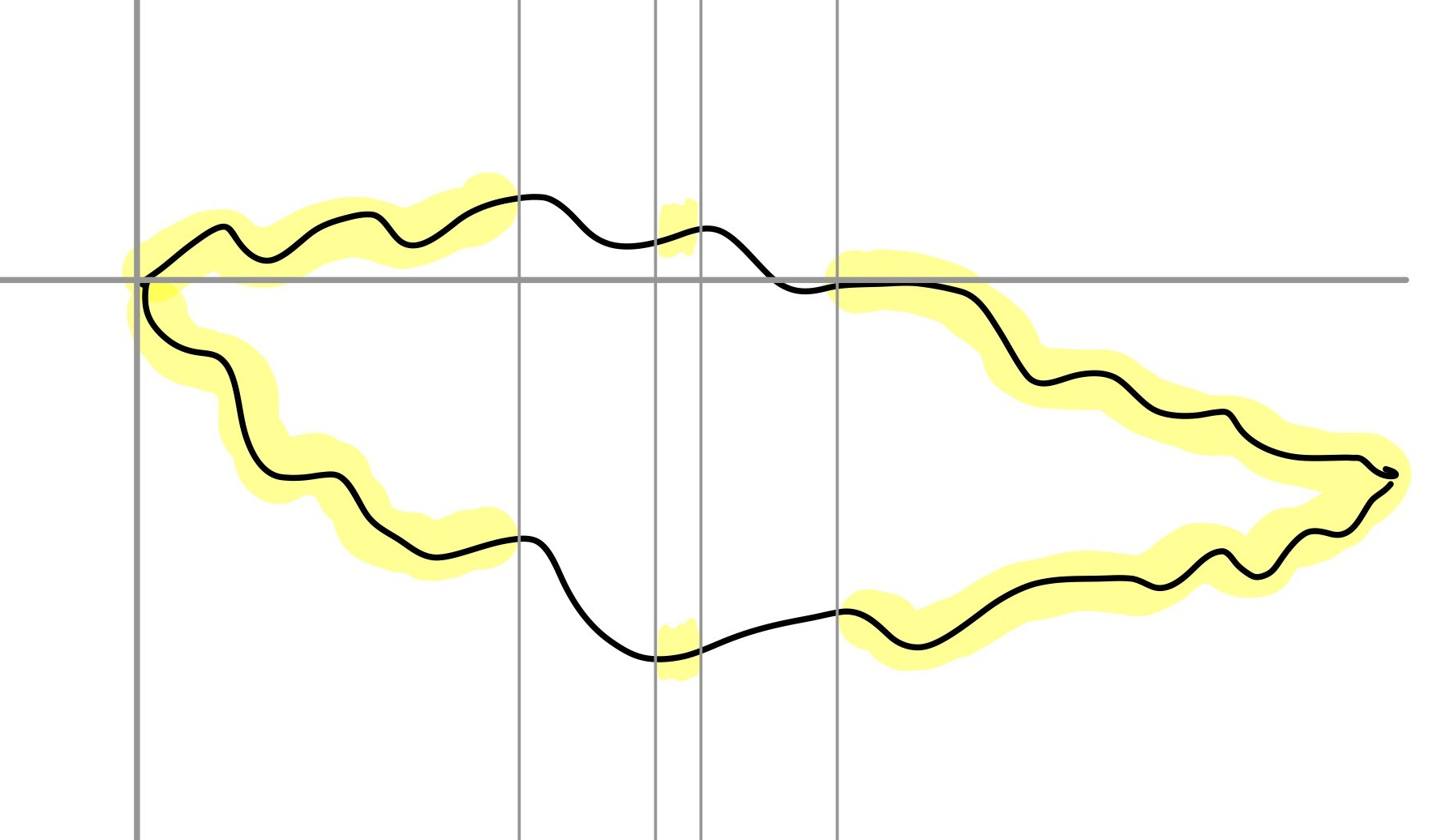}
            \put(48,0){\scriptsize{$y+1$}}
            \put(45,0){\scriptsize{$y$}}
            \put(60,0){\scriptsize{$y+L\delta$}}
            \put(50,45){\scriptsize{$\Lambda_i(\cdot)$}}
            \put(60,20){\scriptsize{$\Lambda_i'(\cdot)$}}
            \put(30,0){\scriptsize{$y-L\delta$}}
			\end{overpic}
			\caption{In the above figure, the yellow highlighted parts indicate which parts of the processes $\Lambda_i, \Lambda_i'$ are measurable w.r.t.~$\mathcal{F}(y)$.}
			\label{fig:enter-label}
		\end{figure}
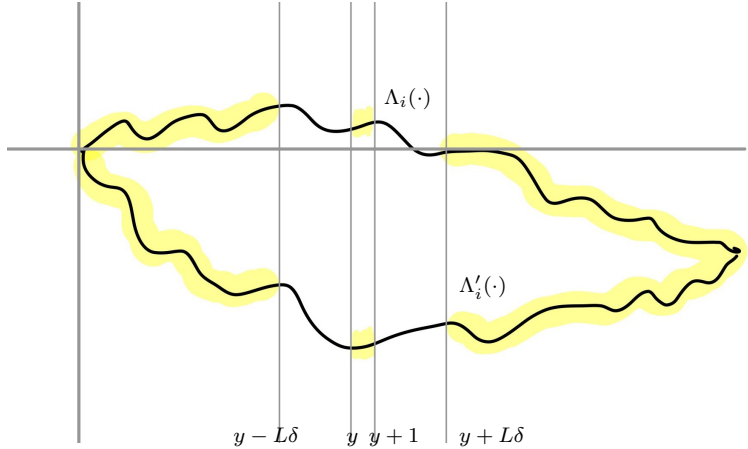
	We refer to Figure \ref{fig:enter-label} for a visualization of $\mathcal{F}(y)$. Clearly the event $\m{D}_{\m{out}}(y) \in \mathcal{F}(y)$. The following lemma studies the probabilities of the above events. Its proof is deferred to Section \ref{sec:lwbd2}.
		
		\begin{lemma}\label{dout} For each $\delta>0$, there exists $c_\delta>0$ and $M=M(\delta)>0$ such that $\Pr\big(\m{D}_{\m{out}}(y)\big) \ge c_{\delta}.$   
		\end{lemma}
		We fix the above choice of $M$ (depending on $\delta$) for the rest of the proof. 
		
		\smallskip

		We now define a bunch of `inner' events (`inner' in the sense that they will be measurable w.r.t.~the processes restricted to $[y-L\delta,y]\cup[y+1,y+L\delta]$). First we define a cut-off function. Let $\chi:\R\rightarrow \R$ be a smooth function such that $\chi|_{[-1,1]}=0$ and $\chi|_{[-2,2]^c}=1$. For every $r>0$, we consider the function $q_r : [0,\infty)\to \R$ defined via 
		\begin{align}
			\label{cutoff}
			q_r(y)=\frac{1}{3}\int_0^y\chi(x)\chi\left( \frac{4}{r+2}x+\frac{2-3r}{r+2}\right)x^{-2/3}dx.
		\end{align}
		We notice that $q_r$ is smooth, $q_r$ is of order at $|y|^{1/3}$ and $q'_r(y)=\frac{1}{3}\chi(y)\chi\left( \frac{4}{r+2}y+\frac{2-3r}{r+2}\right)y^{-2/3} $, so $q_r'(y)=0$ when $y\in [0,1]\cap [\frac{r}{2}-1,r]$. A plot of the $q_r|_{y\ge 1}$ function is given in Figure \ref{fig:enter-label2}.

		We set $q:=q_{L\delta}$ and define the following events:
		\begin{equation}
			\label{innerevents}
			\begin{aligned}
				\m{A}_{+}(y) & :=\Big\{\Lambda_1(x)+\Lambda_2(x)-\Lambda_1(y+1)-\Lambda_2(y+1) \le 1-q(|x-y-1|) \mbox{ and } \\ &  \hspace{2cm}\Lambda_2(x)+\Lambda_3(x)-\Lambda_2(y+1)-\Lambda_3(y+1) \le 1-q(|x-y-1|) \\ & \hspace{8cm}\mbox{ for all } x\in [y+1,y+L\delta] \Big\}, \\
				\m{A}_{-}(y) & :=\Big\{\Lambda_1(x)+\Lambda_2(x)-\Lambda_1(y)-\Lambda_2(y) \le 1-q(|x-y|) \mbox{ and } \\ &  \hspace{1.5cm}\Lambda_2(x)+\Lambda_3(x)-\Lambda_2(y)-\Lambda_3(y) \le 1-q(|x-y|) \mbox{ for all } x\in [y- L\delta,y] \Big\},  \\
				\m{B}_{+}(y) & := \big\{\Lambda_i(x)-\Lambda_i'(x) \ge \sqrt{L} \mbox{ for all } x\in [y+1,y+L\delta]\mbox{ and } i=1,2,3\big\}, \\
				\m{B}_{-}(y) & := \big\{\Lambda_i(x)-\Lambda_i'(x) \ge \sqrt{L} \mbox{ for all } x\in [y-L\delta,y] \mbox{ and } i=1,2,3\big\}, \\  \m{D}_{\m{in}}(y) & := \m{A}_-(y)\cap\m{A}_+(y)\cap \m{B}_-(y) \cap \m{B}_+(y). 
			\end{aligned}
		\end{equation}

		The importance of these events is that on the event $\m{D}_{\m{out}}(y) \cap \m{D}_{\m{in}}(y)$, the random variable appearing in \eqref{def:pfy} is bounded below uniformly. We will prove this in the next step. 
		
		\medskip
		
		\noindent\textbf{Step 2: Reduction to Brownian bridges.} We claim that there exists a deterministic constant $\til{c}_\delta>0$ such that 
		\begin{align}\label{e.claim1}
			\frac{\int\limits_y^{y+1} e^{\Lambda_1(x)+2\Lambda_2(x)+\Lambda_3(x)-\Lambda_1(y)-2\Lambda_2(y)-\Lambda_3(y)}dx}{\int\limits_0^L e^{\Lambda_1(x)+\Lambda_2(x)-\Lambda_1(y)-\Lambda_2(y)}dx\int\limits_0^L e^{\Lambda_2(x)+\Lambda_3(x)-\Lambda_1(y)-\Lambda_2(y)}dx} \ge \til{c}_{\delta} \cdot \ind_{\m{D}_{\m{out}}(y) \cap \m{D}_{\m{in}}(y)}
		\end{align}
		holds almost surely for all $y\in [L/4,3L/4]$ and $L$ large enough.  This follows easily by noticing that
		\begin{itemize}
			\item The $\bigcap_{i=1}^3 \m{Loc}_i(y)$ event controls the integrals on the range $[y,y+1]$.
			\item The $\m{A}_+(y)\cap \m{A}_-(y)$ event controls the integrals on the range $[y-L\delta,y]$ and $[y+1,y+L\delta]$.
			\item The $\bigcap_{i=1}^3 (\m{NMax}_i(y)\cap \m{Loc}_i(y))$ event controls the integrals on the range $[0,y-L\delta]$, $[y+L\delta,L]$.
		\end{itemize}
		Given \eqref{e.claim1}, we thus have
		\begin{align}\label{lwbd10}
			\mathfrak{F}(y) \ge \til{c}_\delta \cdot \Ex\big[\ind_{\m{D}_{\m{out}}(y) \cap \m{D}_{\m{in}}(y)}\big] = \til{c}_\delta \Ex\big[\ind_{\m{D}_{\m{out}}(y)} \Ex[\ind_{\m{D}_{\m{in}}(y)}\mid \mathcal{F}(y)\big]\big].
		\end{align}
		Note that the law of $(\Lambda_i(x),\Lambda_i'(x))_{x\in [y-L\delta,y]\cup [y+1,y+L\delta]}$ conditioned on $\mathcal{F}(y)$ are absolutely continuous w.r.t.~two independent copies of pair of Brownian bridges (with endpoints that are in $\mathcal{F}(y)$, i.e., they are random) with a Radon-Nikodym derivative of the form:
		\begin{align*}
			W_{i,+} & :=\exp\bigg(-\int_{y+1}^{y+L\delta} e^{-(\Lambda_i(s)-\Lambda_i'(s))}ds\bigg), \quad 
			W_{i,-}  :=\exp\bigg(-\int_{y-L\delta}^{y} e^{-(\Lambda_i(s)-\Lambda_i'(s))}ds\bigg).
		\end{align*}
		More precisely, we have
		\begin{align}\label{lwbd11}
			\Ex[\ind_{\m{D}_{\m{in}}(y)}\mid \mathcal{F}(y)\big] = \frac{\Ex_{B}\bigg[\ind_{\m{D}_{\m{in}}(y)}\prod\limits_{i=1}^3 W_{i,+}W_{i,-}\bigg]}{\Ex_{B}\bigg[\prod\limits_{i=1}^3 W_{i,+}W_{i,-} \bigg]} \ge \Ex_{B}\bigg[\ind_{\m{D}_{\m{in}}(y)}\prod\limits_{i=1}^3 W_{i,+}W_{i,-}\bigg],
		\end{align}
		where we write $\Ex_{B}$ to denote that the expectation is computed under Brownian bridge laws. The above inequality follows from the fact that the $W_{i,\pm}$s are all less than $1$. Recall the $\m{B}_{+}(y)$ and $\m{B}_-(y)$ events from \eqref{innerevents} ($\m{D}_{\m{in}}(y)$ is contained in both of them). Note that under $\m{B}_-(y)\cap \m{B}_+(y)$, $W_{i,\pm} \ge \gamma_{\delta}$ for some $\gamma_{\delta}>0$ for all $L\ge 1$. Thus, 
		\begin{equation}
			\label{lwbd12}
			\begin{aligned}
				\Ex_{B}\bigg[\ind_{\m{D}_{\m{in}}(y)}\prod\limits_{i=1}^3 W_{i,+}W_{i,-}\bigg] &  \ge \gamma_{\delta}^6 \cdot \Pr_{B}\big(\m{D}_{\m{in}}(y)\big) \\ & \ge \gamma_\delta^6 \cdot \Pr_{B}\big(\m{A}_+(y)\cap \m{A}_-(y)\big) \Pr_B\big(\m{B}_+(y) \cap \m{B}_-(y) \mid \m{A}_+(y)\cap \m{A}_-(y)).
			\end{aligned}
		\end{equation}
		
		Since the conditional processes on $[y-L\delta,y]$ and $[y+1,y+L\delta]$ are independent, we have $\Pr_{B}\big(\m{A}_+(y)\cap \m{A}_-(y)\big)=\Pr_{B}\big(\m{A}_+(y)\big)\Pr_{B}\big(\m{A}_-(y)\big)$. We have the following estimate on $\Pr_{B}(\m{A}_\pm(y))$.
		
		\begin{lemma}\label{lw34bd} Recall the $\m{MLD}_i(y)$ event from \eqref{outerevents}. There exists $c_{\delta}'>0$ such that for all $y\in [L/4,3L/4]$ we have
			$$\ind_{\cap_{i=1}^3\m{MLD}_i(y)}\cdot \Pr_{B}(\m{A}_\pm(y)) \ge c_{\delta}'\cdot \ind_{\cap_{i=1}^3\m{MLD}_i(y)} \cdot L^{-3/4}.$$
		\end{lemma}
		The proof of the above lemma follows from exact calculations of Brownian motion probability and is deferred to next subsection. 
		Similarly we also have $\Pr_B\big(\m{B}_+(y) \cap \m{B}_-(y) \mid \m{A}_+(y)\cap \m{A}_-(y)\big)=\Pr_B\big(\m{B}_+(y) \mid \m{A}_+(y)\big)\Pr_B\big( \m{B}_-(y) \mid \m{A}_-(y)\big)$. We claim that there exists $\delta>0$ small enough such that
		\begin{align}\label{unisto}
			\ind_{\m{D}_{\m{out}}(y)} \cdot \Pr_B\big(\m{B}_\pm(y) \mid \m{A}_\pm(y)\big) \ge \frac12 \cdot \ind_{\m{D}_{\m{out}}(y)}.
		\end{align}
		We shall prove \eqref{unisto} via a stochastic monotonicity argument in the next step. Note that combining \eqref{lwbd11}, \eqref{lwbd12}, Lemma \ref{lw34bd}, and the claim in \eqref{unisto}, we get that there exists $c_{\delta}''>0$ such that for all $y\in [L/4,3L/4]$ and $L$ large enough the following deterministic bound
		\begin{align*}
			\ind_{\m{D}_{\m{out}}(y)} \Ex[\ind_{\m{D}_{\m{in}}(y)}\mid \mathcal{F}(y)\big] \ge c_{\delta}'' \cdot \ind_{\m{D}_{\m{out}}(y)} \cdot L^{-3/2}.
		\end{align*}
		holds for some $c_{\delta}''>0$. Plugging this back in \eqref{lwbd10} and in view of Lemma \ref{dout} we have $\Pf(y) \ge c_{\delta}c_{\delta}'c_{\delta}''L^{-3/2}$. This completes the proof of the proposition modulo \eqref{unisto} (and Lemmas \ref{dout} and \ref{lw34bd}).
		
		\medskip
		
		\noindent\textbf{Step 3: Stochastic monotonicity argument.} 
		In this step we prove \eqref{unisto}. We shall demonstrate the bound only for the quantity $\Pr_B\big(\m{B}_+(y) \mid \m{A}_+(y)\big)$, the other part being analogous. Set $\overline{\Lambda_i}(x):=\Lambda_i(x)- \Lambda_i(y+1)$ and $\overline{\Lambda_i'}(x):=\Lambda_i'(x)- \Lambda_i'(y+1)$. and define events
		\begin{align*}
			\m{Rise}_{i,+}(y) & := \big\{\overline{\Lambda_i}(x) \le \sqrt{L}/8 \mbox{ for all } x\in [y+1,y-L\delta]\big\}, \\ \m{Rise}_{i,+}'(y) & := \big\{\overline{\Lambda_i'}(x) \le \sqrt{L}/8 \mbox{ for all } x\in [y+1,y-L\delta]\big\}, \\
			\m{Fall}_{i,+}(y) & := \big\{\overline{\Lambda_i}(x) \ge -\sqrt{L}/2 \mbox{ for all } x\in [y+1,y-L\delta]\big\}.
		\end{align*}
		Note that
		\begin{align*}
			\m{D}_{\m{out}}(y) \cap \m{B}_+(y) \supset \m{D}_{\m{out}}(y)\cap \bigcap_{i=1}^3 \m{Rise}_{i,+}'(y) \cap \m{Fall}_{i,+}'(y)
		\end{align*}
		which implies
		\begin{align}\label{dinineq}
			\ind_{\m{D}_{\m{out}}(y)} \Pr_B\big(\m{B}_+(y) \mid \m{A}_+(y)) \ge \ind_{\m{D}_{\m{out}}(y)} \Pr_B\big(\bigcap_{i=1}^3 \m{Rise}_{i,+}'(y) \cap \m{Fall}_{i,+}(y) \mid \m{A}_+(y)).
		\end{align}
		So, if we can control $\Pr(\m{Rise}_{i,+}'(y) \mid \m{A}_+(y))$ and $\Pr(\m{Fall}_{i,+}(y) \mid \m{A}_+(y))$, we are done. We shall demonstrate this control using stochastic monotonicity. 
		Fix an $\e>0$. We claim that
		we can choose $\delta(\e)>0$ such that
		\begin{align}\label{risebd}
			\Pr(\m{Rise}_{i,+}(y) \mid \m{A}_+(y)) \ge 1-\e, \quad \Pr(\m{Rise}_{i,+}'(y) \mid \m{A}_+(y)) \ge 1-\e.
		\end{align}
		We shall prove this claim only for the $\m{Rise}_{1,+}(y)$ event. Rest cases are analogous. Let us define the $\sigma$-field
		\begin{align*}
			\mathcal{G}:=\sigma\{\mathcal{F}(y) \cup \sigma\big\{(\Lambda_i(x),\Lambda_i'(x)) \mid x\in [0,L], i=2,3\big\}\big\}
		\end{align*}
		Conditioned on $\mathcal{G}$, under $\mathbb{P}_B$ law, $\Lambda_1(\cdot)$ is still a Brownian bridge on $[y+1,y+L\delta]$ from $\Lambda_1(y+1)$ to $\Lambda_1(y+L\delta)$. Let us denote this conditional law as $\mathbb{P}_{B|\mathcal{G}}$. Note that $\m{Rise}_{1,+}(y)$ and $\m{A}_+(y)$ are measurable events w.r.t.~ $\sigma\{\Lambda_1(x) \mid x\in [y+1,y+L\delta]\}\cup \mathcal{G}$. Let us set
		\begin{align*}
			\til{\m{A}}_+(y):=\bigcap_{x\in [y+1,y+L\delta]}\big\{\overline{\Lambda_1}(x) \le 1-q(|x-y-1|)-\overline{\Lambda_2}(x)\big\}.
		\end{align*}
		$\overline{\Lambda_2}$ are measurable w.r.t.~$\mathcal{G}$. Thus the above event can be interpreted as an event which requires $\overline{\Lambda_1}(x)$ to be less than a given barrier $f(x):=1-q(|x-y-1|)-\overline{\Lambda_2}(x)$. Thus
		\begin{align*}
			\Pr_{B|\mathcal{G}}\big(\m{Rise}_{1,+}(y) \mid \til{\m{A}}_+(y)\big) = \Pr_{B|\mathcal{G}}^f\big(\m{Rise}_{1,+}(y)\big)
		\end{align*}
		where $\Pr_{B|\mathcal{G}}^f$ is the law of the Brownian bridge $B$ conditioned to be less than the barrier $f$. Now $\m{Rise}_{1,+}(y)$ event is decreasing as we increase the boundaries. So by stochastic monotonicity, taking $f\to \infty$ pointwise we get
		\begin{align*}
			\Pr_{B|\mathcal{G}}^f\big(\m{Rise}_{1,+}(y)\big) \ge \Pr_{B|\mathcal{G}}\big(\m{Rise}_{1,+}(y)\big).
		\end{align*}
		The latter can be assumed to be larger than $1-\e$ by taking $\delta$ small enough depending only on $\e>0$. Thus,
		\begin{align*}
			& \ind_{\m{E}}{\Ex_B\big[\ind_{\m{Rise}_{1,+}(y) \cap \til{\m{A}}_+(y)} \mid \mathcal{G}\big]}  \ge (1-\e)\ind_{\m{E}}{\Ex_B\big[\ind_{\til{\m{A}}_+(y)} \mid \mathcal{G}\big]}.
		\end{align*}
		where $\m{E}:=\bigcap_{x\in [y+1,y+L\delta]}\{\overline{\Lambda_2}(x) \le 1-q(|x-y-1|)-\overline{\Lambda_3}(x)\}$. Taking expectation w.r.t.~$\Lambda_2, \Lambda_3$ on both sides, we get that
		$$\Pr_B\big(\m{Rise}_{1,+}(y) \cap \m{A}_+(y)\big) \ge (1-\e)\Pr_B\big(\m{A}_+(y)\big)$$
		which is equivalent to $\Pr_B\big(\m{Rise}_{1,+}(y) \mid \m{A}_+(y)\big) \ge 1-\e$. This verifies \eqref{risebd}.
		
		Next we claim that by taking $\delta$ a bit smaller we can ensure
		\begin{align*}
			\Pr(\m{Fall}_{i,+}(y) \mid \m{A}_+(y)) \ge 1-3\e.
		\end{align*}
		We shall again prove this only for $i=1$. In fact we shall prove something stronger:
		\begin{align}\label{fallbd}
			\Pr_B\big(\m{Rise}_{2,+}(y) \cap \m{Rise}_{3,+}(y)\cap \m{Fall}_{1,+}(y) \mid \m{A}_+(y) \big) \ge 1-3\e.
		\end{align}
		holds for all small enough $\delta$. Continuing with the notations of the previous claim, let us consider
		\begin{align*}
			\ind_{ \m{Rise}_{2,+}(y) \cap \m{Rise}_{3,+}(y)} \cdot \Pr_{B|\mathcal{G}}^f\big(\m{Fall}_{1,+}(y)\big).
		\end{align*}
		The $\m{Fall}_{1,+}(y)$ is decreasing as we decrease the boundaries. Note that in presence of $\m{Rise}_{2,+}(y)$ event, we have $f(x) \ge R:=1-L^{1/3}-\sqrt{L}/8$. We decrease the barrier to $R$ and also decrease the boundaries of the Brownian bridge: from $\Lambda_1(y+1)$ to $\Lambda_1(y+1)-\sqrt{L}/4$ and from $\Lambda_1(y+L\delta)$ to $\Lambda_1(y+L\delta)-\sqrt{L}/4$ to get a new law which we denote by $\til{\Pr}_{B|\mathcal{G}}^R$. Let us denote the Brownian bridge under $\til{\Pr}_{B|\mathcal{G}}^R$ as $\hat\Lambda_1(\cdot)$. Note that under $\til{\Pr}_{B|\mathcal{G}}^R$ both $\m{Fall}_{1,+}(y)$ and $\sup_{x\in [y+1,y+L\delta]} \{\hat{\Lambda}_1(x)-\Lambda_1(y+1) \le R \}$ events are likely. Thus we have
		\begin{align*}
			&  \ind_{ \m{Rise}_{2,+}(y) \cap \m{Rise}_{3,+}(y)} \Pr_{B|\mathcal{G}}^f\big(\m{Fall}_{1,+}(y)\big) \\ & \ge \ind_{ \m{Rise}_{2,+}(y) \cap \m{Rise}_{3,+}(y)} \til{\Pr}_{B|\mathcal{G}}^R\big(\m{Fall}_{1,+}(y)\big) \\ & \ge \  \ind_{ \m{Rise}_{2,+}(y) \cap \m{Rise}_{3,+}(y)} \til{\Pr}_{B|\mathcal{G}}\bigg(\m{Fall}_{1,+}(y) \cap \sup_{x\in [y+1,y+L\delta]} \left\{\hat{\Lambda}_1(x)-\Lambda_1(y+1) \le R \right\}\bigg)  \\ & \ge (1-\e)  \ind_{ \m{Rise}_{2,+}(y) \cap \m{Rise}_{3,+}(y)}.
		\end{align*}
		where the last inequality follows by taking $\delta$ small enough. Multiplying by $\ind_{\m{E}}$ and then taking expectation w.r.t.~ $\Lambda_2, \Lambda_3$ we get
		\begin{align*}
			& \Pr_B\big(\m{Rise}_{2,+}(y) \cap \m{Rise}_{3,+}(y)\cap \m{Fall}_{1,+}(y) \cap \m{A}_+(y) \mid  \mathcal{F}(y) \big) \\ &  \ge (1-\e)\Pr_B\big(\m{Rise}_{2,+}(y) \cap \m{Rise}_{3,+}(y) \cap \m{A}_+(y) \mid \mathcal{F}(y) \big) \ge (1-3\e)\Pr_B\big(\m{A}_+(y) \mid \mathcal{F}(y) \big).
		\end{align*}
		where the last inequality follow from \eqref{risebd}. This proves \eqref{fallbd}. Thus by union bound we have
		\begin{align*}
			\Pr_B\big(\bigcap_{i=1}^3 \m{Rise}_{i,+}'(y) \cap \m{Fall}_{i,+}'(y) \mid \m{A}_+(y)\big) \ge 1-12\e.
		\end{align*}
		Taking $\e=\frac1{24}$ fixes our choice of $\delta$ and in view of \eqref{dinineq} completes the proof of \eqref{unisto}.

	\medskip

    \noindent\textbf{Step 4. $u=v=0$ case.} Let us explain the modification needed for $u=v=0$. First note that the following slightly stronger version of \eqref{e.claim1} holds: 
    \begin{align*}
			& \frac{\int\limits_y^{y+1} e^{\Lambda_1(x)+2\Lambda_2(x)+\Lambda_3(x)-\Lambda_1(y)-2\Lambda_2(y)-\Lambda_3(y)}dx}{\int\limits_0^L e^{\Lambda_1(x)+\Lambda_2(x)-\Lambda_1(y)-\Lambda_2(y)}dx\int\limits_0^L e^{\Lambda_2(x)+\Lambda_3(x)-\Lambda_1(y)-\Lambda_2(y)}dx} \\ & \hspace{3cm}\ge \til{c}_{\delta} \cdot \ind_{\cap_{i=1}^3(\m{Loc}_i(y)\cap \m{MLD}_i(y)\cap \m{NMax}_i(y))\cap \m{A}_+(y)\cap \m{A}_-(y)}.
		\end{align*}
	Indeed the events on the right hand side of the above equation are include the events that are used to control the integrals appearing on the left hand side (as noted after \eqref{e.claim1}). Taking expectations, this leads to the following analog of \eqref{lwbd10}: \begin{align}\label{lwbd10a}
			\mathfrak{F}(y) \ge  \til{c}_\delta \Ex\big[\ind_{\cap_{i=1}^3(\m{Loc}_i(y)\cap \m{MLD}_i(y)\cap \m{NMax}_i(y))} \Ex[\ind_{\m{A}_+(y)\cap \m{A}_-(y)}\mid \mathcal{F}(y)\big]\big].
		\end{align} 
	When $u=v=0$, the stationary measure is simply a Brownian motion. Thus, the law of $(\Lambda_i(x))_{x\in [y-L\delta,y]\cup [y+1,y+L\delta]}$ conditioned on $\mathcal{F}(y)$ is just Brownian bridges (with random endpoints). Hence, $$\Ex[\ind_{\m{A}_+(y)\cap \m{A}_-(y)}\mid \mathcal{F}(y)\big]=\Pr_B(\m{A}_+(y)\cap \m{A}_-(y))=\Pr_{B}(\m{A}_+(y))\Pr_B(\m{A}_-(y)).$$
    Using Lemma \ref{lw34bd} we thus obtain
    \begin{align}
        \mathfrak{F}(y) \ge  \til{c}_\delta \cdot c_{\delta}' \cdot L^{-3/2} \cdot \prod_{i=1}^3\Pr\bigg(\m{Loc}_i(y)\cap \m{MLD}_i(y)\cap \m{NMax}_i(y)\bigg).
    \end{align}
	Under Brownian motion law, it is clear that the above probabilities are bounded from below. This completes the proof.

	\subsection{Proof of Results from Section \ref{sec:lwbd1}} \label{sec:lwbd2}
	
	\subsubsection{Proof of Lemma \ref{dout}} Recall all the outer events from \eqref{outerevents}. By Proposition \ref{tailest} (\eqref{131} and \eqref{1310} specifically) we have that $\m{Loc}_i(y)$ and $\m{Loc}_i'(y)$ are high probability events (as $M\to \infty$). Thanks to Lemma \ref{l.difflimits} and Fatou's Lemma, as $L\to \infty$
	
	\begin{align*}
		\liminf_{L\to \infty}  \Pr\big(\m{Sep}_i(y) \cap \m{MLD}_i(y) \cap \m{MLD}_i'(y)\cap \m{NMax}_i(y)\big) \ge \inf_{u\in [1/4,3/4]} f(u) 
	\end{align*}
	where 
	\begin{align*}
		f(u) & :=\Pr\bigg(B(u) \in [1,2] , B'(u) \in [-1,-2], B(u-\delta)-B(u) \in [-2\sqrt{\delta}, -\sqrt{\delta}], \\ & \hspace{2cm} B'(u\pm \delta)-B'(u) \in [-2\sqrt{\delta}, -\sqrt{\delta}], \\ & \hspace{2cm} B(x)-B(u) \le -\sqrt{\delta}/2 \mbox{ for all } x\in [0,u-\delta]\cup [u+\delta,1]\bigg).
	\end{align*}
	Here $(B,B')$ are one of the laws described in Lemma \ref{l.difflimits}. By properties of Brownian motion $\inf_{u\in [1/4,3/4]} f(u) \ge c_{\delta}$ for some $c_{\delta}>0$.

\medskip

	\subsubsection{Proof of Lemma \ref{lw34bd}}\label{sec:4.2.2} We now turn towards proving Lemma \ref{lw34bd}. We continue with the same notations as in Section \ref{sec:lwbd1}. We shall prove the bound for $\Pr_B(\m{A}_-(y))$ only, the other one being analogous. Let us focus on the $[y-L\delta,y]$ interval. Note that under the law $\Pr_B$, on $[y-L\delta,y]$ $\Lambda_i$s are independent Brownian motions from $\Lambda_i(y-L\delta)$ to $\Lambda_i(y)$. Recall the $\m{MLD}_i(y)$ event from \eqref{outerevents}. Under this event, we know that $\Lambda_i(y-L\delta)-\Lambda_i(y)\in [-2\sqrt{\delta L},-\sqrt{\delta L}]$. Hence after translation and scaling, it suffices to prove the following result:

	\begin{proposition}\label{lemma1}
		Let $L,M>0$ and $B_1,B_2,B_3$ be three independent Brownian Bridges in $[0,L]$. Suppose that $B_i(0)=0$ and $B_i(L)=x_i\in [-M^{-1}\sqrt{L},-M\sqrt{L}]$, $i=1,2,3$. Recall the cutoff function from $q_r$ from \eqref{cutoff}. Then there exists a constant $c=c(M)>0$ such that $\Pr(\m{E}_{q_L}) \ge cL^{-3/4}$ where
		$$\m{E}_{q_L}:=\big\{B_1(x)+B_2(x)\leq 1-q_L(x), B_2(x)+B_3(x)\leq 1-q_L(x)\text{ for all }x\in[0,L]\big\}.
		$$
	\end{proposition}

     \begin{figure}[h!]
			\centering
			\begin{overpic}[width=7cm]{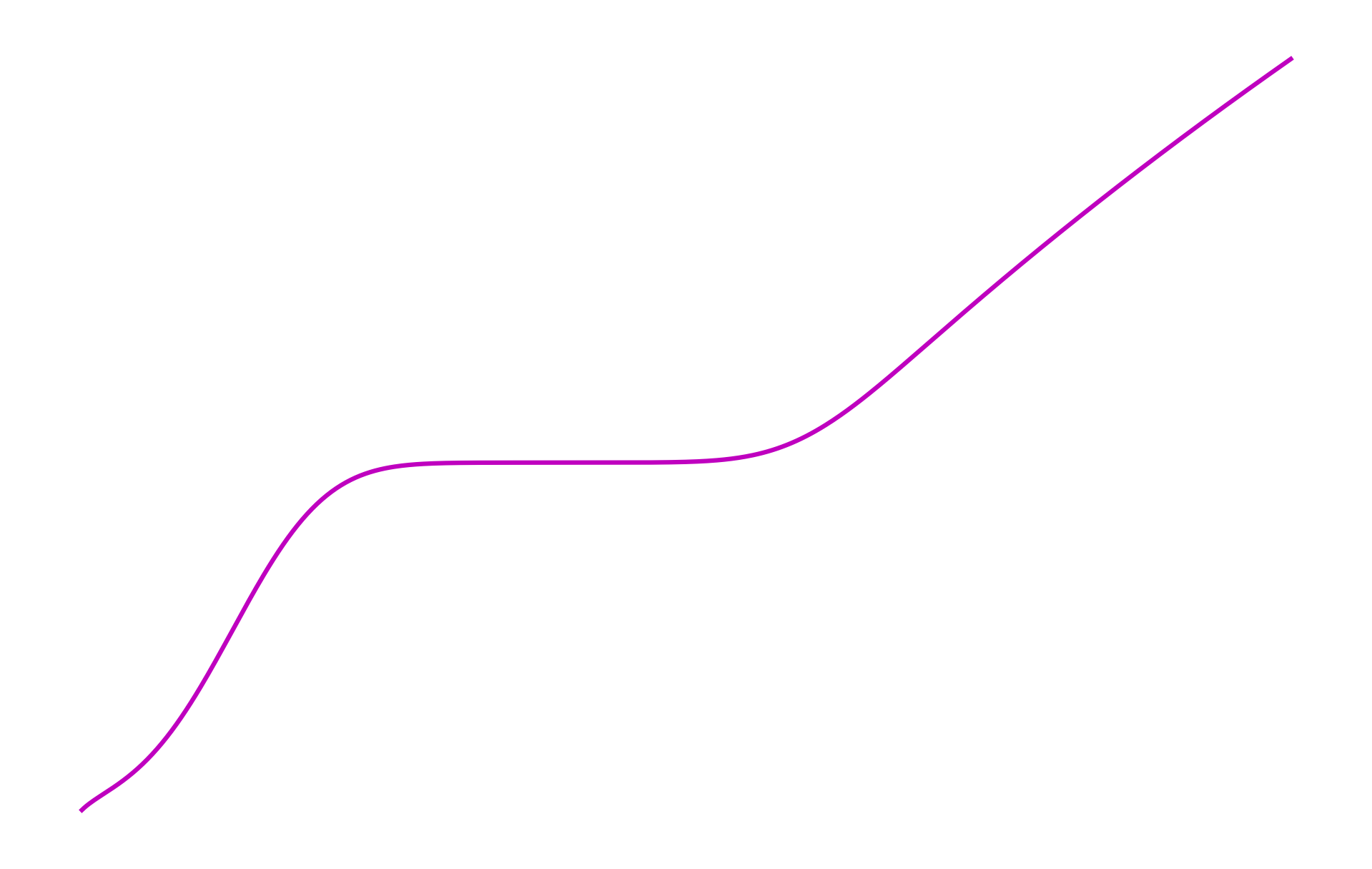}
\put(-2,0){\begin{tikzpicture}
    \draw [line width=1pt,->] (0,0)-- (0,4);
\end{tikzpicture}}
\put(-1,2.3){\begin{tikzpicture}
    \draw [line width=1pt,->] (0,0)-- (7,0);
\end{tikzpicture}}
\end{overpic}
			\caption{Graph of $q_r|_{y\ge 1}$, when $r$ is big enough. Note that $q_r|_{y\le 1}=0$.}
			\label{fig:enter-label2}
		\end{figure}
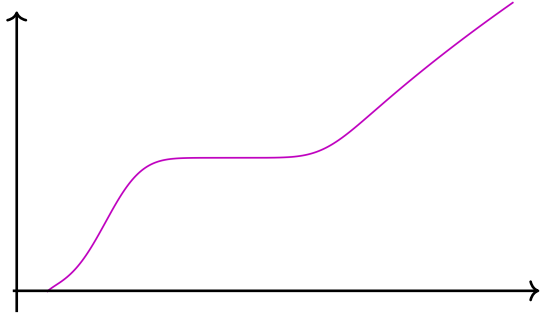

	The proof of the above proposition follows similar lines along Section 4 in \cite{yu2} and involves precise Brownian motion computations. Set $U_k=B_k+B_{k+1}$,
    \begin{equation}\label{vkdef}
    v_k=-\frac{\sqrt{2}}{2}(\sqrt{3},(-1)^k)
    \end{equation}
    for $k=1,2$. Determine $u_1,u_2$ from the equations
	$$v_k \cdot (u_1,u_2) =x_k+x_{k+1}, \mbox{ for }k=1,2.$$
	The key idea is that  after a change of variables we can  find a 2-dimensional Brownian Bridge $V$ on $[0,L]$ from $(0,0)$ to $(u_1,u_2)$ such that 
	\begin{align}\label{xref}
		v_k \cdot V(x) = B_k(x)+B_{k+1}(x), \mbox{ for } k=1,2,
	\end{align}
	We consider the function $\omega :\R^2\rightarrow \R$ such that $\omega (v)=\max_{k=1,2}\{ v\cdot v_k\}$. Then
	\begin{align*}
		\m{E}_{q_L}=\big\{ \omega(V(x))\leq 1-q_L(x)\text{ for all }x\in [0,L]\big\}.
	\end{align*}
	We shall first prove the above proposition when $q_L \equiv 0$:
	\begin{lemma}\label{lemma12}
		Let $V$ and $\omega$ be as above. Then there exists a constant $c=c(M)>0$ such that $$\Pr(\omega(V(x))\leq 1\text{ for all }x\in [0,L]) \ge cL^{-3/4}.$$
	\end{lemma}

	\begin{proof}[Proof of Lemma \ref{lemma12}.]  
		Note that 
		\begin{align}
			\mathbb{P}(\omega(V(x))\leq 1\text{ for all }x\in [0,L]) &= \frac{p_L^{N_1}(\mathbf{0},\mathbf{u})}{p_L(\mathbf{u})},\label{cond1}
		\end{align}
		where $p_L(\mathbf{u})=(2\pi L)^{-1}e^{-(u_1^2+u_2^2)/2L}$ is the standard 2-dimensional heat kernel and $p_L^{N_a}$ is the transition kernel of standard 2-dimensional Brownian motion killed on the boundary of the wedge $N_a:= \omega^{-1}((-\infty,a])$.  $p_L^{N_a}$ admits a precise expression. Indeed, by \cite[p. 379]{carslow1986conduction} we have
		\begin{align}
			p_L^{N_1}(\mathbf{0},\mathbf{u}) = \frac1L\sum_{j=1}^{\infty} I_{\frac{3}{2}j}\left( \frac{\sqrt{6}r}{3L}\right) e^{-1/3L}\sin\left(\frac{\pi}{2}j\right)e^{-r^2/(2L)}\sin\left( \frac{3}{2}j(\theta+\frac{\pi}{3})\right) \label{ineq100}
		\end{align}
		where $I_k$ is the modified Bessel function of the first kind, where $r,\theta$ depends on $\mathbf{u}$ and are chosen so that $re^{i\theta}$ represents the complex number corresponding to $\mathbf{u}+(\sqrt{6}/3,0)$. It is not hard to check that for all $L$ large enough, $c_1\sqrt{L} \le r \le c_1^{-1}\sqrt{L}$ and $|\theta| \le \frac{\pi}{3}-c_2$ for some constants $c_1,c_2>0$ depending only on $M$.  Using the fact that (see \cite[(9.8.18) p. 376]{abramowitz1968handbook})
        $$\frac{(z/2)^a e^{-z}}{\Gamma(a+\frac{1}{2})}\leq I_a(z)\leq \frac{(z/2)^a e^{z}}{\Gamma(a+\frac{1}{2})},$$
		for all $a,x\ge 0$ we obtain that $p_L^{N_1}(\mathbf{0},\mathbf{u}) \ge c\cdot L^{-7/4}$ for some constant $c$ depending only on $M$. On the other hand, note that $p_L(\mathbf{u}) \le C\cdot L^{-1}$ for some constant $C>0$ depending only on $M$. Combining these two bounds, we arrive at the desired result.
	\end{proof}

	\begin{lemma}\label{prop1}
		Adopt the assumption of Proposition \ref{lemma1}. Then, there exists a constant $c(M)>0$, such that $\mathbb{P}\big(\m{E}_{q_L}\mid \m{E}_0\big)\geq c$.
	\end{lemma}
	
	\begin{proof}[Proof of Proposition \ref{lemma1}] The proof is immediate from  Lemmas \ref{lemma1} and \ref{prop1}.
	\end{proof}  
	
	\begin{proof}[Proof of Lemma \ref{prop1}] 
		Let $g$ be the functional such that 
		$$g(Y)=-\int_0^{L} q_L''(x)Y(x)dx,$$
		where $Y:\R\rightarrow \R$ is a continuous function. By It\^o's formula, we have
		\begin{equation*} 
			\int_0^Lq_L'(x)dV_1(x)=q_L'(L)V_1(L)-\int_0^Lq_L''(x)V_1(x)dx=-\int_0^Lq_L''(x)V_1(x)dx=g(V_1).\end{equation*}
		where the last two equalities are due to the fact that $q_L'(x)=q_L''(x)=0$ for $x\ge L/2$. By Girsanov's theorem, we know that the density of the law of $(V(x)+q_L(x)h)_{x\in [0,L]}$ with respect to the law of the Brownian motion is $\text{exp}\bigg( h_1 \int_0^Lq_L'(x)dV_1(x)-\frac{h_1^2}{2}\int_0^L|q_L'(x)|^2dx\bigg)$. Thus for $R>0$ we have
		\begin{align*}
			\mathbb{P}(\m{E}_{q_L}) & \geq \mathbb{P}\left( (V(x)+q_L(x)h)\in N_1 \text{, for all }x\in [0,L] \text{, and } g(V_1)\leq R \right) \\
			& = e^{-\int_0^L|q_L'(x)|^2dx}\cdot \mathbb{E}\left[ e^{-h_1g(V_1)}\ind_{\m{E}_0\cap \{ g(V_1)\leq R\}} \right]
			\ge c\cdot \mathbb{P}\big(\m{E}_0 \cap \{ g(V_1)\leq R\}\big).
		\end{align*}
		for some $c>0$ depending only on $R$. The last inequality follows by noting that $|q_L'(x)|\le |x|^{1/3}$. 
		Thus, since $\m{E}_{q_L}\subset \m{E}_0$,
		\begin{equation}\label{eq11}
			\mathbb{P}[\m{E}_{q_L}|\m{E}_0]\geq c\cdot\mathbb{P}( g(V_1)\leq R |\m{E}_0) \ge c\cdot \left(1-\frac1R\Ex\left[\int_0^L |q_L''(x)||V_1(x)|dx \mid \m{E}_0\right]\right)
		\end{equation}
		where the last inequality follows from Markov's inequality. We claim that  for all $x\in [0,L/2]$
		\begin{align}\label{vineq}
			\mathbb{E}\left[ |V(x)| \bigg| \m{E}_0\right]\leq Cx^{1/2}.
		\end{align}
		for some constant $C>0$ depending only on $M$. Using \eqref{vineq} and the fact that $|q_L''(x)|\le x^{-5/3}$ on $[0,L/2]$ and $q_L''(x)=0$ for $x\ge L/2$, it follows that the conditional expectation on r.h.s.~of \eqref{eq11} is upper bounded by some $C_1$ depending on $M$. Thus, we may choose $R$ large enough so that r.h.s.~of \eqref{eq1} $\ge c/2$. This completes the proof modulo \eqref{vineq}.

		\medskip
		
		Let us now establish \eqref{vineq}. Recall the density $p_L^{N_a}$ introduced in the proof of Lemma \ref{lemma12}. We will use some properties of $p_x^{N_a}$ that can be found in \cite[Section 4]{yu2}. Using $p_x^{N_a}(v_1,v_2)=p_x^{N_0}(v_1-ah,v_2-ah)$, we get 
        \begin{align}
			\mathbb{E}\left[ |V(x)| \bigg| \m{E}_0\right]&= \frac{\int_{N_1}|v|p_x^{N_1}(0,v)p_{L-x}^{N_1}(v,\mathbf{u})du}{p_L^{N_1}(0,\mathbf{u})}=\frac{\int_{N_0}|v+h|p_x^{N_0}(-h,v)p^{N_0}_{L-x}(v,\mathbf{u}-h)dv}{p_L^{N_0}(-h,\mathbf{u}-h)}\nonumber\\
			&\leq |h|+ \frac{\int_{N_0}|v|p_x^{N_0}(-h,v)p^{N_0}_{L-x}(v,\mathbf{u}-h)dv}{p_L^{N_0}(-h,\mathbf{u}-h)},\label{ineq101}
		\end{align}
        where in the last step we used the triangle inequality. Let us now bound the numerator and denominator of the above expression. From Lemma 4.4 in \cite{yu2}, there exists a constant $C_B$ such that for every $r_1,r_2>0$ and $\theta_1,\theta_2\in [0,2\pi)$
		\begin{equation}\label{property}
			p_x^{N_0}(r_1e^{i\theta_1},r_2e^{i\theta_2})\leq \frac{C_B (r_1r_2)^{3/2}}{x^{5/2}}e^{\frac{\sqrt{2}r_1r_2 -r_1^2-r_2^2}{2x}}\prod_{k=1}^2 \left(\frac{\pi}{3}- |\theta_k|\right).
		\end{equation}
	 As in Lemma \ref{lemma12}, the vector $\mathbf{u}-h=r_0e^{i\theta_0}$ is such that $c_1^{-1}\sqrt{L}\leq r_0\leq c_2\sqrt{L}$ with $|\theta_0|\leq \frac{\pi}{3}-c_2$ for some constants $c_1,c_2>0$. By \eqref{property}, we have
		\begin{align*}
			& \int_{N_0}|v|p_x^{N_0}(-h,v)p^{N_0}_{L-x}(v,\mathbf{u}-h)dv\\
			&\leq C\left(\frac{\pi}{3}-|\theta_0|\right)\frac{r_0^{3/2}}{(x(L-x))^{5/2}}e^{-\frac{1}{3}x^{-1}-\frac{r_0^2}{2(L-x)}}\int_0^{\infty}r^5e^{-\frac{r^2}{2}(x^{-1}+(L-x)^{-1})+\sqrt{2}r(\frac{\sqrt{6}}{3}x^{-1}+r_0(L-x)^{-1})}dr\\
			&\stackrel{r\mapsto x^{1/2}r}{\leq} C'\left(\frac{\pi}{3}-|\theta_0|\right)\frac{r_0^{3/2}x^{1/2}}{(L-x)^{5/2}}e^{-\frac{1}{3}x^{-1}-\frac{r_0^2}{2(L-x)}}\int_0^{\infty}r^5e^{-\frac{r^2}{2}(1+x(L-x)^{-1})+cr(x^{-1/2}+r_0x^{-1/2}(L-x)^{-1})}dr,
		\end{align*}
		where $C,C',c$ are positive constants. However, since $x\in [1,L/2]$, $L\geq 4$ and $r_0\leq c_1\sqrt{L}$, we know that $x^{-1/2}+r_0x^{-1/2}(L-x)^{-1}\leq 1+\frac{2c}{\sqrt{L}}\leq c+1$, hence we can bound the last integral from above to get
		\begin{align*}
			\int_{N_0}|v|p_x^{N_0}(-h,v)p^{N_0}_{L-x}(v,\mathbf{u}-h)dv
			&\leq C'\left(\frac{\pi}{3}-|\theta_0|\right)\frac{r_0^{3/2}x^{1/2}}{(L-x)^{5/2}}e^{-\frac{1}{3}x^{-1}-\frac{r_0^2}{2(L-x)}}\int_0^{+\infty}r^5e^{-\frac{r^2}{2}+r(c+1)}dr\\
			&\leq C''x^{1/2}r_0^{3/2}L^{-5/2}\int_0^{+\infty}r^5e^{-c'r^2}dr\leq Cx^{1/2}r_0^{3/2}L^{-5/2}.
		\end{align*}
		On the other hand, due to \eqref{ineq100} and the estimates on $r_0,\theta_0$
		\begin{align*}
			p_L^{N_0}(-h,\mathbf{u}-h)&\geq \frac{1}{L}I_{\frac{3}{2}}\left( \frac{\sqrt{6}r_0}{3L}\right)e^{-\frac{1}{3L}-\frac{r_0^2}{2L}}\sin\left( \frac{3}{2}(\theta_0+\frac{\pi}{3})\right) \\ & \geq \frac{e^{-\frac{1}{3L}-\frac{r_0^2}{2L}}}{L}\sin\left( \frac{3}{2}(\theta_0+\frac{\pi}{3})\right)\left(\frac{\sqrt{6}r_0}{6L}\right)^{3/2}e^{-\frac{\sqrt{6}r_0}{3L}} \geq Cr_0^{3/2}L^{-5/2}.
		\end{align*}
		Plugging the above two bounds in \eqref{ineq101} verifies \eqref{vineq}.
			\end{proof}

	\section{Upper bound for the asymptotic variance} \label{sec:5}

	In this section, we prove the upper bound—Proposition \ref{p.upbd}. In Section \ref{sec:5.1}, we establish this proposition assuming a technical estimate on Brownian bridges (Proposition \ref{coretech}), which forms the core component of the upper bound. Sections \ref{sec:5.2}–\ref{sec:5.4} are devoted to proving this technical proposition.

    \subsection{Proof of Proposition \ref{p.upbd}} \label{sec:5.1} Assume $u,v\ge 0$ with $u+v>0$.
    Let us recall the expression $\Pf(y)$ from \eqref{def:pfy}. 
    To prove \eqref{part11} and \eqref{part12}, we require a few notations to begin with. Let us set
    \begin{equation}
		\label{newdef1}
		\begin{aligned}
           & U_i^{(1)} := \max\left\{\frac{\Lambda_i(y)-\Lambda_i(0)}{\sqrt{y+1}},1\right\}\mbox{ and }U_i^{(2)}:= \max\left\{\frac{\Lambda_i(y+1)-\Lambda_i(L)}{\sqrt{L-y}},1\right\}, \\
           & A(y) :=\int\limits_y^{y+1} e^{\Lambda_1(x)+2\Lambda_2(x)+\Lambda_3(x)-\Lambda_1(y)-2\Lambda_2(y)-\Lambda_3(y)}dx \\ & R(y):=\min\left\{1, e^{\Lambda_1(y+1)+2\Lambda_2(y+1)+\Lambda_3(y+1)-\Lambda_1(y)-2\Lambda_2(y)-\Lambda_3(y)}\right\},
		\end{aligned}
	\end{equation}
	\begin{equation}
		\label{newdef}
		\begin{aligned}
			 \\
			& D_{11}(y):=\int\limits_0^y e^{\Lambda_1(x)+\Lambda_2(x)-\Lambda_1(y)-\Lambda_2(y)}dx, \; D_{12}(y):=\int\limits_0^y e^{\Lambda_2(x)+\Lambda_3(x)-\Lambda_1(y)-\Lambda_2(y)}dx, \\
			& D_{21}(y):= \int\limits_{y+1}^L e^{\Lambda_1(x)+\Lambda_2(x)-\Lambda_1(y+1)-\Lambda_2(y+1)}dx, \\ & D_{22}(y):=\int\limits_{y+1}^L e^{\Lambda_2(x)+\Lambda_3(x)-\Lambda_1(y+1)-\Lambda_2(y+1)}dx,
		\end{aligned}
	\end{equation}
	and consider the $\sigma$-field
	$$\mathcal{G}(y)=\sigma\big\{ (\Lambda_i(x), \Lambda_i'(x)), \mbox{ for }x\in \{0\}\cup [y,y+1]\cup \{L\}, i\in \{1,2,3\}\big\}.$$
Note that $A(y),R(y), U_i^{(j)}$ are $\mathcal{G}(y)$-measurable. We have the following estimates on these quantities.
	
	\begin{lemma}\label{techle} There exists $c>0$ such that for every $y\in [0,L-1]$, $M\ge 1$, $i\in \{1,2,3\}$, and $j\in \{1,2\}$ we have
		\begin{align*}
			\Ex[A(y)^3] \le c^{-1}, \quad \Ex[R(y)^{-6}] \le c^{-1}, \quad   \Pr\big(U_i^{(j)} \ge M\big) \le e^{-cM^2}.
		\end{align*}
	\end{lemma}

\begin{proof} Thanks to the Gaussian tail estimates established in Lemma \ref{tailest}, the tail estimate for $U_i^{(j)}$ is immediate. The Gaussian tail estimates imply that  the exponential moments: $$\Ex\left[\exp\left(C\sup_{x\in [y,y+1]} \big|\Lambda_i(x)-\Lambda_i(y)\big|\right)\right],$$ are all finite and uniformly bounded in $y$. Hence the finiteness of moments of $(A(y))^3$ and $(R(y))^6$ follows. 
\end{proof}

Given the definitions in \eqref{newdef1}, \eqref{newdef}, and \eqref{def:pfy}, we have
\begin{align*}
    \Pf(y) = \Ex\left[\frac{A(y)}{(D_{11}(y)+D_{21}(y-1))(D_{12}(y)+D_{22}(y-1))}\right].
\end{align*}
Following the definitions, it is not hard to check that $D_{1j}(y) \ge R(y)D_{1j}(y)$ and $D_{2j}(y-1) \ge R(y)D_{2j}(y)$ almost surely. Thus,
\begin{align} \label{rtwe}
		\Pf(y) & \le \Ex\left[A(y)R(y)^{-2}\prod_{j=1}^2(D_{1j}(y)+D_{2j}(y))^{-1}\right] \\ & \le \Ex\left[A(y)R(y)^{-2} \Ex\left[\prod_{j=1}^2(D_{1j}(y)+D_{2j}(y))^{-1}\mid \mathcal{G}(y)\right]\right]. \label{rtwe2}
	\end{align}
	Note that by stochastic monotonicity, $(\Lambda_i(x)-\Lambda_i(y))_{x\in [0,y]}$ conditioned on $\mathcal{G}(y)$ is stochastically larger than a Brownian bridge $(\mathfrak{B}_{i}^{(1)}(x))_{x\in [0,y]}$ from $-\sqrt{y+1} \cdot U_i^{(1)}$ to $0$. And similarly, $(\Lambda_i(x)-\Lambda_i(y+1))_{x\in [y+1,L]}$ conditioned on $\mathcal{G}(y)$ is stochastically larger than a Brownian bridge $(\mathfrak{B}_{i}^{(2)}(x))_{x\in [y+1,L]}$ from $0$ to $-\sqrt{L-y}\cdot U_i^{(2)}$. Thus, the inner expectation in \eqref{rtwe2} is upper bounded by Brownian bridge analog of the same expression. The core technical result that completes the upper bound argument is the following estimate on this Brownian bridge analog.
	
	\begin{proposition}\label{coretech} Fix $L\ge 1$, $L_1,L_2 \in [L^{1/4},L/2]$. and $b\ge 1$. Suppose $(\mathfrak{B}_i^{1})_{i=1,2,3,j=1,2}$ are independent Brownian bridges on $[0,L_j]$ with $\mathfrak{B}_i^{j}(0)=0$ and $\mathfrak{B}_i^j(L_j)=-b\sqrt{L_j}$. There exists an absolute constant $C>0$ such that
		\begin{align}\label{oneljbd}
			\Ex\left[\prod_{k=1}^2\left(\int_0^{L_j} e^{\mathfrak{B}_{k}^j(x)+\mathfrak{B}_{k+1}^j(x)}dx\right)^{-1}\right] \le C \cdot e^{Cb^{3/2}} L_j^{-3/4},
		\end{align}
		\begin{align}\label{twol1l2bd}
			\Ex\left[\prod_{k=1}^2\left(\int_0^{L_1} e^{\mathfrak{B}_{k}^1(x)+\mathfrak{B}_{k+1}^1(x)}dx+\int_0^{L_2} e^{\mathfrak{B}_{k}^2(x)+\mathfrak{B}_{k+1}^2(x)}dx\right)^{-1}\right] \le C \cdot e^{Cb^{3/2}} L_1^{-3/4}L_2^{-3/4}.
		\end{align}
	\end{proposition}

	Let us assume Proposition \ref{coretech} for a moment.  Set $U:=\max_{i,j} U_i^{(j)}$.
	Thanks to \eqref{twol1l2bd}, and the stochastic monotonicity argument explained earlier, for $y\in [L^{1/4},L-L^{1/4}]$ we have
	\begin{align*}
		\mbox{r.h.s.~of \eqref{rtwe}} & \le C y^{-3/4}(L-y)^{-3/4} \Ex\bigg[ A(y)R(y)^{-2}\exp(C U^{3/2})\bigg] \\ & \le C y^{-3/4}(L-y)^{-3/4} \Ex\big[ A(y)^{-3}\big] \Ex\big[R(y)^{-6}\big] \Ex\bigg[\exp(3CU^{3/2})\bigg],
	\end{align*}
	where the last inequality follows via H\"older inequality. By Lemma \ref{techle}, the above three expectation are uniformly bounded. Thus we have \eqref{part11}. For $y\in [0,L^{1/4}]\cup [L-L^{1/4},L-1]$, we rely on \eqref{oneljbd} instead to derive \eqref{part12}. This completes the proof of the upper bound when $u,v\ge 0$ with $u+v>0$. 
    
    \smallskip
    
    When $u = v = 0$, the stationary measure is simply a Brownian motion. In this case, Lemma \ref{techle} still holds by direct estimates for Brownian motion. The stochastic monotonicity argument used earlier is not required here, as the underlying laws are already those of Brownian bridges. One can then directly apply Proposition \ref{coretech} to the inner expectation in \eqref{rtwe2} to obtain the desired upper bound.
	
	\subsection{Framework for the proof of Proposition \ref{coretech}} \label{sec:5.2} We will prove \eqref{twol1l2bd} only, as the proof of \eqref{oneljbd} is similar and simpler. The case $b \ge \log L$ is relatively straightforward. Indeed, by Jensen inequality and Brownian motion moment computation we see that
		\begin{align*}
			\Ex\left[\prod_{k=1}^2\left(\int_0^{L_j} e^{\mathfrak{B}_{k}^j(x)+\mathfrak{B}_{k+1}^j(x)}dx\right)^{-1}\right] & \le \prod_{k=1}^2 \Ex\left[\left(\int_0^1 e^{\mathfrak{B}_{k}^k(x)+\mathfrak{B}_{k+1}^k(x)} dx\right)^{-1}\right] \\ & \le \prod_{k=1}^2 \left(\int_0^1 \Ex[e^{\mathfrak{B}_{k}^j(x)+\mathfrak{B}_{k+1}^j(x)}] dx\right)^{-1} \le C e^{Cb/\sqrt{L_1}+Cb/\sqrt{L_2}}.
		\end{align*}
        Note that if $b \ge \log L$, we have 
        \begin{align*}
Ce^{Cb/\sqrt{L_1}+Cb/\sqrt{L_2}} \le C' \cdot e^{\frac12C'b^{3/2}}  \le C' \cdot e^{C'b^{3/2}} L_1^{-3/4}L_2^{-3/4} 
        \end{align*}
        for some $C'>0$. This verifies \eqref{twol1l2bd}. So, we may assume $b \le \log L$ for all $i,j$ for the rest of the proof.



 		

As in Section \ref{sec:4.2.2}, we begin by rewriting the expression in \eqref{twol1l2bd} in terms of two-dimensional Brownian bridges. 
Set $\mathbf{w}_j:=\bigg(b\sqrt{\frac23L_j},0\bigg)$. Let $V^i$ be independent planar Brownian bridges of length $L_i$ from the origin to $\mathbf{w}_i$. We have
\begin{align}\label{rbf}
    \mbox{r.h.s.~of~\eqref{twol1l2bd}} = \Ex\left[\prod_{k=1}^2 \left(\int_0^{L_1} e^{v_k\cdot V^1(x)}dx+ \int_0^{L_2} e^{v_k\cdot V^2(x)}dx\right)^{-1}\right].
\end{align}

We remark that in \cite{yu2}, the authors analyze expressions of the form
\[
\Ex\left[\prod_{k=1}^2 \left(\int_0^{L} e^{v_k \cdot V(x)}\,dx\right)^{-1}\right],
\]
where $V$ is a two-dimensional Brownian bridge of length $L$ from the origin to origin. They show that this quantity is of order $L^{-3/2}$. The main contribution to the expectation arises when the path $V$ remains confined to the wedge $N_a=\omega^{-1}((-\infty,a])$. Since the Brownian bridge starts and ends at the origin—which lies close to the boundary of the wedge—maintaining this constraint incurs a probability of $L^{-3/4}$ for each endpoint, leading to an overall contribution of $L^{-3/4} \cdot L^{-3/4} = L^{-3/2}$.  In our case, the Brownian bridges $V^j$ start at the origin but terminate at points in the bulk, specifically at $\mathbf{w}_j = \left(b\sqrt{2L_j/3}, 0\right)$. As a result, the probability that $V^j$ remains confined to the wedge $N_a$ is of order $L_j^{-3/4}$.

\medskip

To show this rigorously, we define few constants, events, and stopping times.
		\[
		\delta=10^{-10}, \quad \beta=10^{-5}, \quad M=\lfloor \delta \log L\rfloor + 1, \quad q_M=\infty, \quad q_m=100^m-1 \mbox{ for }1\le m\le M-1.
		\]
		For $1\leq m\leq M,$ and $i=1,2$ let $\tau_m^i$ (resp.~$\til{\tau}_m^i$) be the first time (resp.~last) in $[0,L_i]$ such that $\omega(V^i(x))=q_{m-1}$.
		These stopping times exist almost surely. We refer to Figure \ref{lowerbdpf} and its caption to illustrate how the stopping times are used to prove the $L^{-3/4}$ bound. 
        \begin{figure}[h!]
    \centering
    \includegraphics[width=0.5\linewidth]{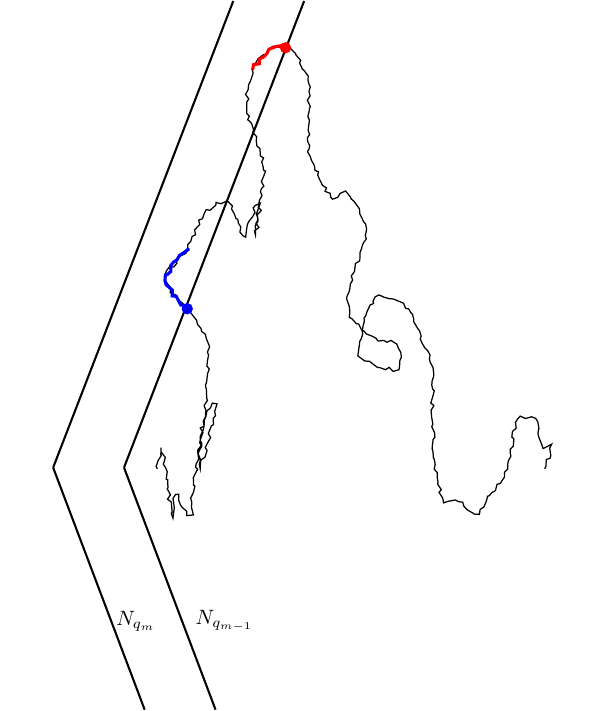}
    \caption{In above $V$ is a planar Brownian bridge on $[0,L]$ from the origin to $(\sqrt{L},0)$. Suppose we are on the event $\m{E}_{q_{m},0,L} \cap \m{E}_{q_{m-1},0,L}^c$. The blue point is the first time $\tau_m$ when it exits $N_{q_{m-1}}$. If the time is less than $L^{1-\beta}$, we consider the region from $[\tau_{m},\tau_{m}+1]$ (the blue highlighted path). The integral is of the order $e^{-q_{m-1}}$ and the cost of staying in wedge $N_{q_m}$ on $[\tau_m+1,L]$ is of the order $L^{-3/2}|w|^{1/2}(q_m-\omega(w))=L^{-3/4}$. If $\tau_m \ge L^{1-\beta}$, we consider the red point which is the last time $\til{\tau}_m$ and consider the region from $[\til\tau_{m}-1,\til\tau_{m}]$ (the red highlighted path). The integral is of the order $e^{-q_{m-1}}$ and the cost of staying in wedge $N_{q_m}$ on $[0, \til\tau_{m}-1]$ is of the order $L^{-3(1-\beta)/2}|V(\til\tau_m)|^{1/2}(q_m-\omega(V(\til\tau_m))) = L^{-3(1-\beta)/2+1/2} \ll L^{-3/4}$ for $\beta$ small enough. }
    \label{lowerbdpf}
\end{figure}

		For $q>0,x_1<x_2$ define the event:
        \begin{align*}
            \m{E}_{q,x_1,x_2}^i & := \{V^i(x) \in N_q \mbox{ for all }x\in [x_1,x_2]\}.
        \end{align*}
        We consider the events:
        \begin{equation*}
            \begin{aligned}
		   \ \m{A}_m^i = \m{E}_{q_m,0,L_i}^i\setminus \m{E}_{q_{m-1},0,L_i}^i, \quad  \m{B}_{m,1}^i = \{\tau_{m}^i \in [0,L_i^{1-\beta}]\},  \quad \m{B}_{m,2}^i = \{\tau_{m}^i \in (L_i^{1-\beta},L_i]\}.
		\end{aligned}
        \end{equation*}
		Set
		\begin{equation*}
			G_{m_1,m_2,i_1,i_2} := \mathbb{E}\left(\prod_{k=1}^2 \left(\int_0^{L_1} e^{v_k\cdot V^1(x)}dx+ \int_0^{L_2} e^{v_k\cdot V^2(x)}dx\right)^{-1}\ind_{\m{A}_{m_1}^1\cap \m{A}_{m_2}^2\cap \m{B}_{m_1,i_1}^1 \cap \m{B}_{m_2,i_2}^{2}}\right).
		\end{equation*}
		Note that
		\begin{align}
		    \label{bound3}
            \mbox{r.h.s.~of \eqref{rbf}} \le \sum_{i_1=1}^2\sum_{i_2=1}^2\sum_{m_1=1}^M\sum_{m_2=1}^M G_{m_1,m_2,i_1,i_2}.
		\end{align}

        In the following lemma we collect several estimates for $G_{m_1,m_2,i_1,i_2}$.
		\begin{lemma}\label{gbond} 
			When  $m_1=M$ or $m_2=M$ we have
	\begin{align}\label{gammaine0}
				G_{m_1,m_2,i_1,i_2} \le Ce^{-L^\delta/C}.
			\end{align}
             When $i_1=2$ or $i_2=2$ and $m_1,m_2 \le M-1$ we get \begin{align}\label{gammaine2}
				G_{m_1,m_2,i_1,i_2} \le C L^{-\delta}L_1^{-3/4}L_2^{-3/4}.
			\end{align}
  When $i_1=i_2=1$ and $m_1,m_2 \le M-1$ we have 
    \begin{align}\label{gammaine3}
				G_{m_1,m_2,i_1,i_2} \le C |b|^{10} q_{m_1}^{10}q_{m_2}^{10} e^{-q_{m_1-1}\vee q_{m_2-1}} L_1^{-3/4}L_2^{-3/4}.
			\end{align}
		\end{lemma}		

        Plugging all the above estimates in \eqref{bound3} and noting that $M\le C\log L$,  we arrive at \eqref{twol1l2bd}.

	\subsection{Proof of Lemma \ref{gbond}}	\label{sec:5.3}
	The proof of the estimates in Lemma \ref{gbond} rest on the following three technical lemmas.

	\begin{lemma}\label{GmQqmcomp} There exists a constant $C>0$ such that for all $L_1,L_2\geq 100$, $i_1,i_2 \in \{1,2\}$, and $1\leq \ell,m \leq M$ we have
		\begin{align} \label{genine}
			G_{m_1,m_2,i_1,i_2}^2 \le 2 \cdot P_{q_{m_1},q_{m_2}}( \vec{w};\vec{L})\cdot \Ex\left[Q_{{m_1},{m_2},i_1,i_2}({x}_*\to {y}_*;\mathbf{R})\ind_{\m{B}_{m_1,i_1}^1 \cap \m{B}_{m_2,i_2}^{2}}\right].
		\end{align}
where ${x}_*= (\vec{x}_{*1}, \vec{x}_{*2})$, $y_*=(\vec{y}_{*1}, \vec{y}_{*2})$, $\mathbf{R}=(R_1,R_2)$ with
\begin{align}\label{xyrdef}
\big(\vec{x}_{*,j}, \vec{y}_{*j}, R_j\big)=\begin{cases}
    \big(V^j(\tau_{m_j}^j), \mathbf{w}_j, L_j-\tau_{m_j}^j\big)  & \mbox{ if } i_j=1 
    \\ \big(V^j(\til{\tau}_{m_j}^j), 0, \til{\tau}_{m_j}^j\big)  & \mbox{ if } i_j=2, 
\end{cases}
\end{align}
and
		\begin{align}\label{pdef}
			P_{q_{m_1},q_{m_2}}(\vec{w};\vec{L}):=
			\mathbb{E}_{\vec{0}\to \vec{w}}^{L_1,L_2} \left[\int_0^{1} e^{2\sqrt{2}\vert V^1(x)\vert}dx\wedge \int_0^{1} e^{2\sqrt{2}\vert V^2(x)\vert}dx.;\m{E}_{q_{m_1},0,L_1}^1\m{E}_{q_{m_2},0,L_2}^2\right],
		\end{align}
		\begin{align}
		    \label{qmdef}
            Q_{{m_1},{m_2},i_1,i_2}({x}_*\to {y}_*;\mathbf{R}) := \mathbb{E}_{{x}_*\to {y}_*}^{R_1,R_2} \left[F_{i_1}^1\wedge F_{i_2}^2;\m{E}_{q_{m_1},0,R_1}^1\m{E}_{q_{m_2},0,R_2}^2\right], 
		\end{align}
		and
        \begin{align*}
            F_{i_j}^j:=\exp(-q_{m_j-1})\int_0^{1} e^{2\sqrt{2}\vert V^j(x)-\vec{x}_{*j}\vert}dx.
        \end{align*}
	\end{lemma}
	
	\begin{lemma}\label{Qnbound1} There is a constant $C>0$ such that for all  $i_1,i_2 \le 2$ and $m_1,m_2 \le M-1$ we have
		\begin{equation}
		    \label{bdf2}
            \begin{aligned}
			Q_{m_1,m_2,i_1,i_2}({x}_*\to {y}_*;\mathbf{R}) & \le C b^5 \cdot e^{-{q_{m_1-1}\vee q_{m_2-1}}} \cdot S_{1}S_2
		\end{aligned}
		\end{equation}   
        where
        \begin{align}\label{sjdef}
            S_i:=& \frac{(y_{*i}+q_{m_i})^{3/2}}{R_i^{3/2}}\bigg[\sqrt{|x_{*i}|}+\sqrt{q_{m_i}}\bigg]e^{C\frac{(q_{m_i}+|y_{*i}|)(q_{m_i}+|x_{*i}|)}{R_i}},
        \end{align}
        and
		\begin{align}\label{bdf3}
			P_{q_{m_1},q_{m_2}}( \vec{w};\vec{L}) \le C b^5\cdot \prod_{i=1}^{2} \frac{|\mathbf{w}_i|^{3/2}q_{m_i}^{1/2}}{L_i^{3/2}}e^{C\frac{q_{m_i}(q_{m_i}+|\mathbf{w}_{i}|)}{L_i}} \le Cb^{10}q_{m_1}q_{m_2}L_1^{-3/4}L_2^{-3/4}.
		\end{align}
    For all $i_1,i_2 \le 2$ and $1\le m_1,m_2 \le M$ we have
    \begin{align}\label{bdf4}
			Q_{m_1,m_2,i_1,i_2}(x_*\to y_*;\mathbf{R}) & \le C e^{-{q_{m_1-1}\vee q_{m_2-1}}} \prod_{j=1}^2 e^{C|\vec{x}_{*j}|/R_j+C|\vec{y}_{*j}|/R_j},
		\end{align}
        and
		\begin{align}\label{bdf5}
			P_{q_{m_1},q_{m_2}}( \vec{w};\vec{L}) \le C.
		\end{align}
	\end{lemma}

	\begin{lemma}\label{vtau} Fix $m\le M-1$. Fix any $\delta>0$. For all $\xi>0$ with $\xi \le L^{1/4-\delta/2}$ we have
		\begin{align*}
			\Pr\Big( |V^j(\tau_{m}^j)| \ge \xi^2+b, \tau_{m}^j \le L_j^{1-\beta}\Big) \le e^{-\xi^{2-4\delta}/C}+ C\frac{(q_{m}+b)^{3/2}}{\xi^{1+2\delta}}.
		\end{align*}
	\end{lemma}
	
	Lemma \ref{GmQqmcomp} uses basic inequalities and Markov property of the Brownian motion -- we shall prove it in a moment. Lemmas \ref{Qnbound1} and \ref{vtau} relies on some explicit probability computation related to Brownian bridges. We defer the proof of these two lemmas in Sections \ref{sec:5.3} and \ref{sec:5.4} respectively. Let us assume these three lemmas and prove the bounds in Lemma \ref{gbond}. 
	\begin{proof}[Proof of Lemma \ref{gbond}] For clarity we split the proof into two steps.
    
    \medskip

    \noindent\textbf{Step 1.} In this step we prove \eqref{gammaine0}.
    Recall $x_*, y_*, \mathbf{R}$ from \eqref{xyrdef}. Note that
    \begin{align*}
        |\vec{x}_{*j}| \le \sup_{x\in [0,L_j]} |V^j(x)|.
    \end{align*}
    Note that if $i_j=1$, on $\m{B}_{m_j,1}^j$ we have $\tau_{m_j}^j \le L_j^{1-\beta}$. So, $R_j = L_j-\tau_{m_j}^j \ge L_j/2$. If $i_j=2$, on $\m{B}_{m_j,2}^j$ we have $\til{\tau}_{m_j}^j \ge \tau_{m_j}^j \ge L_j^{1-\beta}$. So, $R_j = \til\tau_{m_j}^j \ge L_j^{1-\beta}$. Thus, on $\m{B}_{m_j,i_j}^j$ we have $R_j \ge L_j^{1-\beta}$. Since $|\vec{y}_{*j}| \le Cb\sqrt{L_j}$, we thus see that $|\vec{y}_{*j}|/R_j \le 1$ for large enough $L_j$ (i.e., large enough $L$). Using these observations, we deduce from \eqref{bdf4} that when $m_1=M$ or $m_2=M$ we have
    \begin{align*}
        Q_{m_1,m_2,i_1,i_2}(x_*\to y_*;\mathbf{R})\ind_{\m{B}_{m_1,i_1}^1\cap \m{B}_{m_2,i_2}^2} \le C e^{-q_{M-1}}\exp\bigg(C\sum_{j=1}^2 \sup_{x\in [0,L_j]} |V^j(x)|/L_j^{1-\beta}\bigg)
    \end{align*}
   By Gaussian tail estimates, we have $\Ex\bigg[\exp\bigg(C\sum_{i=1}^2 \sup_{x\in [0,L_i]} |V^i(x)|/L_i^{1-\beta}\bigg)\bigg] < C$. As $q_{M-1} \ge L^{\delta}/C$, we get that
   \begin{align*}
       \Ex\left[Q_{m_1,m_2,i_1,i_2}(x_*\to y_*;\mathbf{R})\ind_{\m{B}_{m_1,i_1}^1\cap \m{B}_{m_2,i_2}^2}\right] \le Ce^{-L^{\delta}/C}.
   \end{align*}
   Combining this with the uniform bound from \eqref{bdf5}, in view of \eqref{genine}, we arrive at the desired estimate in \eqref{gammaine0}.
   
    \medskip

    \noindent\textbf{Step 2.} In this step we prove \eqref{gammaine2} and \eqref{gammaine3}. Assume $m_1, m_2 \le M-1$. Recall $S_j$ from \eqref{sjdef}.  We claim that
\begin{align}\label{sjij1}
    & \Ex\left[S_j\ind_{\m{B}_{m_j,1}^j}\right] \le Cq_{m_j}^4b^3 L_j^{-3/4} \mbox{ when } i_j=1, \\ \label{sjij2}
   & \Ex\left[S_j\ind_{\m{B}_{m_j,2}^j}\right] \le Cq_{m_j}^{3/2}b L_j^{-3/4}\cdot L_j^{-1/2+3\beta/2} \mbox{ when } i_j=2.
\end{align}
Assuming the above claim,
In view of the bound in \eqref{bdf2}, we arrive at the following estimates:
\begin{align*}
 &    \Ex\left[Q_{{m_1},{m_2},1,1}({x}_*\to {y}_*;\mathbf{R})\ind_{\m{B}_{m_1,1}^1 \cap \m{B}_{m_2,1}^{2}}\right]\le Cb^{10} e^{-q_{m_1-1}\vee q_{m_2-1}}q_{m_1}^4q_{m_2}^4L_1^{-3/4}L_2^{-3/4}.
     \\ & \Ex\left[Q_{{m_1},{m_2},1,2}({x}_*\to {y}_*;\mathbf{R})\ind_{\m{B}_{m_1,1}^1 \cap \m{B}_{m_2,2}^{2}}\right] \le C L^{-\beta}\cdot L_1^{-3/4}L_2^{-3/4}, \\ & \Ex\left[Q_{{m_1},{m_2},2,1}({x}_*\to {y}_*;\mathbf{R})\ind_{\m{B}_{m_1,2}^1 \cap \m{B}_{m_2,1}^{2}}\right] \le C L^{-\beta}\cdot L_1^{-3/4}L_2^{-3/4}, \\ & \Ex\left[Q_{{m_1},{m_2},2,2}({x}_*\to {y}_*;\mathbf{R})\ind_{\m{B}_{m_1,2}^1 \cap \m{B}_{m_2,2}^{2}}\right] \le C L^{-\beta}\cdot L_1^{-3/4}L_2^{-3/4}.
\end{align*}
Combining this with the estimates from \eqref{bdf3}, in view of the bound \eqref{genine}, verifies \eqref{gammaine2} and \eqref{gammaine3}. We are thus left to show \eqref{sjij1} and \eqref{sjij2}. Recall $S_j$ from \eqref{sjdef} again. Let us suppose $i_j=1$. we have $\vec{y}_{*j}=\mathbf{w}_j=(b\sqrt{2L_j/3},0)$ and on $\m{B}_{m_j,1}^j$ we have $R_j \ge L_j/2$. Thus when $i_j=1$, we have 
 \begin{equation}\label{sj1bd0}
			S_j\ind_{\m{B}_{m_j,1}^j}  \le    Cb^{3/2}L_j^{-3/4}\bigg[\sqrt{|V^j(\tau_{m_j}^j)|}+\sqrt{q_{m_i}}\bigg]e^{Cb\frac{(q_{m_j}+|V^j(\tau_{m_j}^j)|)}{\sqrt{L_j}}}\ind_{\m{B}_{m_j,1}^j}.
		\end{equation} 
Observe that
\begin{equation}\label{sj1bd1}
    \begin{aligned}
& \Ex\left[\bigg[\sqrt{|V^j(\tau_{m_j}^j)|}+\sqrt{q_{m_i}}\bigg]e^{Cb\frac{(q_{m_j}+|V^j(\tau_{m_j}^j)|)}{\sqrt{L_j}}} \ind_{|V^j(\tau_{m_j}^j)|\le L_j^{1/2-2\delta}}\ind_{\m{B}_{m_j,1}^j}\right] \\ & \le C\sqrt{q_{m_j}}+C\cdot \Ex\left[\sqrt{|V^j(\tau_{m_j}^j)|}\ind_{|V^j(\tau_{m_j}^j)|\le L_j^{1/2-2\delta}}\ind_{\m{B}_{m_j,1}^j}\right] \\ & \le C\sqrt{q_{m_j}} +Cb+\int_0^{L_j^{\frac14-\delta}} \Pr\Big(|V^j(\tau_{m_j}^j)| \ge \xi^2+b, \tau_{m_j}^j \le L_j^{1-\beta}\Big)d\xi \le Cq_{m_j}^4b,
\end{aligned}
\end{equation}
where in the last line we use Lemma \ref{vtau}.
Now On $\ind_{\m{B}_{m_j,1}^j}$ $|x_{*j}| \le \sup_{x\in [0,L_j^{1-\beta}]} |V^j(x)|:=L_j^{(1-\beta)/2}|b|X_\beta^j$ where $X_\beta^j$ has Gaussian tails. Thus,

\begin{equation}\label{sj1bd2}
    \begin{aligned}
    & \Ex\left[\bigg[\sqrt{|V^j(\tau_{m_j}^j)|}+\sqrt{q_{m_i}}\bigg]e^{Cb\frac{(q_{m_j}+|V^j(\tau_{m_j}^j)|)}{\sqrt{L_j}}} \ind_{|V^j(\tau_{m_j}^j)|\ge L_j^{1/2-2\delta}}\ind_{\m{B}_{m_j,1}^j}\right] \\ & \le L_j^{(1-\beta)/4}\Ex\left[\bigg(\sqrt{X_\beta^j}+1\bigg)e^{CbL^{-\beta/2}X_{\beta}^j}\ind_{X_\beta^j \ge L^{\beta/2-2\delta}}\right] \\ & \le L_j^{(1-\beta)/4}\Ex\left[\bigg(\sqrt{X_\beta^j}+1\bigg)^2e^{2CbL^{-\beta/2}X_{\beta}^j}\right]\Pr\bigg(X_\beta^j \ge L^{\beta/2-2\delta}\bigg) \\ & \le CL_j^{(1-\beta)/4} e^{-L_j^{\beta-4\delta}/C} \le Cq_{m_j}^4b.
\end{aligned}
\end{equation}
Combining \eqref{sj1bd1} and \eqref{sj1bd2}, in view of \eqref{sj1bd0} we arrive at \eqref{sjij1}.

\smallskip
    
Let us now suppose $i_j=2$. In this case, $\vec{y}_{*j} = (0,0)$. On $\m{B}_{m_j,2}^j$ we have $R_j \ge L_j^{1-\beta}$.  On $\m{B}_{m_j,2}^j$  we have 
    \begin{equation*}
			S_j \ind_{\m{B}_{m_j,2}^j} \le  {q_{m_j}^{3/2}}L_j^{-3/2(1-\beta)}\bigg[\sqrt{|x_{*j}|}+\sqrt{q_{m_j}}\bigg]e^{C\frac{q_{m_j}(q_{m_j}+|x_{*j}|)}{R_j}}. 
		\end{equation*}    
We use $|\vec{x}_{*j}| \le \sup_{x\in [0,L_j]} |V^j(x)|:=L_j^{1/2}X_0^j$ to get
\begin{equation*}
			S_j  \le  {q_{m_j}^{3/2}}L_j^{-3/2(1-\beta)}L_j^{1/4}\bigg[\sqrt{|X_0^j|}+1\bigg]e^{Cq_{m_j}X_{0}^jL_j^{-1/2+\beta}}. 
		\end{equation*}    
It is clear that $\Ex\big[(\sqrt{|X_0^j|}+1)e^{Cq_{m_j}X_{0}^jL_j^{-1/2+\beta}}\big] \le Cb$. Thus we arrive at \eqref{sjij2}. This completes the proof. 
\end{proof}

	We now turn towards the proof of Lemma \ref{GmQqmcomp}.

	\begin{proof}[Proof of Lemma \ref{GmQqmcomp}]
For simplicity set $\m{S}:=\m{A}_{m_1}^1\cap \m{A}_{m_2}^2 \cap \m{B}_{m_1,i_1}^1 \cap \m{B}_{m_2,i_2}^{2}$, and
		\begin{align*}
			Z_k := \left(\int_0^{L_1} e^{v_k \cdot V^1(x)} dx+\int_0^{L_2} e^{v_k \cdot V^2(x)} dx\right)^{-1},
		\end{align*}
        where we recall the definition of $v_k$ in \eqref{vkdef}.  Set \begin{align*}
    I_j:= \begin{cases}
        [\tau_{m_j}^j,\tau_{m_j}^j+1] & i_j=1 \\ [\til{\tau}_{m_j}^j-1,\til{\tau}_{m_j}^j] & i_j=2.
    \end{cases}
    \end{align*}
		Observe that \begin{align*}
			Z_k \le Z_k':= \left(\int_0^{1} e^{v_k \cdot V^1(x)} dx+\int_0^{1} e^{v_k \cdot V^2(x)} dx\right)^{-1}, \ \ Z_k \le Z_k'':= \left(\int_{I_1} e^{v_k \cdot V^1(x)} dx+\int_{I_2} e^{v_k \cdot V^2(x)} dx\right)^{-1}.
		\end{align*}
Write $Z_1Z_2 = (Z_1 \vee Z_2)(Z_1 \wedge Z_2)$. 
and use the above bound, followed by an application of Cauchy-Schwarz inequality, to get
		\begin{align*}
			G_{m_1,m_2,i_1,i_2}^2 = (\Ex[ Z_1Z_2\ind_{\m{S}}])^2 \le (\Ex\left[(Z_1' \vee Z_2')(Z_1'' \wedge Z_2'')\ind_{\m{S}}\right])^2 \le \Ex\left[(Z_1' \vee Z_2')^2\ind_{\m{S}}\right] \Ex\left[(Z_1'' \wedge Z_2'')^2\ind_{\m{S}}\right].
		\end{align*}
		We shall bound the last two expectations separately. For the first expectation note that
		\begin{align*}
			(Z_1'\vee Z_2')^2 \le Z_1'^2+Z_2'^2  & \le \sum_{k=1}^2 \bigwedge_{j=1}^2 \left(\int_0^1 e^{v_k\cdot V^j(x)}dx\right)^{-2} \\ & \le \sum_{k=1}^2 \bigwedge_{j=1}^2\left(\int_0^1 e^{-2v_k\cdot V^j(x)}dx\right)  \le 2\bigwedge_{j=1}^2\left(\int_0^1 e^{2\sqrt{2}|V^j(x)|}dx\right).
		\end{align*}
		Multiplying by $\ind_{\m{S}}$ and taking expectations in all sides of the above chain of inequalities we arrive at $\Ex\left[(Z_1' \vee Z_2')^2\ind_{\m{S}}\right] \le 2P_{m_1,m_2}(\vec{w};\vec{L})$
		where $P_{m_1,m_2}(\vec{v};\vec{L})$ is defined in \eqref{pdef}. We next claim that 
		\begin{align}\label{ttau}
			\Ex\left[(Z_1'' \vee Z_2'')^2\ind_{\m{S}}\right] \le \Ex\left[Q_{{m_1},{m_2},i_1,i_2}({x}_*\to {y}_*;\mathbf{R}) \ind_{\m{B}_{m_1,i_1}^1 \cap \m{B}_{m_2,i_2}^{2}}\right].
		\end{align}
This will verify \eqref{genine}. Let $k_m^i$ (resp.~$\til{k}_m^i$)  be s.t. $\omega(V^i(\tau_m^i)) = v_{{k}_m^i} \cdot V^i(\tau_m^i)$ (resp.~$\omega(V^i(\til\tau_m^i)) = v_{\til{k}_m^i} \cdot V^i(\til\tau_m^i)$). For simplicity assume $i_1=i_2=1$ (the other cases are analogous).
		For the second term note that
		\begin{align*}
			(Z_1'' \wedge Z_2'')^2 & \le \bigwedge_{j=1}^2\left(
			\int_{\tau_{m_j}^j}^{\tau_{m_j}^j+1} e^{v_{k_{m_j}^j}\cdot V^j(x)} dx\right)^{-2}  \\ & \le \bigwedge_{j=1}^2 
			\int_{\tau_{m_j}^j}^{\tau_{m_j}^j+1} e^{-2v_{k_{m_j}^j }\cdot V^j(x)} dx  \\ & = \bigwedge_{j=1}^2\left(
			e^{-2v_{k_{m_j}^j }\cdot V^j(\tau_{m_j}^j)}\int_{\tau_{m_j}^j}^{\tau_{m_j}^j+1} e^{-2v_{k_{m_j}^j }(V^j(x)-V^j(\tau_{m_j}^j))} dx\right).
		\end{align*}
		On the event $\m{B}_{m_j,1}^j$, we have $\tau_{m_j}^j \le L_j^{1-\beta}$. Thus $v_{k_{m_j}}^j\cdot V^j(\tau_{m_j}^j) = \omega(V^j(\tau_{m_j}^j))=
			    q_{m_j-1},$
		for $j=1,2$. Note that $\m{E}_{q_{m_j},0,L_j}^j\cap \m{B}_{m_j,i_j}^j=\m{E}_{q_{m_j},\tau_{m_j}^j,L_j}^j\cap \m{B}_{m_j,i_j}^j$ for $j=1,2$.
		Thus following the definition of $Q$ in \eqref{qmdef}, by Markov property we arrive at \eqref{ttau}. This completes the proof.
	\end{proof}

	\subsection{Proof of results from Section \ref{sec:5.3}} \label{sec:5.4}
\subsubsection{Proof of Lemma \ref{Qnbound1}}
To prove Lemma \ref{Qnbound1}, we first study the probability that a planar Brownian bridge stays in the wedge $N_a:=\omega^{-1}((-\infty,a])$ where $\omega(\cdot)$ is defined around \eqref{xref}. The following lemma is an extension of Lemma 4.11 in \cite{yu2}.
 
	\begin{lemma}\label{l4.11equiv}
		 Fix any $q>0$. Let $\mathbf{w}_1,\mathbf{w}_2\in N_q$. Let $V$ be a planar Brownian bridge on $[0,J]$ from $\mathbf{w}_1$ to $\mathbf{w}_2$. There exists an absolute constant $C>0$ such that
		\begin{equation*}
			\mathbb{P}(V(x)\in N_q \mbox{ for all }x\in [0,J]) \leq C J^{-\frac{3}{2}} e^{C\frac{\vert \mathbf{w}_1-qh\vert\vert \mathbf{w}_2-qh\vert}{J}}\prod_{k=1}^2 |\mathbf{w}_k-qh|^{\frac12}|q-\omega(\mathbf{w}_k)|.
		\end{equation*}
	\end{lemma}
	
	\begin{proof} Our proof is adapted from Lemma 4.11 \cite{yu2} which proves the above result when $\mathbf{w}_2=\mathbf{0}$.
		Recall the  densities $p_J$ and $p_J^{N_a}$ of 2D Brownian motion and 2D killed Brownian motion defined around \eqref{cond1}. We have
		\begin{equation*}
			\mathbb{P}(V(x)\in N_q \mbox{ for all }x\in [0,J]) = \frac{p_J^{N_q}(\mathbf{w}_1,\mathbf{w}_2)}{p_J(\mathbf{w}_1,\mathbf{w}_2)}=\frac{p_J^{N_0}(\mathbf{w}_1-qh,\mathbf{w}_2-qh)}{p_J(\mathbf{w}_1,\mathbf{w}_2)}.
		\end{equation*}
		By Lemma 4.4 in \cite{yu2}, we have
		\begin{equation}\label{w0c}
        \begin{aligned}
			 p_J^N(\mathbf{w}_1-qh,\mathbf{w}_2-qh)
			& \leq  \frac{C}{J^{\frac{5}{2}}} \exp\left(\frac{C\vert \mathbf{w}_1-qh\vert\vert \mathbf{w}_2-qh\vert -\vert \mathbf{w}_1-qh\vert^2 - \vert \mathbf{w}_2-qh\vert^2}{2J}\right)\\
			& \hspace{3cm}\cdot \prod_{k=1}^2 \vert \mathbf{w}_k-qh\vert^{3/2}\left\vert\frac{\pi}{3}-\vert \mathrm{arg}(\mathbf{w}_k-qh)\vert\right\vert,
		\end{aligned}
        \end{equation}
		where we take $\mathrm{arg}$ to be in $[-\frac{\pi}{2},\frac{\pi}{2}].$ Now we claim there exists an absolute constant $C$ such that
        \begin{align}
            \label{wclaim}
            \left\vert \frac{\pi}{3}-\vert \mathrm{arg}(\mathbf{w}-qh)\vert\right\vert \leq C \frac{\vert q-\omega(\mathbf{w})\vert}{\sqrt{2}\vert \mathbf{w}-qh\vert}
        \end{align}
	for all $\mathbf{w}\in N_q$. Assuming \eqref{wclaim}, plugging the above bound in \eqref{w0c}, and using the fact that $p_J(\mathbf{w}_1,\mathbf{w}_2)=\frac1{2\pi J}\exp(-|\mathbf{w}_1-\mathbf{w}_2|^2/2J)$, we obtain	
    \begin{align*}
			& \mathbb{P}(V(x)\in N_q \mbox{ for all }x\in [0,J]) \\ & \hspace{1cm}\le \frac{C}{J^{3/2}} \exp\left(\frac{C\vert \mathbf{w}_1-qh\vert\vert \mathbf{w}_2-qh\vert + \vert \mathbf{w}_1-\mathbf{w}_2\vert^2 -\vert \mathbf{w}_1-qh\vert^2 - \vert \mathbf{w}_2-qh\vert^2}{2J}\right) \\ & \hspace{2cm}\cdot \prod_{k=1}^2 \vert \mathbf{w}_k-qh\vert^{\frac{1}{2}} \vert \omega(\mathbf{w}_k)-q\vert.
		\end{align*}
       As $C\vert \mathbf{w}_1-qh\vert\vert \mathbf{w}_2-qh\vert + \vert \mathbf{w}_1-\mathbf{w}_2\vert^2 -\vert \mathbf{w}_1-qh\vert^2 - \vert \mathbf{w}_2-qh\vert^2 = (C-2)\vert \mathbf{w}_1-qh\vert\vert \mathbf{w}_2-qh\vert$,
		we arrive at the desired bound in Lemma \ref{l4.11equiv} by adjusting the constant $C$. Thus we are left to show \eqref{wclaim}. 
    Without loss of generality,  assume $\mathrm{arg}(\mathbf{w}-qh)\geq 0$. then for some constant $C>0$ we have
		\[
		\left\vert \frac{\pi}{3}- \mathrm{arg}(\mathbf{w}-qh)\right\vert \leq C \sin\left(\frac{\pi}{3}-\arg(\mathbf{w}-qh)\right) = -C\cos\left(\frac{5\pi}{6}-\mathrm{arg}(\mathbf{w}-qh)\right).
		\]
		Now note that
		\[
		v_1 \cdot (\mathbf{w}-qh) = \sqrt{2} \vert \mathbf{w}-qh\vert \cos\left(\frac{5\pi}{6}-\mathrm{arg}(\mathbf{w}-qh)\right)
		\]
		and so
		\begin{eqnarray*}
			\left\vert \frac{\pi}{3}-\vert\mathrm{arg}(\mathbf{w}-qh)\vert\right\vert \leq  \frac{C}{\sqrt{2}\vert \mathbf{w}-qh\vert} v_1 \cdot (qh-\mathbf{w}) = \frac{C}{\sqrt{2}\vert \mathbf{w}-qh\vert} (q-v_1\cdot \mathbf{w}).
		\end{eqnarray*}
		Since $\arg(\mathbf{w}-qh)\geq 0,$ we have $v_1 \cdot \mathbf{w} \geq v_2 \cdot \mathbf{w},$ or equivalently $v_1 \cdot \mathbf{w} = \omega(\mathbf{w}).$ This verifies \eqref{wclaim} completing the proof. \end{proof}

We are now ready to prove Lemma \ref{Qnbound1}.

	\begin{proof}[Proof of Lemma \ref{Qnbound1}] We shall prove \eqref{bdf2} and \eqref{bdf4}; the proof of \eqref{bdf3} and \eqref{bdf5} is similar and easier. We recall the definition of $Q$ from \eqref{qmdef}. Note that 
        \begin{equation}\label{eqn2wopower}
		\begin{aligned}
			& Q_{{m_1},{m_2},i_1,i_2}({x}_*\to {y}_*;\mathbf{R}) \\ & := \mathbb{E}_{{x}_*\to {y}_*}^{R_1,R_2} \left(F_{i_1}^1\wedge F_{i_2}^2\ind_{\m{E}_{q_{m_1},0,R_1}^1\m{E}_{q_{m_2},0,R_2}^2}\right)\\
			& \le \mathbb{E}_{x_*\to y_*}^{R_1,R_2}\bigg[F_{i_1}^1 \wedge F_{i_2}^2 \cdot \mathbb{P}_{1,V_1(1)}^{R_1,\vec{y}_{*1}}(\m{E}_{q_{m_1},1,R_1}) \mathbb{P}_{1,V_2(1)}^{R_2,\vec{y}_{*2}}(\m{E}_{q_{m_2},1,R_2})\bigg]\\
			&= \int_{N_{q_{m_1}}}\int_{N_{q_{m_2}}} \hspace{-0.5cm}\mathbb{E}_{\vec{x}\to (z_1,z_2)}^{1,1}[F_{i_1}^1 \wedge F_{i_2}^2] \prod_{j=1}^2 p_{\frac{R_j-1}{R_j}}\left(z_j-\frac{\vec{y}_{*j}+\vec{x}_{*j}(R_j-1)}{R_j}\right)\cdot\mathbb{P}_{1,z_j}^{R_j,\vec{y}_{*j}}(\m{E}_{q_{m_j},1,R_j}^j)dz_j.
		\end{aligned}
        \end{equation}
Here $\Pr_{x_1,\mathbf{u}_1}^{x_2,\mathbf{u}_2}$ denote the law of a planar Brownian bridge on $[x_1,x_2]$ from $\mathbf{u}_1$ to $\mathbf{u}_2$. A simple Gaussian moment computation shows that 
\begin{align*}
    \mathbb{E}_{\vec{x}\to (z_1,z_2)}^{1,1}[F_{i_1}^1 \wedge F_{i_2}^2] & \le \mathbb{E}_{\vec{x}\to (z_1,z_2)}^{1,1}[F_{i_j}^j] \\ & = e^{-q_{m_j-1}} \cdot \mathbb{E}_{\vec{x}\to (z_1,z_2)}^{1,1}\left[\int_0^1 e^{2\sqrt2 |V^j(x)-\vec{x}_{*j}|}dx\right] \\ & \le C e^{-q_{m_j-1}+C|z_j-\vec{x}_{*j}|} \\ & \le Ce^{-q_{m_j-1}} \cdot e^{C|z_1-\vec{x}_{*1}|+C|z_2-\vec{x}_{*2}|}
\end{align*}
holds for $j=1,2$. Thus,
\begin{align*}
    \mathbb{E}_{\vec{x}\to (z_1,z_2)}^{1,1}[F_{i_1}^1 \wedge F_{i_2}^2] \le Ce^{-q_{m_1-1}\vee q_{m_2-1}} \cdot e^{C|z_1-\vec{x}_{*1}|+C|z_2-\vec{x}_{*2}|}.
\end{align*}
Plugging this inequality back in   \eqref{eqn2wopower} we get
\begin{align} \nonumber
 & Q_{{m_1},{m_2},i_1,i_2}({x}_*\to {y}_*;\mathbf{R}) \\ & \nonumber\le C \cdot e^{-q_{m_1-1}\vee q_{m_2-1}} \\ & \hspace{1cm}\cdot\prod_{j=1}^2\int_{N_{q_{m_j}}}   p_{\frac{R_j-1}{R_j}}\left(z_j-\frac{\vec{y}_{*j}+\vec{x}_{*j}(R_j-1)}{R_j}\right)\cdot\mathbb{P}_{1,z_j}^{R_j,\vec{y}_{*j}}(E_{q_{m_j},1,R_j}^j) \cdot e^{C|z_j-\vec{x}_{*j}|} dz_j. \label{qbg}
\end{align}
Using $\mathbb{P}_{1,z_j}^{R_j,\vec{y}_{*j}}(E_{q_{m_j},1,R_j}^j)\le 1$, and a change of variable leads to
\begin{align*}
    & \int_{N_{q_{m_j}}}   p_{\frac{R_j-1}{R_j}}\left(z_j-\frac{\vec{y}_{*j}+\vec{x}_{*j}(R_j-1)}{R_j}\right)\cdot\mathbb{P}_{1,z_j}^{R_j,\vec{y}_{*j}}(E_{q_{m_j},1,R_j}^j) \cdot e^{C|z_j-\vec{x}_{*j}|} dz_j \\ & \le e^{C|\vec{x}_{*j}|/R_j+C|\vec{y}_{*j}|/R_j}\int_{\R^2}   p_{\frac{R_j-1}{R_j}}\left(z_j\right) e^{C|z_j|} dz_j \le Ce^{C|\vec{x}_{*j}|/R_j+C|\vec{y}_{*j}|/R_j}.
\end{align*}
This verifies \eqref{bdf4}. On the other hand, to show the bound in \eqref{bdf2}, 
using 
 Lemma \ref{l4.11equiv} we notice that
        \begin{align*}
           & \int_{N_{q}}   p_{\frac{R-1}{R}}\left(z_j-\frac{\vec{y}+\vec{x}(R-1)}{R}\right)\cdot\mathbb{P}_{1,z}^{R,\vec{y}}(E_{q,1,R}) \cdot e^{C|z-\vec{x}|} dz \\ & \le CR^{-3/2}(|\vec{y}|+q)^{3/2}\int_{N_{q}}   p_{\frac{R-1}{R}}\left(z-\frac{\vec{y}+\vec{x}(R-1)}{R}\right)|z-qh|^{\frac12} (q-\omega(z))e^{CR^{-1}|z-qh||\vec{y}-qh|+C|z-\vec{x}|} dz_j \\ &  \le CR^{-3/2}(|\vec{y}|+q)^{3/2} \cdot [\sqrt{|\vec{x}|}+\sqrt{q}]e^{C\frac{(q+|\vec{y}|)(q+|\vec{x}|)}{R}}
        \end{align*}
	where the last bound follows from Gaussian computations. Inserting the above bound back in \eqref{qbg}, we arrive at \eqref{bdf2}.
		This completes the proof.
	\end{proof}
	
	\subsubsection{Proof of Lemma \ref{vtau}} \label{sec:5.5}  To lighten the notation, we work with a planar Brownian Bridge $V$ on $[0,L]$ from $(0,0)$ to $\vec{w}:=(b\sqrt{2L/3},0)$. Fix any $q\in (0,L^{1/50}]$ and recall the wedge $N_q$. Let $\tau(q)$ be the first time when $\omega(V(x))=q$.  Fix any $\delta,\beta>0$. We claim that there exists an absolute constant $C>0$ such that for all $\xi>0$ with $\xi \le L^{1/4-\delta/2}$ we have 
        \begin{align}\label{ew0}
            \Pr\Big(|V(\tau(q/2))| \ge \xi^2+b, \tau(q/2) \le L^{1-\beta}\Big) \le Ce^{-\xi^{2-4\delta}/C}+ C\frac{(q+b)^{3/2}}{\xi^{1+2\delta}}.
        \end{align}
    Note that Lemma \ref{vtau} is a direct consequence of the above claim (taking $V=V^j$ and $q=2q_m$ and $L=L^j$). To prove the claim, we work with two new stopping times.  Set $R(x):=\frac{x}{L}\vec{w}(1-x/L)^{-1}$. Let $\varsigma(q)$ and $\varsigma'(q)$ be the first times when
    \begin{align*} 
		\frac{V(x)}{1-\frac{x}{L}} \in \partial N_{q}, \qquad \frac{V(x)}{1-\frac{x}L}-R(x) \in \partial N_{q+3b} 
	\end{align*} 
respectively. Note that 
\begin{itemize}[leftmargin=20pt]
    \item on the event $\tau(q/2) \le L^{1-\beta}$, we have $\tau(q/2) \le \varsigma(q)$ for large enough $L$, and
    \item if $|R(\varsigma')| \le b$, then $V(\varsigma')/(1-\varsigma'/L)$ lies outside $N_{q}$. Thus, $\{\varsigma(q) \le \varsigma'\} \supset \{R(\varsigma')\le b\}$.
\end{itemize}
Thus we deduce
\begin{align}\nonumber
    \Pr\Big(|V(\tau(q/2))| \ge \xi^2+b, \tau(q/2) \le L^{1-\beta}\Big) & \le \Pr \bigg( \sup_{x\in [0, \varsigma(q)]} |V(x)| \ge \xi^2 +b\bigg) \\ & \le \Pr\bigg(\sup_{x\le \varsigma'(q)} |V(x)| \ge \xi^2+b\bigg)+\Pr\big(|R(\varsigma'(q))| \ge b\big). \label{ew1}
\end{align}
Hence it suffices to bound the right hand side of the above equation.   For simplicity we write $\varsigma'=\varsigma'(q)$. The stopping time $\varsigma'$ is much easier to analyze by a series of transformation which we now explain. Note that
		$\hat{V}(x):=V(x)-\frac{x}{L}\vec{w}$ is a standard planar Brownian bridge from $0$ to $0$ and
		\begin{align*}
			Y(x):=\frac{L+x}{L}\hat{V}\bigg(\frac{Lx}{L+x}\bigg)
		\end{align*}
		is a planar Brownian motion. Let $\kappa$ be the first time $Y(\cdot) \in \partial N_{q+3b}$. Then $\varsigma'=\frac{L\kappa}{L+\kappa}$. Define $Z(x):=(Y(x)-(q+3b)h)^{3/2}$. where we raise an element of $\R^2$ to the 3/2 power by thinking of it as an element of $\mathbb{C}$. The map $z \to z^{3/2}$ is a conformal isomorphism between the interior of $N_0$ and
the interior of the right half plane $\{z : \mathrm{Re}(z) \ge 0\}$. By conformal invariance of planar Brownian motion, $Z(\cdot)$ is a time changed planar Brownian motion starting at $Z(0)=((q+3b)|h|)^{3/2}$. $\kappa$ is the first time $Z(\cdot)$ hits $i\mathbb{R}$. So, by Brownian motion calculations
		\begin{align}\label{taup}
			\Pr(\varsigma' \ge Lr/(L+r))=\Pr(\kappa \ge r) \le C((q+3b)|h|)^{3/2}/\sqrt{r}
		\end{align}
		for all $r>0$.
		As $\{\varsigma' \ge Lr/(L+r)\}=\{R(\varsigma')\ge rb\sqrt{2/3L}\}$, we have
		\begin{align*}
			\Pr(|R(\varsigma')| \ge  rb\sqrt{2/3L})  \le C((q+3b)|h|)^{3/2}/\sqrt{r} 
		\end{align*}
		for all $r>0$. Let us change $r$ to $\xi^{2+4\delta}$ in the above equation and note that $\xi^{2+4\delta} \le \sqrt{L}$ (as $\xi \le L^{\frac14-\delta/2})$. Then $\xi^{2+4\delta}b\sqrt{2/3L} \le b$ and thus
		\begin{align} \label{refe}
			\Pr(|R(\varsigma')| \ge b)  \le C((q+3b)|h|)^{3/2}/\xi^{1+2\delta}. 
		\end{align}
This bounds the second term in \eqref{ew1}.	We now bound the first term in \eqref{ew1}. Towards this end, fix any $g,r>0$. Observe that
		\begin{align*}
		    \Pr\bigg(\sup_{x\le \varsigma'} |V(x)| \ge g\sqrt{\frac{Lr}{L+r}}+\frac{br\sqrt{L}}{L+r}\bigg) & \le \Pr\bigg(\sup_{x\le \frac{Lr}{L+r}} |V(x)| \ge g\sqrt{\frac{Lr}{L+r}}+\frac{br\sqrt{L}}{L+r}\bigg)+\Pr\bigg(\varsigma' \ge \frac{Lr}{L+r}\bigg) \\ & \le \Pr\bigg(\sup_{x\le \frac{Lr}{L+r}} |\hat{V}(x)| \ge g\sqrt{\frac{Lr}{L+r}}\bigg)+\Pr\bigg(\varsigma' \ge \frac{Lr}{L+r}\bigg) \\ & \le Ce^{-g^2/C}+ C((q+3b)|h|)^{3/2}/\sqrt{r}	\end{align*}
		for some absolute constant $C>0$. The last inequality follows from \eqref{taup} and Gaussian tail bounds. We change $r$ to $\xi^{2+4\delta}$ and $g$ to $\xi^{1-2\delta}$ in the above equation and use that $\xi^{2+4\delta} \le \sqrt{L}$ to get
		$$\Pr\bigg(\sup_{x\le \varsigma'} |V(x)| \ge \xi^{2}+b\bigg) \le Ce^{-\xi^{2-4\delta}/C}+ C((q+3b)|h|)^{3/2}/\xi^{1+2\delta}.$$
		Plugging this bound and the bound in \eqref{refe} back in \eqref{ew1}, we arrive at \eqref{ew0}. This completes the proof.

\appendix

\section{Moment bounds of Open SHE}\label{appC}
In this section, we collect moment bounds for the solution of the open SHE \eqref{opshe} with Robin boundary conditions \eqref{robin}. Most of the results in this appendix are either proved in the literature or follow by straightforward adaptations of standard arguments.

Set $A=(u-\frac12)$ and $B=-(v-\frac12)$ for simplicity. 
We first discuss properties of the solution of the heat equation with Robin boundary conditions:
\begin{equation}\label{prob}
\begin{aligned}
    & \frac{\p v}{\p t} = \frac{\p^2v}{\p z^2}, \quad z \in [0,L],\; t>0 \\
& \left.\frac{\p v}{\p z}\right\vert_{z=0} = A v(t,0), \; \left.\frac{\p v}{\p z}\right\vert_{z=L} = B v(t,L).
\end{aligned}
\end{equation}

The following lemma provides the description of the fundamental solution to the above equation.

\begin{lemma}\label{robinkernel} For each $n\ge 0$, let $\kappa_n \in ( (n-\frac{1}{2})\pi,(n+\frac{1}{2})\pi )$ be the unique solution to
\begin{equation*}
\tan(L\kappa_n) = \frac{\kappa_n (A-B)}{\kappa_n^2+AB}
\end{equation*}
in the interval $( (n-\frac{1}{2})\pi,(n+\frac{1}{2})\pi )$.
Set  $\psi_n(z)=a_n\cos\left(\kappa_n z \right) + a_n\frac{A}{\kappa_n}\sin\left(\kappa_nz\right)$, where 
    $$a_n^{-2}:=\int_0^L\bigg( \cos(\kappa_n x)+\frac{A}{\kappa_n}\sin (\kappa_n x)\bigg)^2dx.$$
The kernel \begin{equation}\label{Robin}
K_t^{A,B}(x,y)=\sum_{n=0}^{\infty}e^{-t\kappa_n^2}\psi_n(x)\psi_n(y),
\end{equation}
solves \eqref{prob}  with Dirac delta initial data $\delta_0$. 
\end{lemma}
\begin{remark} \label{rem1}
    It is a straightforward calculation to verify that $\lim_n a_n^{-2}=\frac{L}{2}$, therefore $(a_n)_{n\in\mathbb{N}}$ is a bounded sequence.
\end{remark}
\begin{proof}
We first note that there exists a unique weak solution to \eqref{prob}. Consider the space $\mathcal{A}=\{u\in H^2([0,L]): u_x(0)=Au(0),\; u_x(L)=Bu(L)\}$, where the equalities are taken in the sense of traces. We observe that the second derivative operator (i.e the Laplacian) $u\mapsto u_{xx}$ is selfadjoint in $\mathcal{A}$. Indeed, after integration by parts we have 
\begin{align*}
\int_0^L u_{xx}(x)v(x)dx-\int_0^Lu(x)v_{xx}(x)dx&=\bigg[ u_x(x)v(x)  \bigg]_0^L-\bigg[ u(x)v_x(x) \bigg]_0^L\\
&=Bu(L)v(L)-Au(0)v(0)- (Bu(L)v(L)-Au(0)v(0))=0.
\end{align*}
Since the Laplacian is selfadjoint, we deduce that $\Delta^{-1}$ is also selfadjoint and, by standard theory of elliptic equations, it is also compact. Therefore the spectral theorem can be applied and hence we can consider $(\lambda_n)_{n\in\mathbb{N}}$ the eigenvalues of the problem
\begin{equation}\label{eigenvalue}
\begin{cases}
    (\psi(x))_{xx}=\lambda \psi(x),\;\;x\in [0,L],\\
    \psi_x(0)=A\psi(0),\; \psi_x(L)=B\psi(L).
\end{cases}
\end{equation}
We call $\psi_n(x)$ the eigenvector that corresponds to the eigenvalue $\lambda_n$. Note that $(\psi_n)_{n\in\mathbb{N}}$ can be chosen to be an orthonormal basis of $L^2([0,L])$. Let $K$ be as in \eqref{Robin}. It is, therefore, straightforward to see that $v(t,x):=\int_0^LK_t^{A,B}(x,z)v_0(z)dz$ satisfies 
$$v(0,x)=\sum_{n=0}^{+\infty}\psi_n(x)\int_0^L\psi_n(z)v_0(z)dz=v_0(x)\;\;\;\text{in }L^2,$$
because $(\psi_n)_{n\in\mathbb{N}}$ is an orthonormal basis of $L^2([0,L])$, and that 
$$\p_xv(t,x)=\int_0^L\p_xK_t^{A,B}(x,z)v_0(z)dz=\begin{cases}
    Av(t,0), &\text{ if }x=0,\\
    Bv(t,L),&\text{ if }x=L.
\end{cases}$$
In addition, for $t>0$:
$$\p_tv(t,x)=\int_0^L \p_tK_t^{A,B}(z,x)v_0(x)dx=\int_0^L \p_{xx}K_t^{A,B}(x,z)v_0(z)dz=\p_{xx}v(t,x),$$
thus $v$ is a solution of \eqref{prob}. The formulas for $\psi_n(x)$ and $\lambda_n$ follow by solving \eqref{eigenvalue} explicitly and by choosing $a_n$ such that $(\psi_n)_{n\in\mathbb{N}}$ is an orthonormal basis.
\end{proof}
We next record regularity estimates for $K_t^{A,B}$ from \cite{parekh2019kpz}.

\begin{lemma}[Proposition 3.31 in \cite{parekh2019kpz}] \label{lem:reg} Fix a terminal time $T>0$. For any $b>0$, there exists  some constant $C=C(A,B,L,b,\tau)>0$ such that for all $0<s<t<T$ and $x,y,z\in [0,L]$ we have that
\begin{align} \nonumber
    & K_t^{A,B}(x,y) \le Ct^{-1/2}\exp(-b|x-y|/\sqrt{t}), \\
  \label{eq:hkreg}  & |K_t^{A,B}(x,z)-K_t^{A,B}(y,z)| \le Ct^{-1}|x-y| \\ \nonumber
    & |K_t^{A,B}({x,y})-K_s^{A,B}({x,y})| \le Cs^{-3/2}|t-s|.
\end{align}    
\end{lemma}
We now turn towards the moment bounds for the open SHE. We recall the notation introduced in the beginning of Section \ref{sec:3}.

\begin{proposition}[Theorem 2.1 in \cite{parekh2022ergodicity}] \label{p:mombd}
Fix any $T>0$ and $p \ge 1$ There exists a constant $C(T,L,A,B,p)>0$ such that for all $x,y\in [0,L]$ and $t\in (0,T]$, we have
\begin{equation*}
\mathbb{E}[\mathcal{Z}(y,0;x,t)^p]\leq 
    {Ct^{-p/4}K_t^{A,B}(x,y)^{p/2}} \le Ct^{-p/2}e^{-2p|x-y|/\sqrt{t}}.
\end{equation*}
\end{proposition}

 The following lemma states that when the initial data is constant, the negative moments of the open SHE are uniformly bounded. Its proof can be adapted from the arguments in \cite{hu2018asymptotics}, and we omit the details here.

\begin{proposition}\label{p.unbd} Fix any $p \ge 1$. There exists constant $C(T,L,A,B,p)>0$ such that for all $t\le T$ we have $\Ex\left[\mathcal{Z}(\mathbf{1},0;x,t)^{-p}\right] \le C.$
\end{proposition}

If the time parameter is bounded away from zero, we have the following moments of the spatial supremum of the SHE are bounded uniformly.

\begin{proposition}\label{p:negmombd}
Fix any $T, p \ge 1$. There exists a constant $C(T,L,A,B,p) >0$ for all $t\in [T^{-1},T]$ we have
\begin{equation}\label{eq1}
\mathbb{E}\left[\sup_{x,y\in [0,L]}\mathcal{Z}(y,0;x,t)^{-p}\right]+\mathbb{E}\left[\sup_{x,y\in [0,L]}\mathcal{Z}(y,0;x,t)^{p}\right] \leq C.
\end{equation}
\end{proposition}

\begin{proof} Let us write $\mathcal{Z}(x,t)=\mathcal{Z}(0,0;x,t)$ for simplicity. We claim that there exists a constant $C' = C'(T,L,A,B,p)>0$ such that
\begin{equation}\label{momdiff}
\sup_{t \in [T^{-1},T],\;x\neq y \in [0,L]} \frac{\mathbb{E} \left[\vert \mathcal{Z}(x,t)-\mathcal{Z}(y,t)\vert^p\right]}{\vert x-y \vert^{p/2}} \leq C'.
\end{equation}
Given the pointwise moment bounds and regularity bounds, a standard chaining argument leads to the uniform bound in \eqref{eq1}. We thus focus on proving \eqref{momdiff}. Towards this end, we make use of the formula
\begin{equation*}
\mathcal{Z}(x,t)= K_t(x,0)+\int_0^t\int_0^LK_{t-s}(x,y)\mathcal{Z}(y,s)\xi(s,y)dyds,
\end{equation*}
Using the inequality $|a+b|^p \le C |a|^p+|b|^p$ we have 
\begin{align}\nonumber
    \mathbb{E}\bigg[|\mathcal{Z}(x,t)-\mathcal{Z}(y,t)|^p\bigg] & \leq C \vert K_t(x,0) - K_t(y,0) \vert^p \\ &   \hspace{-1cm}+ C\Ex \bigg[\bigg|\int_0^t\int_0^L (K_{t-s}(x,w)-K_{t-s}(y,w))\mathcal{Z}(w,s)\xi(s,w)dwds\bigg|^p\bigg]. \label{2term}
\end{align}
The first term can be controlled by regularity estimate of the heat kernel from Lemma \ref{lem:reg} (\eqref{eq:hkreg} specifically). For the second term, by Burkholder-Davis-Gundy inequality we have
\begin{align}
  \eqref{2term}  & \le   C\bigg(\int_0^t \int_0^L (K_{t-s}(x,w)-K_{t-s}(y,w))^2 \mathbb{E}\left[|\mathcal{Z}(w,s)|^2\right] dwds\bigg)^{\frac{p}2}. \label{3term}
\end{align}
 It is thus suffices to bound the term in \eqref{3term}. Towards this end, we split the $s$ integral into two parts: $[0,T^{-1}/2]$ and $[T^{-1}/2,T]$.
In the range $s\in [0,T^{-1}/2]$, we use the bounds  from Proposition \ref{p:mombd} and the bounds from Lemma \ref{lem:reg} to get
\begin{align*}
   \mathbb{E}\left[|\mathcal{Z}(w,s)|^2\right]  \le Ct^{-1}e^{-b|w|/\sqrt{t}}.
\end{align*}
 Using this bound we get
\begin{align*}
     & \int_0^{T^{-1}/2} \int_0^L (K_{t-s}(x,w)-K_{t-s}(y,w))^2 \mathbb{E}\left[|\mathcal{Z}(w,s)|^2\right] dwds \\ & \le \int_0^{T^{-1}/2} \int_0^L (K_{t-s}(x,w)-K_{t-s}(y,w))^2 s^{-1/2}K_s^{A,B}(w,0) dwds \\ & \le \int_0^{T^{-1}/2} \int_0^L C(t-s)^{-2}(x-y)^2 s^{-1}e^{-b|w|/\sqrt{s}} dwds \\ & \le C(x-y)^2\int_0^{T^{-1}/2}\int_0^k s^{-1}e^{-b|w|/\sqrt{s}} dwds \le C(x-y)^2.
\end{align*}
In the range $s\in [T^{-1}/2,T]$ we note that $ \mathbb{E}\left[|\mathcal{Z}(w,s)|^2\right]  \le C$ for smooth and white noise and hence
\begin{align*}
   & \int_{T^{-1}/2}^t \int_0^L (K_{t-s}(x,w)-K_{t-s}(y,w))^2 \mathbb{E}\left[|\mathcal{Z}(w,s)|^2\right] dwds \\ & \hspace{5cm} \le  \int_{T^{-1}/2}^t \int_0^L (K_{t-s}(x,w)-K_{t-s}(y,w))^2 dw ds.
\end{align*}
We claim that
\begin{align}
    \label{eq:toshow}
    \int_0^t \int_0^L (K_{t-s}(x,z)-K_{t-s}(y,z))^2 dzds \le C \vert x-y\vert.
\end{align}
Plugging this bound above we get the desired estimate. We thus focus on proving \eqref{eq:toshow}. To this end we recall the expression \eqref{Robin} for the heat kernel.
We have by Parseval identity that
\[
\int_0^L \left(K_{t-s}(x,z) - K_{t-s}(y,z)\right)^2 dz = \sum_{n=1}^\infty (\psi_n(x)-\psi_n(y))^2 e^{2\lambda_n(t-s)}.
\]
From the explicit expressions of $\psi_n$ we obtain
\begin{align*}
\int_0^L (K_{t-s}(x,z)-K_{t-s}(y,z))^2 dz & = \sum_{n=1}^\infty a_n^2 \left(1+\frac{A}{\kappa_n}\right)^2 \min\bigg\{\kappa_n^2 \vert x-y\vert^2,4\bigg\} e^{2 \lambda_n (t-s)}\\
& \le C \sum_{n=1}^\infty \min\bigg\{\kappa_n^2 \vert x-y\vert^2,4\bigg\} e^{2 \lambda_n (t-s)}.
\end{align*}
Thus
\begin{align*}
\int_0^t \int_0^L \left(K_{t-s}(x,z) - K_{t-s}(y,z)\right)^2 dz ds &\le C \sum_{n=1}^\infty \min\bigg\{\kappa_n^2 \vert x-y\vert^2,4\bigg\} \int_0^t e^{2 \lambda_n (t-s)} ds\nonumber\\
&\le C  \sum_{n=1}^\infty \min\bigg\{\kappa_n^2 \vert x-y\vert^2,4\bigg\} \frac{1}{\vert \lambda_n \vert}\nonumber\\
&= C\sum_{n=1}^\infty \min \bigg\{ \vert x-y\vert^2,\frac{4}{\kappa_n^2}\bigg\}.
\end{align*}
Splitting this sum we obtain
\begin{align*}
\sum_{n=1}^\infty \min \bigg\{ \vert x-y\vert^2,\frac{4}{\kappa_n^2}\bigg\} &= \bigg(\sum_{n : \kappa_n > \frac{2}{\vert x-y\vert}} \frac{4}{\kappa_n^2}\bigg ) + \vert x-y\vert^2 \#\bigg\{n\in \mathbb{N}^* : \kappa_n \leq \frac{2}{\vert x-y\vert}\bigg\}\\
&\le C \sum_{n : \pi n - \frac{\pi}{2} > \frac{2}{\vert x-y\vert}} \frac{4}{\pi^2}\frac{1}{\left(n-1\right)^2} + \vert x-y\vert ^2 \# \bigg\{ n\in \mathbb{N}^*: \pi(n-1) \leq \frac{2}{\vert x-y\vert}\bigg\}\\
& \le C \vert x-y\vert.
\end{align*}
This proves \eqref{eq:toshow}.
\end{proof}

    \bibliographystyle{alpha}		
	\bibliography{openkpz}

\end{document}